\documentclass[12pt]{amsart}

\usepackage{amsmath}
\usepackage{mathtools}
\usepackage{amssymb}
\usepackage{amsthm}
\usepackage{amscd}
\usepackage{mathrsfs}

\usepackage{graphicx}
\usepackage[dvipsnames]{xcolor}
\usepackage[normalem]{ulem}
\usepackage{comment}

\usepackage{multirow}
\usepackage{bigdelim}
\usepackage{enumerate}
\usepackage{enumitem}
\usepackage[all]{xy}
\usepackage{tikz}

\usepackage[most]{tcolorbox}
\tcbuselibrary{breakable, skins, theorems}
\tcbset{
  mybox/.style={
    colback=white,
    colframe=green!35!black,
    fonttitle=\bfseries,
    breakable=true
  }
}

\usepackage[final]{showkeys}

\setlist[enumerate]{label=$(\arabic*)$}

\usepackage[colorlinks,linkcolor=red,anchorcolor=blue,citecolor=blue]{hyperref}
\newtheorem{thm}{Theorem}[section]

\newtheorem{cor}[thm]{Corollary}
\newtheorem{prop}[thm]{Proposition}

\newtheorem{lem}[thm]{Lemma}

\theoremstyle{definition}
\newtheorem{defn}[thm]{Definition}

\newtheorem{thmdefn}[thm]{Theorem-Definition}

\newtheorem{ex}[thm]{Example}

\theoremstyle{remark}
\newtheorem{rem}[thm]{Remark}
\newtheorem{setup}[thm]{Setup}

\newtheorem{notation}[thm]{Notation}
\newtheorem{claim}[thm]{Claim}
\newtheorem{assup}[thm]{Assumption}

\numberwithin{equation}{section}

\newcommand{\rk}{\operatorname{rk}}

\newcommand{\codim}{\operatorname{codim}}

\newcommand{\Supp}{\operatorname{Supp}}

\newcommand{\End}{\operatorname{End}}

\newcommand{\Hom}{\mathscr{H}\!\mathit{om}}

\DeclareMathOperator{\Ric}{Ric}

\DeclareMathOperator{\tr}{tr}
\DeclareMathOperator{\Id}{Id}
\DeclareMathOperator{\res}{res}
\DeclareMathOperator{\length}{length}

\DeclareMathOperator{\Diff}{Diff}
\DeclareMathOperator{\Lie}{Lie}
\DeclareMathOperator{\Autop}{Aut}
\DeclareMathOperator{\Kerop}{Ker}

\DeclareMathOperator{\grad}{grad}

\newcommand{\cF}{\mathcal{F}}

\DeclareMathOperator{\Sym}{Sym}

\newcommand{\Tor}{\mathrm{tor}}

\newcommand{\Aut}[1]{\Autop(#1)}
\newcommand{\Ker}[1]{\Kerop(#1)}

\newcommand{\R}{\mathbb{R}}
\newcommand{\Z}{\mathbb{Z}}
\newcommand{\N}{\mathbb{N}}
\newcommand{\C}{\mathbb{C}}
\newcommand{\Q}{\mathbb{Q}}

\newcommand{\T}{\mathbb{T}}
\newcommand{\G}{\mathbb{G}}

\DeclareMathOperator{\ch}{ch}
\DeclareMathOperator{\td}{td}
\DeclareMathOperator{\Td}{Td}
\DeclareMathOperator{\CH}{CH}

\DeclareMathOperator{\Cone}{Cone}
\newcommand{\dR}{\mathrm{dR}}

\newcommand{\underalign}[2]{\quad \underset{\mathclap{\strut #1}}{#2}\quad}

\title[Miyaoka--Yau inequality and delta invariant II -Soliton case-]{The Miyaoka--Yau inequality and the delta invariant for Fano varieties II -Soliton case-}

\author{Tomoyuki HISAMOTO}
\address{Graduate School of Mathematics, Nagoya University, Furocho, Chikusa, Nagoya,
Japan}
\email{{\tt  hisamoto@math.nagoya-u.ac.jp}}

\author{Masataka IWAI}
\address{Department of Mathematics, Graduate School of Science, the University of Osaka,
1-1, Machikaneyama-cho, Toyonaka, Osaka 560-0043, Japan.}
\email{{\tt masataka@math.sci.osaka-u.ac.jp}}
\email{{\tt masataka.math@gmail.com}}

\date{\today, version 0.01}

\begin{document}

\begin{abstract}
%We establish the following Miyaoka--Yau inequality for any \(n\)-dimensional klt Fano variety \(X\), possibly K-unstable, in terms of its delta invariant:
%\[\left(2(n+1)\widehat{c}_2(X)-n c_1(X)^2\right)\cdot c_1(X)^{n-2}\ge -n \left(1-\min\{1,\delta(X)\}\right)^2 \cdot c_1(X)^n.\]
Inspired by recent work of Inoue and Hallam--Lahdili, 
we prove an equivariant Miyaoka--Yau inequality with a correction term for smooth Fano manifolds admitting K\"ahler--Ricci solitons. It recovers the usual Miyaoka--Yau inequality in the K\"ahler--Einstein case. Our proof uses weighted equivariant Bogomolov--Gieseker and Langer inequalities for torsion-free sheaves, together with a weighted slope estimate for the canonical extension sheaf. We also construct toric examples showing that the existence of a soliton implies neither weighted semistability of the canonical extension sheaf nor nonnegativity of the equivariant Miyaoka--Yau number.

\end{abstract}

\maketitle

%%%%%%%%%%%%%%%
% Show entries down to subsubsections in the table of contents.
\setcounter{tocdepth}{2}
%%%%%%%%%%%%%%%%

%%%%%%%%%%%%%%%%%%
% Customize the indentation of section, subsection, and subsubsection entries.
\makeatletter
% Section entries: no extra indentation.
\renewcommand{\tocsection}[3]{%
  \indentlabel{\@ifnotempty{#2}{\ignorespaces#1 #2.\quad}}#3%
}
% Subsection entries: add a small indentation.
\renewcommand{\tocsubsection}[3]{%
  \hspace{0.75em}\indentlabel{\@ifnotempty{#2}{\ignorespaces#1 #2.\quad}}#3%
}
% Subsubsection entries: add a larger indentation.
\renewcommand{\tocsubsubsection}[3]{%
  \hspace{1.5em}\indentlabel{\@ifnotempty{#2}{\ignorespaces#1 #2.\quad}}#3%
}
\makeatother
%%%%%%%%%%%%%%%%%%%%%%%%%%%

\tableofcontents

\section{Introduction}
\label{sec-Introduction}
Using his solution of the Calabi conjecture in \cite{Yau78}, Yau \cite{Yau77} proved that every \(n\)-dimensional compact K\"ahler manifold with ample canonical bundle satisfies
\begin{equation}
\label{equa-intro-yau}
    \big(2(n+1)c_2(X)-n c_1(X)^2\big) \cdot c_1(K_X)^{n-2} \geq 0.
\end{equation}
Independently, Miyaoka \cite{Miy77} proved that $3c_2(X)-c_1(X)^2 \geq 0$ for every smooth projective surface $X$ with big canonical bundle. 
The inequality \eqref{equa-intro-yau} is now called the \emph{Miyaoka--Yau inequality}. 
When the canonical divisor is positive, extensions to klt varieties and log pairs have been studied in \cite{GKPT19b, GT22, CGG23, ZZZ25, IJZ25}.
 In particular, \cite{IJZ25} established the Miyaoka--Yau inequality for klt varieties of arbitrary dimension with big canonical divisor by using non-pluripolar products.

Our main result is an equivariant Miyaoka--Yau inequality with a correction term for K\"ahler--Ricci solitons.
%Ricci solitons play an important role in the study of the Ricci flow introduced by Hamilton, as illustrated by Perelman's celebrated work. 
%K\"ahler--Ricci solitons are the K\"ahler analogues of Ricci solitons. 
K\"ahler--Ricci solitons are K\"ahler analogues of Ricci solitons and arise naturally as limits of the normalized K\"ahler--Ricci flow on Fano manifolds.
They were first studied in \cite{Cao97} and further developed in \cite{TZ00, TZ02, WZ04}. The solution of the Hamilton--Tian conjecture (see \cite{Tian97} for its formulation)\cite{CW12, CW17a, CW17b} shows that the normalized K\"ahler--Ricci flow on a Fano manifold converges along subsequences to a possibly singular \(\mathbb{Q}\)-Fano variety carrying a weak K\"ahler--Ricci soliton. It was further proved in \cite{CSW18} that this limit is unique and that the degeneration admits an algebro-geometric description. A purely algebro-geometric construction was recently obtained in \cite{BLXZ23}. These results suggest that K\"ahler--Ricci solitons play an important role in the study of general Fano manifolds.
%Recently, the Hamilton--Tian conjecture was solved in \cite{CW12, CW17a, CW17b}. and the work \cite{CSW18} showed that the normalized K\"ahler--Ricci flow on a Fano manifold yields a canonical algebraic  degeneration whose limit admits a K\"ahler--Ricci soliton. 

However, in the soliton setting, one cannot expect the Miyaoka--Yau inequality to hold for the usual Chern classes.  Indeed, every toric Fano manifold admits a K\"ahler--Ricci soliton by \cite{WZ04}, whereas the toric Fano $4$-fold $\mathbb{P}_{\C\mathbb{P}^3}(\mathcal{O}_{\C\mathbb{P}^3} \oplus \mathcal{O}_{\C\mathbb{P}^3}(3))$ does not satisfy the Miyaoka--Yau inequality. 
Thus, the usual Chern classes must be replaced by suitable alternatives.

To overcome this difficulty, we use equivariant intersection theory inspired by the recent works of Inoue \cite{Inoue20} and Hallam--Lahdili \cite{HL24}. Following Inoue's approach, which combines algebraic geometry and complex geometry, and extending the work of Hallam--Lahdili to coherent sheaves within an algebro-geometric framework, we establish the following equivariant 
Miyaoka--Yau inequality with a correction term.

\begin{thm}[{cf.~Corollary \ref{cor-equiv-MY-corrected-delta-one}}]
\label{thm-main-equiv}
Let \(X\) be an \(n\)-dimensional smooth Fano manifold with \(n\geq2\) admitting a K\"ahler--Ricci soliton \((\omega,\xi)\), that is, there exist a compact torus \(\T\subset\operatorname{Aut}^{\circ}(X)\), an element \(\xi\in\mathfrak t\coloneqq\operatorname{Lie}(\T)\), a \(\T\)-invariant K\"ahler form \(\omega\in c_1(X)\), and a canonically normalized moment map \(\mu\colon X\to\mathfrak t^\vee\) such that
 \[\operatorname{Ric}(\omega)-\omega=\frac{\sqrt{-1}}{2\pi}\partial\bar\partial\langle\mu,\xi\rangle. \]
Set $\mu_\xi\coloneqq\langle\mu,\xi\rangle$.
Then the following equivariant Miyaoka--Yau inequality with a correction term holds:
\begin{equation*}
\begin{aligned}
&\left(\left(2(n+1)c_2^\T(X)-n c_1^\T(X)^2\right)\cdot e^{c_1^\T(X)}\right)_\T(\xi) \\
&\geq 
-\frac{n+1}{2(e^{c_1^\T(X)})_\T(\xi)} 
\left(\int_X\left|
\frac{(e^{c_1^\T(X)})_\T(\xi)}{\int_X\omega^{[n]}}
e^{-\mu_\xi}-1-\mu_\xi
\right|e^{\mu_\xi}\omega^{[n]}\right)^2.
\end{aligned}
\end{equation*}
\end{thm}

In particular, by \cite{WZ04}, every toric Fano manifold satisfies this inequality. 
Moreover, when $X$ admits a K\"ahler--Einstein metric, $\xi=0$, so this inequality recovers the usual Miyaoka--Yau inequality proved by \cite{CO75}.
We refer to \cite{TZ00, HL23} for the equivalence between the above definition of a K\"ahler--Ricci soliton and the usual definition formulated in terms of the Lie derivative.
More generally, we obtain a formulation involving the greatest weighted Ricci lower bound $\delta'_v(X)$; see Theorem \ref{thm-equiv-MY-corrected-delta}.

To understand this theorem, we will explain the equivariant intersection number $(c^{\T}_2(X) \cdot e^{c^{\T}_1(X)})_{\T}(\xi)$ and the canonical normalization of \(v\).
First, let \(c_2^{\T}(X)\in H_{\T}^{4}(X,\mathbb R)\) denote the second equivariant Chern class of \(X\), and let
\(
\pi_*\colon H_{\T}^{\bullet}(X,\mathbb R)\rightarrow \operatorname{Sym}\mathfrak t^\vee
\)
be the equivariant pushforward. 
Following \cite{Inoue20, HL24}, by using the Taylor expansion \(e^x=\sum_{k=0}^{\infty}\frac{x^k}{k!}\), we define
\[
(c^{\T}_2(X) \cdot e^{c^{\T}_1(X)})_{\T}
\coloneqq \sum_{k=0}^{\infty}
\pi_{*}\left(c^{\T}_2(X) \cdot \frac{1}{k!}c^{\T}_1(X)^{k}\right)
\in \widehat{\Sym\,\mathfrak{t}^{\vee}}.
\]
By \cite[Theorem 1.1]{Inoue20}, this expression can be evaluated at any \(\xi\in\mathfrak t\), giving a well-defined real number
\(
(c^{\T}_2(X) \cdot e^{c^{\T}_1(X)})_{\T}(\xi).
\)
Similarly, we define the equivariant intersection numbers $(c^{\T}_1(X) \cdot e^{c^{\T}_1(X)})_{\T}(\xi)$ and $(e^{c^{\T}_1(X)})_{\T}(\xi)$.

Second, we say that \emph{$v$ is canonically normalized} if
\begin{equation}
\label{eq-defn-canonically-normalized}
(c^{\T}_1(X) \cdot e^{c^{\T}_1(X)})_{\T}(\xi)
=n(e^{c^{\T}_1(X)})_{\T}(\xi).
\end{equation}
This condition is equivalent to
$\int_{X} \mu_\xi e^{\mu_\xi}\omega^n=0$, 
where $\mu_{\xi} \coloneqq \langle\mu,\xi\rangle$, and hence agrees with the normalization in \cite{TZ02, AJL23}.
%Hence the inequality \eqref{eq-MY-solition-delta} is equivalent to
%$$\left( \left(2(n+1)c^{\T}_2(X) -n c^{\T}_1(X)^2\right)\cdot e^{c^{\T}_1(X)} \right)_{\T}(\xi)\ge  - n^3 (1 - \delta'_v(X))^2 \cdot \left( e^{c^{\T}_1(X)}\right)_{\T}(\xi).$$

One may ask whether the existence of a K\"ahler--Ricci soliton with soliton vector $\xi$ implies
\[
\left(\left(2(n+1)c_2^\T(X)-n c_1^\T(X)^2\right)\cdot e^{c_1^\T(X)}\right)_\T(\xi)\geq0.
\]
The following theorem gives a counterexample.

\begin{thm}[{=Example \ref{ex-negative-equiv-MY-sixfold}}]
\label{thm-negative-equiv-MY}
Let $X\coloneqq\mathbb P_{\C\mathbb P^5}(\mathcal O\oplus\mathcal O(5))$, and let $\T$ be the maximal compact torus of its acting algebraic torus.
Then $X$ admits a $\T$-invariant K\"ahler--Ricci soliton $(\omega,\xi)$ such that
\[
\left(\left(14c_2^\T(X)-6c_1^\T(X)^2\right)\cdot e^{c_1^\T(X)}\right)_\T(\xi)<0.
\]
\end{thm}

We now explain the strategy of the proof of Theorem \ref{thm-main-equiv}.
Our proofs are based on Tian's approach to the Miyaoka--Yau inequality \cite{Tian92}.
Tian constructed a Hermite--Einstein metric on the canonical extension sheaf of a Fano manifold admitting a K\"ahler--Einstein metric (see Definition \ref{defn-canonical-extension-equiv}), thereby proving its semistability.
Applying the Bogomolov--Gieseker inequality to this sheaf gives the Miyaoka--Yau inequality.
Without assuming semistability, Langer's inequality \cite{Lan04} gives a lower bound for the discriminant in terms of the maximal and minimal slopes.
Using the equivariant Bogomolov--Gieseker inequality and the equivariant Hodge index inequality, we prove a weighted equivariant version of Langer's inequality.
In particular, the following holds for exponential weights.

\begin{thm}[{= Corollary \ref{cor-equiv-BG-soliton-case} \& Theorem \ref{thm-langer-ineq-equiv}}]
\label{thm-main-equiv-langer}
Assume Setup \ref{setup-soliton-situation}, and let $v(\mu)=e^{\mu_\xi}$ for some $\xi\in\mathfrak t$.
Then, for any $\G$-equivariant torsion-free sheaf $\mathcal E$ of rank $r$, we have
\[
\left(\left(2rc_2^\T(\mathcal E)-(r-1)c_1^\T(\mathcal E)^2\right)\cdot v(\alpha_\T)\right)
\geq
-\frac{r^2}{(v(\alpha_\T))}
\left(\mu_v^{\max}(\mathcal E)-\mu_v(\mathcal E)\right)
\left(\mu_v(\mathcal E)-\mu_v^{\min}(\mathcal E)\right).
\]
Here, $\mu_v^{\max}(\mathcal E)$ and $\mu_v^{\min}(\mathcal E)$ denote the maximal and minimal weighted slopes of $\mathcal E$, respectively.

In particular, if $\mathcal E$ is $v$-semistable, then $\left(\left(2rc_2^\T(\mathcal E)-(r-1)c_1^\T(\mathcal E)^2\right)\cdot v(\alpha_\T)\right) \ge 0$.
\end{thm}

We apply this approach to exponential weights in the proof of Theorem \ref{thm-main-equiv}.
However, the canonical extension sheaf need not be $v$-semistable even when $X$ admits a K\"ahler--Ricci soliton (Example \ref{ex-unstable-canonical-extension}).
We therefore first prove a weighted slope bound with a correction term for its torsion-free quotients (Theorem \ref{thm-slope-equiv-ineq}).
We then use the intermediate estimate in the proof of Theorem \ref{thm-langer-ineq-equiv}, which is expressed in terms of the slopes of the quotients in the Harder--Narasimhan filtration.
Combining this estimate with the slope bound and Lemma \ref{lem-numerical-tail-bound} gives the equivariant Miyaoka--Yau inequality with a correction term.

To prove Theorem \ref{thm-main-equiv-langer}, we first establish a Bogomolov--Gieseker inequality for $v$-semistable torsion-free sheaves with the exponential weight $v(\mu)=e^{\mu_\xi}$.
The main difficulty is that such sheaves may be singular. Their equivariant second Chern classes cannot be defined by the usual Chern--Weil construction using a smooth metric, and the metric proof in \cite[Theorem 6.1]{HL24} does not apply directly.

We address these difficulties using the algebraic equivariant intersection theory of Edidin--Graham \cite{EG98, EG00}.
In this framework, we formulate weighted slope stability following \cite{HL24} and treat the equivariant second Chern classes of torsion-free sheaves.

In Section \ref{sec-preliminary}, we review the algebraic approach of \cite{EG98, EG00} and the complex-geometric approach of \cite{Inoue20, HL24}.
In Section \ref{sec-G-equivariant-intersection-stability}, we define weighted slopes and weighted intersection numbers for torsion-free sheaves, and establish basic results, including the existence of Harder--Narasimhan filtrations, following \cite{Kob14, GKP16a}.
In Section \ref{sec-BG-polynomial-weight}, we prove the Bogomolov--Gieseker inequality for $v$-semistable torsion-free sheaves with polynomial weights $v(\mu)=(c+b\mu_\xi)^m$, using the semi-simple fibration constructed in \cite{AJL23, HL24}.
In Section \ref{sec-BG-soliton}, we pass to the exponential weight $v(\mu)=e^{\mu_\xi}$ using the limit $(1+x/m)^m\to e^x$ as $m\to\infty$.
Combining the resulting inequality with the Harder--Narasimhan filtration gives the weighted equivariant Langer inequality.

\medskip
\noindent{\bfseries Acknowledgments.}
The authors thank the organizers of the ``31st Symposium on Complex Geometry'' held in Japan, where discussions provided the starting point for this work. 
They also thank Prof. Sebastian Boucksom for kindly answering their questions and sharing his unpublished notes.  
%T.\,H.\ thanks him and Prof. Siarhei Finski for helpful discussions concerning Appendix \ref{sec-Calabi}. 
They also thank Prof. Simon Jubert for his valuable comments concerning the greatest weighted Ricci lower bound $\delta_v'(X)$.
M.\,I.\ was supported by the Grant-in-Aid for Early Career Scientists, No.\ 22K13907. 
T.\,H.\ is supported by JSPS Grant-in-Aid for Scientific Research $(\mathrm{C})$, No.\ 26K06811. 

\medskip
\noindent{\bfseries Use of AI.}
Sections~\ref{sec-preliminary}--\ref{sec-BG-soliton} are based on the authors' earlier research, developed with the assistance of ChatGPT 5.5-5.6 Sol and Gemini 3.1.
After the authors had completed their earlier research, ChatGPT 6 Astra suggested the counterexamples in Section~\ref{sec-examples}.
After that, we revised our earlier work in Section~\ref{sec-proof-soliton-MY}, following suggestions from ChatGPT 6 Astra.
The authors have verified most of the paper, including all arguments in Sections~\ref{sec-Introduction}--\ref{sec-proof-soliton-MY}.
For Section~\ref{sec-examples}, we have checked the methods used in the computations, but have not independently verified all numerical calculations, including the values of $\tau$ in Examples \ref{ex-unstable-canonical-extension} and \ref{ex-negative-equiv-MY-sixfold}, 
because these calculations were too extensive for us to check by hand.

The authors used ChatGPT 5.5-5.6 Sol, ChatGPT 6 Astra, and Gemini 3.1
to improve the writing, search the literature, and discuss
and check mathematical arguments and calculations.
The authors reviewed almost all AI-assisted material included in
the paper and checked the arguments, calculations, and
references themselves.
The authors take full responsibility for the results and
proofs in this paper and for citing previous work accurately
and fully.

\medskip
\noindent{\bfseries Notation and conventions.}
%\subsection*{Notation and conventions}\label{subsec-notation}
In this paper, we denote by $\mathbb{N}$ the set of non-negative integers. 
Unless otherwise stated, we work over the complex number field $\C$. 
(Only in Subsection \ref{subsec-equiv-Chern-RR}, we consider a base field other than \(\C\).)
For any real vector space $M$, we write $M^\vee$ for its dual.
%All complex spaces are assumed to be \textit{second countable}.A \textit{complex analytic variety} or simply an \textit{analytic variety} is a reduced and irreducible complex space. Unless otherwise stated, analytic variety $X$ has complex dimension $n$. 
We employ the standard notation and conventions in \cite{Har77, KM98}.
%Additionally, we say that a normal analytic variety $X$ is a \emph{klt variety} if the pair $(X, 0)$ is klt; equivalently, $X$ is log terminal.

Unless stated otherwise, all sheaves considered in this paper are assumed to be coherent.
%Let $\mathcal{E}$ be a torsion-free (coherent) sheaf on a normal algebraic variety $X$.
For any torsion-free (coherent) sheaf $\mathcal{E}$, we define the dual reflexive sheaf $\mathcal{E}^\vee$ by
\(
\mathcal{E}^\vee \coloneqq \Hom(\mathcal{E}, \mathcal{O}_X).
\)
%For any positive integer $m \in \mathbb{N}$, we define the reflexive tensor power by\(\mathcal{E}^{[\otimes m]} \coloneqq (\mathcal{E}^{\otimes m})^{\vee\vee}.\)
%Given a morphism $f \colon Y \rightarrow X$ between analytic varieties, the reflexive pullback of $\mathcal{E}$ is defined as\(f^{[*]}\mathcal{E} \coloneqq (f^*\mathcal{E})^{\vee\vee}.\)
%A reflexive rank-one sheaf $\mathcal{E}$ is called a \textit{$\mathbb{Q}$-line bundle} if there exists an integer $m \in \mathbb{N}$ such that $\mathcal{E}^{[m]}$ is locally free. We denote the torsion part of a sheaf by $\tor$.
We also identify ``locally free sheaves of finite rank'' with ``vector bundles.''
When working on smooth complex varieties, we use the terms ``(Cartier) divisors'' and ``line bundles'' interchangeably.

\section{Equivariant Chern class and equivariant intersection theory}
\label{sec-preliminary}
In this section, we review the algebro-geometric intersection theory developed in \cite{EG98, EG00}, including equivariant Chow groups and the equivariant Grothendieck-Riemann--Roch theorem, and the complex-geometric equivariant intersection theory defined in \cite{Inoue20, HL24}.

\subsection{Equivariant Chern class and Langer's  formula}
\label{subsec-equiv-Chern-RR}
We recall the equivariant Chow groups of \cite{EG98, EG00}. 
Following \cite{EG98}, a scheme means a separated scheme of finite type over a field $k$, 
in particular, it is Noetherian. An algebraic group is always assumed to be linear.
To simplify some arguments, we also assume that $k$ has characteristic zero.  

In \cite{EG98}, the assumption that ``$X$ is an algebraic space'' is imposed because quotients such as $X/G$ do not necessarily exist as schemes. However, it is quite difficult to work with algebraic spaces. We therefore implicitly impose the following assumption throughout this paper.
\begin{assup}[{Existence of the quotient as a scheme; cf. \cite[Proposition 23]{EG98}}]
Let $G$ be a $g$-dimensional algebraic group, and let $X$ be an $n$-dimensional scheme with a $G$-action. Let $U$ be a scheme on which $G$ acts freely. If the principal bundle quotient $U\to U/G$ exists, then the principal bundle quotient
\(
X\times U \rightarrow (X\times U)/G
\)
exists in the category of schemes.
\end{assup}
Beginning with Section \ref{sec-G-equivariant-intersection-stability}, we mainly consider the case where $G$ is an algebraic torus.
Hence condition $(3)$ in \cite[Proposition 23]{EG98} is satisfied, and the above assumption holds. 
%In particular, the principal bundle quotient \(X\times U \rightarrow X\times^G U\) exists in the category of schemes (see \cite[Proposition 23]{EG98}). 
Therefore, it is not necessary to work with algebraic spaces.

\subsubsection{Equivariant Chern classes and Grothendieck--Riemann--Roch}

We recall Chern classes and Grothendieck-Riemann--Roch theorem in the $G$-equivariant setting according to  \cite[Section 2]{EG98},  \cite[Section 3]{EG00}, and \cite[Appendix B]{BHJ17}.
Our notation follows \cite[Appendix B]{BHJ17}.

\begin{defn}
\cite[Subsection 2.2]{EG98}
\label{defn-equivChow}
Let $G$ be a $g$-dimensional algebraic group, and let $X$ be an $n$-dimensional scheme with a $G$-action. 
For $i \in \Z$, choose a representation $V$ of $G$ and an open subset $U \subset V$ on which $G$ acts freely and such that $\codim (V \setminus U)>n-i$. Let $U\to U/G$ be the corresponding principal $G$-bundle, and define the diagonal action on $X\times U$ by $g(x,u)=(gx,gu)$. This action is also free, and the quotient
\[
X\times U \to (X\times U)/G\eqqcolon X_G
\]
exists in the category of schemes and is a principal $G$-bundle.

We define $\CH_i^G(X)$, the \emph{$G$-equivariant Chow homology group} of degree $i$, by
\[
\CH_i^G(X) \coloneqq \CH_{i+ \dim V -g}(X_G)=\CH_{i+ \dim X_G - \dim X}(X_G).
\]
We denote by $\CH^i_G(X)$ the operational Chow group defined in \cite[Subsection 2.6]{EG98}, and call it the \emph{$G$-equivariant Chow cohomology group} of degree $i$.
\end{defn}
%(We note that this is slightly different from the definition in \cite{EG00}.)
We do not discuss the definition of $\CH^i_G(X)$, since this construction is quite complicated.
%(see \cite[Subsection 2.6]{EG98}). 
Notice that $\CH^i_G(X) \cong \CH^G_{n-i}(X)$ when $X$ is smooth by \cite[Proposition 4]{EG98}.
Definition \ref{defn-equivChow} is independent of the choice of $V$ and $U$ by \cite[Proposition 1]{EG98}. 
%The proof uses Bogomolov's double fibration argument. 
The principal bundle $U\to U/G$ can also be regarded as a finite-dimensional approximation of the classifying bundle $EG\to BG$.

\begin{ex}\cite[Example 3.1]{EG98}
\label{ex-EG98-3.1}
Consider the case $\G \coloneqq \mathbb{G}_m$. Set $V \coloneqq k^{m+1}$ and $U \coloneqq  V \setminus \{0\}$. Then we obtain the principal bundle quotient
$$
U = V \setminus \{0\} \to U/\G = \mathbb{P}^{m}.
$$
If $X_{\G} \coloneqq (X\times U)/\G$, then each fiber of $X_{\G} \to \mathbb{P}^{m}$ is isomorphic to $X$.
When $k=\C$, set $E_m\G \coloneqq U=\C^{m+1}\setminus \{0\}$ and $B_m\G \coloneqq U/\G=\C\mathbb{P}^{m}$. Taking the limit as $m\to \infty$, we obtain the classifying space $S^{\infty} \to \C \mathbb{P}^{\infty}$ of $\G$.
\end{ex}

Let $G$ be an algebraic group, and let $X$ be a scheme with a $G$-action.
 Let $K_G^0(X)$ be the Grothendieck group of  $G$-equivariant vector bundles. It is a commutative ring with respect to the tensor product. A $G$-equivariant morphism induces a pull-back
$$
f^{*}\colon K^0_G(Y) \to K^0_G(X),
$$
which is functorial.
 Let $K_0^G(X)$ be the Grothendieck group of virtual $G$-equivariant coherent sheaves. It is a $K_G^0(X)$-module with respect to the tensor product, and every proper $G$-equivariant morphism $f\colon X\to Y$ induces a push-forward homomorphism
\[
f_!\colon K_0^G(X)\to K_0^G(Y)
\quad 
f_![\mathcal{E}] \coloneqq \sum_{q\in \mathbb{N}}(-1)^q [R^qf_*\mathcal{E}].
\]
%$K_G^0(\operatorname{Spec}k)=K_0^G(\operatorname{Spec}k)$ is identified with the representation ring $R(G)$. Hence $K_0^G(X)$ and $K^0_G(X)$ are $R(G)$-modules.

Let $\CH_G^\bullet(X)\coloneqq\bigoplus_{d\in \mathbb{N}} \CH_G^d(X)$ be the $G$-equivariant Chow cohomology ring. Since $\CH_G^d(X)$ may be nonzero for infinitely many $d\in \mathbb{N}$, as in \cite[Section 2]{EG00}, we define the formal completion $\widehat{\CH}_G^\bullet(X)$ of $\CH_G^\bullet(X)$ by
\[
\widehat{\CH}_G^\bullet(X)\coloneqq\prod_{d\in \mathbb{N}} \CH_G^d(X).
\]
We will often use the following groups with rational coefficients, defined by
\[
\CH_G^\bullet(X)_{\Q} \coloneqq 
\bigoplus_{d\in \mathbb{N}} \CH_G^d(X)\otimes_{\Z}\Q
\quad\text{and}\quad
\widehat{\CH}_G^\bullet(X)_{\mathbb{Q}} \coloneqq
\prod_{d\in \mathbb{N}} \CH_G^d(X)\otimes_{\Z}\Q.
\]
%We use the same convention below.

By \cite[Subsection 2.4]{EG98} and \cite[Definition 3.1]{EG00}, the $G$-equivariant Chern character gives a ring homomorphism
\begin{equation}
\label{eq-Gequiv-K}
\ch^G\colon K_G^0(X)\to \widehat{\CH}_G^\bullet(X)_{\mathbb{Q}},
\quad \ch^G(E) \coloneqq \ch(E_{G}),
\end{equation}
which is functorial with respect to pull-backs. 
Here, for a locally free sheaf $E$, we consider the quotient map $E \times U \to E_G$ and choose $U$ and $V$ so that $\codim(V \setminus U)$ is sufficiently large.
%so that $\codim(V \setminus U)>n$, as in Definition \ref{defn-equivChow}.
Moreover, for a $G$-equivariant line bundle $L$, we have
\[
\ch^G(L)=e^{c_1^G(L)}
=\left(\frac{c_1^G(L)^d}{d!} \right)_{d \in \N}
\in \widehat{\CH}_G^\bullet(X)_{\mathbb{Q}} 
%=\sum_{p=0}^{\infty}\frac{c_1^G(L)^p}{p!}
%\in \widehat{\CH}_G^\bullet(X)_{\mathbb{Q}}.
\]
 The $G$-equivariant Chow homology group
\(
\CH^G_\bullet(X) \coloneqq \bigoplus_{p\in \mathbb{Z}} \CH_p^G(X)
\)
is a $\CH_G^\bullet(X)$-module, and
\(
\CH_G^d(X)\cdot \CH_p^G(X)\subset \CH_{p-d}^G(X)
\)
holds.
Note that $\CH_p^G(X)=0$ for $p>\dim X$, whereas in general $\CH_p^G(X)$ may be nonzero for infinitely many negative integers $p$. As in the case of the Chow cohomology groups, set
\[
\widehat{\CH}_\bullet^G(X) \coloneqq \prod_{p\in \mathbb{Z}} \CH_p^G(X).
\]
We also define
\(
\CH^G_\bullet(X)_{\Q}\coloneqq \CH^G_\bullet(X)\otimes_{\Z}\Q
\) and \(
\widehat{\CH}_\bullet^G(X)_{\mathbb Q}\coloneqq \widehat{\CH}_\bullet^G(X)\otimes_{\Z}\Q.
\)
%\quad \text{and} \quad
%\widehat \CH_\bullet^G(X)_{\mathbb Q} \coloneqq \prod_{p\in\mathbb Z}\CH_p^G(X)\otimes_{\Z}\Q.
%\CH^G_\bullet(X)_{\Q} \coloneqq \bigoplus_{p\in \mathbb{Z}} \CH_p^G(X)\otimes_{\Z}\Q.
%and  define $\widehat \CH_\bullet^G(X)_{\mathbb Q} \coloneqq \prod_{p\in\mathbb Z}\CH_p^G(X)\otimes\mathbb Q$.

We first recall the Grothendieck--Riemann--Roch theorem and summarize the form needed in this paper, following \cite[Theorem 2.16]{Inoue20}.
%This may be a slight modification of \cite[Theorem 3.2]{EG00}.
\begin{thm}\cite[Theorem 3.1]{EG00} 
\label{thm-GRR-equiv}
For every algebraic group $G$ and every algebraic $G$-scheme $X$, one can associate a homomorphism of abelian groups
\[
\tau_X^G\colon K_0^G(X)\rightarrow \widehat \CH_\bullet^G(X)_{\mathbb Q},
\]
such that $\tau_X^G$ satisfies the following properties:
\begin{enumerate}
\item \label{item-GRR-equiv-1}
For any proper $G$-equivariant morphism $f\colon X\to Y$ of algebraic schemes and any $\alpha\in K_0^G(X)$, the Grothendieck--Riemann--Roch theorem holds:
\[
f_*\tau_X^G(\alpha)=\tau_Y^G(f_!\alpha).
\]
Here, if $\alpha=[\mathcal{E}]$ is represented by a $G$-equivariant coherent sheaf $\mathcal{E}$, then $f_!\alpha$ is the element
\[
f_!\alpha \coloneqq \sum_i (-1)^i [R^if_*\mathcal{E}]\in K_0^G(Y).
\]
The higher direct image sheaves $R^if_*\mathcal{E}$ are naturally $G$-equivariant.
\item  \label{item-GRR-equiv-2} 
For any $\alpha\in K_0^G(X)$ and $\beta\in K^0_G(X)$, 
\[
\tau_X^G(\alpha\otimes\beta)=\ch^G(\beta)  \cap \tau_X^G(\alpha).
\]
\item \label{item-GRR-equiv-3} 
For any $G$-invariant closed subscheme $Z\subset X$ of pure dimension $p$, 
\[
\tau_X^G(\mathcal O_Z)_{(p)}=[Z]_G\in \CH_p^G(X)_{\mathbb Q}.
\]
Here $\tau_X^G(\mathcal O_Z)_{(p)}\in \CH_p^G(X)_{\mathbb Q}$ denotes the component of degree $p$ of $\tau_X^G(\mathcal O_Z)$.
\item  \label{item-GRR-equiv-4} 
If $X$ is smooth and $\Td^G_{p}(X) \coloneqq \tau_X^G(\mathcal O_X)_{(p)} \in  \CH_{p}^{G}(X)_{\mathbb Q}$, then
\[
\Td^G_{p}(X)=\td_G^{n-p}(X) \cap [X]_G
\in \CH_{p}^{G}(X)_{\mathbb Q}.
\]
Here $\td_G^{n-p}(X) \in \CH^{n-p}_{G}(X)_{\mathbb Q}$ is the usual $G$-equivariant Todd class of the tangent sheaf $T_X$, defined by Chern roots.
\end{enumerate}
\end{thm}

The following lemma appears in \cite[Example 18.3.11]{Ful93} when $G$ is trivial. It may be known, but we include a proof because we could not find one.
\begin{lem}
\label{lem-length-tau}
Let $G$ be an algebraic group, and let $X$ be an $n$-dimensional integral algebraic $G$-scheme. Let $\eta$ be the generic point of $X$. 
Then for any $G$-equivariant coherent sheaf $\mathcal{E}$, 
$$
\tau_X^G(\mathcal{E})_{(n)}=\operatorname{length}_{\mathcal O_{X,\eta}}(\mathcal{E}) \cdot [X]_G
\in {\CH}_n^G(X)_{\mathbb Q}
$$
holds, where $\operatorname{length}_{\mathcal O_{X,\eta}}(\mathcal{E})$ denotes the length of $\mathcal{E}_\eta$ over $\mathcal O_{X,\eta}$.
\end{lem}

\begin{proof}
We use the following d\'evissage of coherent sheaves from \cite[Lemma 30.12.6]{stacks-project}.

\begin{lem}[{\cite[Lemma 30.12.6]{stacks-project}}]
Let $X$ be a Noetherian scheme. Let $\mathcal{P}$ be a property of coherent sheaves on $X$ satisfying the following conditions:
\begin{enumerate}
\item For any short exact sequence of coherent sheaves, if two of the three sheaves have property $\mathcal{P}$, then the remaining one also has property $\mathcal{P}$.
\item For any integral closed subscheme $Z \subset X$ with generic point $\xi$, there exists a coherent sheaf $\mathcal{G}$ satisfying the following conditions:
\begin{enumerate}
\item[$(a)$] $\operatorname{Supp}(\mathcal{G}) = Z$,
\item[$(b)$]  $\mathcal{G}_{\xi}$ is annihilated by $\mathfrak{m}_{\xi}$,
\item[$(c)$]  $\dim_{\kappa(\xi)} \mathcal{G}_{\xi} = 1$, and
\item[$(d)$]  property $\mathcal{P}$ holds for $\mathcal{G}$.
\end{enumerate}
\end{enumerate}
Then property $\mathcal{P}$ holds for every coherent sheaf on $X$.
\end{lem}
Let $\mathcal{P}$ denote the property
\(
\tau_X^G(\mathcal{E})_{(n)}=\operatorname{length}_{\mathcal O_{X,\eta}}(\mathcal{E}) \cdot [X]_G.
\)
It is enough to check conditions (1) and (2) as in the above lemma.
$(1)$ follows immediately, since, for any exact sequence of coherent sheaves
$
0 \to \mathcal{E}_1 \to \mathcal{E}_2 \to \mathcal{E}_3 \to 0,
$
the additivity of $\tau^G_X$ and of length gives
$$
\tau^G_X(\mathcal{E}_2)
=
\tau^G_X(\mathcal{E}_1) + \tau^G_X(\mathcal{E}_3) 
\quad \text{and} \quad 
\operatorname{length}_{\mathcal O_{X,\eta}}(\mathcal{E}_2)
=
\operatorname{length}_{\mathcal O_{X,\eta}}(\mathcal{E}_1) + \operatorname{length}_{\mathcal O_{X,\eta}}(\mathcal{E}_3).
$$

We prove $(2)$. Let $Z \subset X$ be an integral closed subscheme, and let
$i \colon Z \hookrightarrow X$ be the inclusion. Then $\mathcal{G} \coloneqq \mathcal{O}_{X}/\mathcal{I}_Z$ satisfies the desired conditions. Indeed, conditions (a)--(c) follow immediately from $\mathcal{G}=i_{*}\mathcal{O}_Z$.
For condition (d), if $Z=X$, then
$$
\tau^G_X(\mathcal{G})_{(n)}
=\tau^G_X(\mathcal{O}_{X})_{(n)}
\underalign{\text{(Thm. \ref{thm-GRR-equiv} \ref{item-GRR-equiv-3})}}{=}[X]_{G}
=\operatorname{length}_{\mathcal O_{X,\eta}}(\mathcal{O}_{X})[X]_{G}.
$$
If $Z \neq X$, then $i_{*}\tau^G_{Z}(\mathcal{O}_Z)$ is an equivariant cycle of dimension at most $\dim X-1$. Hence,
$$
\tau^G_{X}(\mathcal{G})_{(n)}
\underalign{\text{(by $R^qi_{*}=0$)}}{=}\tau^G_{X}(i_{!}\mathcal{O}_Z)_{(n)}
\quad
\underalign{\text{(Thm. \ref{thm-GRR-equiv} \ref{item-GRR-equiv-1})}}{=}\quad
\left(i_{*}\tau^G_{Z}(\mathcal{O}_Z)\right)_{(n)}
%\underbrace{\left(i_{*}\tau^G_{Z}(\mathcal{O}_Z)\right)}_{\text{equivariant cycle of dimension $\le \dim X-1$}}_{(n)}
=0.
$$
On the other hand, $\operatorname{length}_{\mathcal O_{X,\eta}}(\mathcal{G})=0$. Thus, condition (d) also holds. \end{proof}

\subsubsection{Chern characters of coherent sheaves}

\begin{lem}
\label{lem-equiv-ch-coh}
Let $G$ be an algebraic group, and let $X$ be an $n$-dimensional algebraic $G$-scheme. Assume that every $G$-equivariant coherent sheaf admits a finite $G$-equivariant locally free resolution.
%For example, this holds when $X$ is smooth and projective by Thomason's equivariant resolution theorem.
Then there exists an abelian group homomorphism
$$
\ch'^G\colon K_{0}^G(X) \to \widehat \CH^{\bullet}_G(X)_{\mathbb Q}
$$
such that, for every $\mathcal{E} \in  K_{0}^G(X)$, we have
$$
\tau^G_{X}(\mathcal{E})
=
\ch'^G(\mathcal{E}) \cap \tau^G_{X}(\mathcal{O}_X)
%\underbrace{\tau^G_{X}(\mathcal{O}_X)}_{\eqqcolon\Td^G(X) \in \CH_{\bullet}^G(X)_{\mathbb Q}}.
$$
Thus, we also denote $\ch'^G$ by $\ch^G$.
In particular, if $X$ is smooth, then
$$
\tau_{X}^G(\mathcal{E})
=
\left(\ch^G(\mathcal{E}) \cdot \td^G(T_X)\right) \cap [X]_G.
%\ch^G(\mathcal{E}) \cdot \underbrace{\td^G(T_X)}_{\in \CH^{\bullet}_{G}(X)_{\mathbb Q}} \cap [X]_G.
$$
\end{lem}
Note that, when $X$ is smooth and projective, every $G$-equivariant coherent sheaf admits a finite $G$-equivariant locally free resolution by \cite[Corollary 5.8]{Tho3} (see also \cite[Introduction]{EG00}).

\begin{proof}
Since every coherent sheaf admits a finite locally free resolution, there exists an abelian group isomorphism
$
K_{0}^G(X) \cong K^{0}_G(X).
$
Thus we can define
$$
\ch'^G \colon K_{0}^G(X) \to K^{0}_G(X) \overset{\ch}{\to} \widehat \CH^{\bullet}_G(X)_{\mathbb Q}.
$$
Equivalently, if there exists a finite resolution by $G$-equivariant locally free sheaves
$
0 \to E_k \to E_{k-1} \to\cdots \to E_1\to E_0 \to \mathcal{E} \to 0,
$
then
$$
\ch'^G(\mathcal{E}) \coloneqq \sum_{i=0}^{k} (-1)^i \ch^G(E_i)
\in \widehat \CH^{\bullet}_{G}(X)_{\Q}.
$$
By the additivity of \(\tau_X^G\), the latter assertion follows from
\begin{align*}
\tau^G_{X}(\mathcal{E})
\underalign{\text{(Additivity)}}{=}
\sum_{i=0}^{k} (-1)^i \tau^G_{X}(E_i)
\underalign{\text{(Thm. \ref{thm-GRR-equiv}  \ref{item-GRR-equiv-2})}}{=}
\quad
\sum_{i=0}^{k} (-1)^i (\ch^G(E_i) \cap \tau^G_{X}(\mathcal{O}_X))
=
\ch'^G(\mathcal{E}) \cap \tau^G_{X}(\mathcal{O}_X).
\end{align*}
In particular, if $X$ is smooth, then $\tau^G_{X}(\mathcal{O}_X)=\Td^G(X) = \td^G(T_X) \cap [X]_G$ by Theorem \ref{thm-GRR-equiv} \ref{item-GRR-equiv-4}, and the claim follows.
\end{proof}

Using this, we can define the Chern classes $c^G_1(\mathcal{E}), \ldots, c^G_n(\mathcal{E})$ for a $G$-equivariant coherent sheaf $\mathcal{E}$. 
Since only $c^G_1(\mathcal{E})$ and $c^G_2(\mathcal{E})$ are used in this paper, we define only these two classes.
\begin{defn}
\label{defn-chern-Bogomolov}
Let $G$ be an algebraic group, and let $X$ be a smooth projective $G$-scheme. For any $G$-equivariant coherent sheaf $\mathcal{E}$, we define the \emph{$G$-equivariant Chern character} $\ch^G(\mathcal{E})$, and define the first and second Chern classes by
$$
c^{G}_1(\mathcal{E}) \coloneqq  \ch^{G}_{1}(\mathcal{E}) \in \CH^{1}_{G}(X)_{\Q},
\quad 
c^{G}_2(\mathcal{E}) \coloneqq  
\frac{1}{2}\ch^{G}_1(\mathcal{E})^{2}
-\ch^{G}_{2}(\mathcal{E})  \in \CH^{2}_{G}(X)_{\Q}.
$$
Moreover, for any torsion-free coherent sheaf $\mathcal E$ of rank $r$, we define the \emph{Bogomolov discriminant} by
$$
\Delta^{G}(\mathcal E) \coloneqq  2rc^G_2(\mathcal E) - (r-1)c^G_1(\mathcal E)^2 
=\ch^{G}_{1}(\mathcal E)^2 - 2r \ch^{G}_{2}(\mathcal E)
\in  \CH^{2}_{G}(X)_{\Q}.
$$
\end{defn}

\begin{cor}
\label{cor-exact-ch-vect}
Let $G$ be an algebraic group, and let $X$ be a smooth projective $G$-scheme. For an exact sequence of $G$-equivariant coherent sheaves
$
0 \to \mathcal S \to \mathcal  E \to \mathcal  Q \to 0,
$
we have
$$
{\ch^G}(\mathcal  E) = {\ch^G}(\mathcal  S) + {\ch^G}(\mathcal  Q).
$$
In particular,
$$
c^{G}_1(\mathcal  E) = c^{G}_1(\mathcal  S) +  c^{G}_1(\mathcal  Q)
\quad \text{and} \quad 
c^{G}_2(\mathcal  E) = c^{G}_2(\mathcal  S) +  c^{G}_2(\mathcal  Q) + c^{G}_1(\mathcal  S) c^{G}_1(\mathcal  Q).
$$
\end{cor}
\begin{proof}
By the group isomorphism $K^0_{G}(X) \cong K_0^G(X)$, we may assume that $\mathcal S, \mathcal E, \mathcal Q$ are locally free. Recall from \eqref{eq-Gequiv-K} that the Chern character is defined by using the locally free sheaf $\mathcal {E}_G$ on $X_G$ as
${\ch^G}(\mathcal E) \coloneqq \ch(\mathcal {E}_G).$
Thus the claim follows from the exactness of $0 \to \mathcal {S}_G \to \mathcal {E}_G \to \mathcal {Q}_G \to 0$ on $X_G$.
\end{proof}

\begin{lem}[{Langer's formula; cf. \cite{Lan04}}]
\label{lem-Langer-ineq-chow}
Let $G$ be an algebraic group, and let $X$ be a smooth projective $G$-scheme. Let $\mathcal{E}$ be a $G$-equivariant torsion-free sheaf, and assume that there exists a filtration
\[
0 \eqqcolon \mathcal{E}_0 \subsetneq \mathcal{E}_1 \subsetneq \ldots \subsetneq \mathcal{E}_{s}  \coloneqq  \mathcal{E}
\]
such that each quotient
$\mathcal{G}_i  \coloneqq  \mathcal{E}_i / \mathcal{E}_{i-1}$ is a torsion-free sheaf of rank $r_i$.
Then the following identity holds:
\begin{align*}
%\label{eq-Langer-filtration}
\frac{\Delta^{G}(\mathcal{E}) }{r} 
 &= \ \sum_{i=1}^{s} \frac{\Delta^{G}(\mathcal{G}_{i})}{r_i} 
 - \frac{1}{r} \sum_{1 \le i < j \le s} r_i r_j \left( \frac{c^G_1(\mathcal{G}_{i})}{r_i} - \frac{c^G_1(\mathcal{G}_{j})}{r_j} \right)^2 \in  \CH^{2}_{G}(X)_{\Q}.
\end{align*}
\end{lem}
\begin{proof}
By Definition \ref{defn-chern-Bogomolov}, $\frac{\Delta^{G}(\mathcal{E}) }{r} = \frac{c^G_1(\mathcal{E})^2}{r} - 2  \ch^{G}_{2}(\mathcal{E})$. Therefore
\begin{equation}
\label{eq-Langer-filtration-1}
\begin{aligned}
\frac{\Delta^{G}(\mathcal{E}) }{r} - \sum_{i=1}^{s} \frac{\Delta^{G}(\mathcal{G}_{i})}{r_i} 
&\underalign{}{=} \quad\left( \frac{c^G_1( \mathcal{E})^2}{r} - \sum_{i=1}^{s}\frac{c^G_1( \mathcal{G}_i)^2}{r_i} \right)
-2\left( \ch^{G}_{2}(\mathcal{E}) - \sum_{i=1}^{s}  \ch^{G}_{2}(\mathcal{G}_i) \right)\\
&\underalign{\text{(Additivity of $\ch$)}}{=}\quad \frac{c^G_1( \mathcal{E})^2}{r} - \sum_{i=1}^{s}\frac{c^G_1( \mathcal{G}_i)^2}{r_i}.
\end{aligned}
\end{equation}
%Thus it is enough to prove
%$$
%r\left(  \sum_{i=1}^{s}\frac{c^G_1( \mathcal{G}_i)^2}{r_i} - \frac{c^G_1( \mathcal{E})^2}{r}\right)
%=\sum_{1 \le i < j \le l} r_i r_j \left( \frac{c^G_1(\mathcal{G}_{i})}{r_i} - \frac{c^G_1(\mathcal{G}_{j})}{r_j} \right)^2.
%$$
From $r = \sum_{i=1}^{s}r_i$ and $c^G_1(\mathcal{E}) = \sum_{i=1}^{s}c^G_1(\mathcal{G}_i)$, it also follows that
\begin{equation}
\label{eq-Langer-filtration-2}
\begin{aligned}
&\sum_{1 \le i < j \le s} r_i r_j \left( \frac{c^G_1(\mathcal{G}_{i})}{r_i} - \frac{c^G_1(\mathcal{G}_{j})}{r_j} \right)^2 
\underalign{\text{}}{=}  
\frac{1}{2} 
\sum_{i=1}^{s}\sum_{j=1}^{s}r_i r_j \left( \frac{c^G_1(\mathcal{G}_{i})}{r_i} - \frac{c^G_1(\mathcal{G}_{j})}{r_j} \right)^2  
\\
%&
%\underalign{}{=}
%\frac{1}{2} 
%\sum_{i=1}^{s}\sum_{j=1}^{s}
%\left( \frac{r_jc^G_1(\mathcal{G}_{i})^2}{r_i} +
%\frac{r_ic^G_1(\mathcal{G}_{j})^2}{r_j} 
%-  2 c^G_1(\mathcal{G}_{i})c^G_1(\mathcal{G}_{j})\right)\\
&\underalign{}{=}
\sum_{i=1}^{s}\sum_{j=1}^{s}
\left(\frac{r_jc^G_1(\mathcal{G}_{i})^2}{r_i} - c^G_1(\mathcal{G}_{i})c^G_1(\mathcal{G}_{j}) \right) 
\underalign{}{=}
r \sum_{i=1}^{s}
\frac{c^G_1( \mathcal{G}_i)^2}{r_i}
- c^G_1( \mathcal{E})^2.
\end{aligned}
\end{equation}
%In the last equality, we use $r = \sum_{i=1}^{s}r_i$ and $c^G_1( \mathcal{E}) = \sum_{i=1}^{s}c^G_1( \mathcal{G}_i)$. 
Thus the desired equation follows from \eqref{eq-Langer-filtration-1} and \eqref{eq-Langer-filtration-2}.
\end{proof}

\begin{lem}
\label{lem-torsion-eff-cycle}
Let $G$ be an algebraic group, and let $X$ be a smooth projective $G$-scheme. Let $\mathcal{E}$ be a $G$-equivariant coherent sheaf, and set $Z \coloneqq \Supp(\mathcal{E})$. Assume that $Z$ has dimension $n-p$ with $0<p\leq n$, and let $Z_1, \ldots, Z_k \subset Z$ be the $G$-invariant irreducible components of $Z$ of dimension $n-p$. Then
$$
\ch^G_0(\mathcal{E})  =\cdots =\ch^G_{p-1}(\mathcal{E})=0
\quad
\text{and}
\quad 
\ch^G_p(\mathcal{E}) \cap [X]_G
= \sum_{i=1}^{k} 
\operatorname{length}_{\mathcal O_{X,\eta_i}} (\mathcal{E}_{\eta_i}) [Z_i]_G.
$$
Here $\eta_i$ is the generic point of the reduced irreducible component $Z_i$, and the length is taken over $\mathcal O_{X,\eta_i}$.

\end{lem}
\begin{proof}
In this proof, equip $Z$ with the closed subscheme structure defined by $\operatorname{Ann}_{\mathcal O_X}(\mathcal E)$. This ideal sheaf is $G$-invariant since $\mathcal E$ is $G$-equivariant.
Let $i\colon Z\hookrightarrow X$ be the inclusion. Then $\mathcal E=i_*(\mathcal E|_Z)$, since $\operatorname{Ann}_{\mathcal O_X}(\mathcal E)$ annihilates $\mathcal E$. Hence we obtain
\begin{align*}
\ch^G(\mathcal{E})\cap \tau^G_{X}(\mathcal{O}_X)
& \underalign{\text{(Lem. \ref{lem-equiv-ch-coh})}}{=}
\tau^G_{X}(\mathcal{E})
\underalign{\text{}}{=}
\tau^G_{X}(i_{*}(\mathcal{E}\vert_{Z}))
\underalign{\text{(by $R^qi_{*}=0$)}}{=}
\tau^G_{X}(i_{!}(\mathcal{E}\vert_{Z}))
\\
& \underalign{\text{(Thm. \ref{thm-GRR-equiv}  \ref{item-GRR-equiv-1})}}{=}
\underbrace{i_{*}\tau^G_{Z}(\mathcal{E}\vert_{Z})}_{\raisebox{-0.5ex}{\tiny \text{equivariant cycle of dimension $\le n-p$}}}.
\end{align*}
Since $\tau^G_{X}(\mathcal{O}_X)_{(n)}= [X]_G$ by Theorem \ref{thm-GRR-equiv}  \ref{item-GRR-equiv-3}, by comparing the components of dimension $n$, we obtain
$$
\ch_0^G(\mathcal{E}) \cap [X]_G
= \left( i_{*}\tau^G_{Z}(\mathcal{E}\vert_{Z}) \right)_{(n)} 
%\underalign{\text{by $n-p <n$}}{=} 
\underalign{\raisebox{-0.5ex}{\tiny\text{(by $n-p <n$)}}}{=}
0.
$$
Since $X$ is smooth, $X_G$ is also smooth. Hence, by \cite[Theorem 17.4.2]{Ful93}, the cap product map
$$
\cap[X]_G \colon \CH^{j}_{G}(X)=\CH^{j}(X_G) \overset{}{\to} \CH_{\dim X_G - j}(X_G) =\CH_{n-j}^{G}(X)
$$
is an isomorphism. Therefore, $\ch_0^G(\mathcal{E})=0$.
The same argument shows that 
$$
\ch_0^G(\mathcal{E})  =\cdots =\ch_{p-1}^G(\mathcal{E})=0
\quad \text{and} \quad
\ch_{p}^{G}(\mathcal{E})\cap [X]_{G}
=
 i_{*}\left(\tau^{G}_{Z}(\mathcal{E}\vert_{Z})\right)_{(n-p)}.
$$

Choose a sufficiently large mixed quotient $Z_G=(Z\times U)/G$ as in \cite[Section~3.2]{EG00}, and set $d\coloneqq\dim U-\dim G$.
Let $\mathcal E_G$ be the coherent sheaf on $Z_G$ induced by $\mathcal E|_Z$, and let $\eta_{i,G}$ be the generic point of $(Z_i)_G$.
The highest-dimensional component formula in ordinary Riemann--Roch gives
\[
\tau_{Z_G}(\mathcal E_G)_{(n-p+d)}
\quad
\underalign{\text{\cite[Thm~18.3]{Ful93}}}{=}
\quad
\sum_{i=1}^k
\operatorname{length}_{\mathcal O_{Z_G,\eta_{i,G}}}
((\mathcal E_G)_{\eta_{i,G}})[(Z_i)_G].
\]
Under the identification $\mathrm{CH}_{n-p+d}(Z_G)\simeq\mathrm{CH}_{n-p}^G(Z)$, the left-hand side represents $\tau_Z^G(\mathcal E|_Z)_{(n-p)}$.
Since
\(
\operatorname{length}_{\mathcal O_{Z_G,\eta_{i,G}}}
((\mathcal E_G)_{\eta_{i,G}})
=
\operatorname{length}_{\mathcal O_{X,\eta_i}}(\mathcal E_{\eta_i}),
\)
pushing forward by $i$ gives the desired formula.

\end{proof}

For the first and second Chern classes, Lemma \ref{lem-torsion-eff-cycle} can be stated as follows. Notice that $c^{G}_2(\mathcal{E}) \coloneqq  
\frac{1}{2}c^{G}_1(\mathcal{E})^{2}
-\ch^{G}_{2}(\mathcal{E})$.
\begin{cor}
\label{cor-1st2nd-semipositive}
Let $G$ be an algebraic group, and let $X$ be a smooth projective $G$-scheme. Let $\mathcal{E}$ be a $G$-equivariant coherent sheaf, and set $Z \coloneqq \Supp(\mathcal{E})$. Assume that $Z$ has codimension $p \ge 1$, and let $Z_1, \ldots, Z_k \subset Z$ be the $G$-invariant irreducible components of $Z$ of dimension $n-p$.

\begin{enumerate}
\item If $p=1$, then $\ch^G_0(\mathcal{E})=0$ and
$$\ch^G_1(\mathcal{E}) \cap [X]_G
=c^G_1(\mathcal{E}) \cap [X]_G
= \sum_{i=1}^{k}
\operatorname{length}_{\mathcal O_{X,\eta_i}}(\mathcal E_{\eta_i})
[Z_i]_G.
$$
\item If $p=2$, then $\ch^G_0(\mathcal{E})=\ch^G_1(\mathcal{E})=0$ and
$$
\ch^G_2(\mathcal{E}) \cap [X]_G
=-c^G_2(\mathcal{E}) \cap [X]_G
= \sum_{i=1}^{k} 
\operatorname{length}_{\mathcal O_{X,\eta_i}} (\mathcal{E}_{\eta_i}) [Z_i]_G.
$$
\end{enumerate}
\end{cor}

\subsection{Equivariant cohomology group and Cartan model}
\label{subsec-equivariant-complex-algebraic}
Since we wish to apply the methods of Subsection \ref{subsec-equiv-Chern-RR} to complex geometry, we now explain the relationship between the $G$-equivariant Chow group and the equivariant singular cohomology group when the base field is $\C$.
This part is discussed in detail in \cite{GS99, BGV04} and  is also briefly summarized in \cite{Lib07, Inoue20}.

\subsubsection{Relation between the equivariant cohomology group and the $G$-equivariant Chow group}
We first recall the equivariant singular cohomology group.
\begin{defn}\cite[Chapter 1]{GS99}
\label{defn-clasifying}
Let $X$ be a topological space with a continuous left action of a topological group $G$. Let $EG \to BG$ be the classifying space of $G$, and assume that $EG$ is equipped with a left $G$-action. 

The product $EG \times X$ admits the left $G$-action defined by \(g \cdot (p, x) \coloneqq (gp, gx)\), and we set \(EG \times_G X \coloneqq (EG \times X)/G\).
For $K=\Z, \Q, \R$, we define the \emph{equivariant singular cohomology group} by 
\[
H^{\bullet}_G(X, K) \coloneqq H^{\bullet}(EG \times_G X, K).
\]
\end{defn}
For consistency with the notation of this paper, we use left actions on $X$ and $EG$, whereas \cite[Section 4]{Inoue20} uses right actions. 
%By the definition of the classifying space, principal $G$-bundles over a topological space $Y$ are in one-to-one correspondence with continuous maps $Y \to BG$, modulo homotopy equivalence.
%By the definition we obtain the following.
\begin{lem}
\label{lem-cartan-equivcoh}
Let $X$ be a topological space with a continuous  left action of an algebraic torus $\G \coloneqq (\mathbb{G}_m)^{l}$.  Let $\T\subset \G$ be the maximal compact torus, that is, $\T=(S^1)^l$. 
Then, for $K=\Z, \Q, \R$, we have the ring isomorphism
\begin{equation}
\label{eq-rest-GT}
\res_{\T}^{\G}
\colon
H^{\bullet}_\G(X,K)
\overset{\sim}{\to}
H^{\bullet}_\T(X,K).
\end{equation}
\end{lem}
\begin{proof}
For simplicity, assume that $l=1$. The classifying space of $\G$ is
$$
E\G=\lim_{m \to \infty}\C^{m+1}\setminus \{0\}
\to 
\lim_{m \to \infty}\C\mathbb{P}^{m}=B\G.
$$
On the other hand, the classifying space of $\T$ is
$$
E\T=\lim_{m\to \infty}S^{2m+1}
\to 
\lim_{m\to \infty}\C\mathbb{P}^{m}=B\T.
$$
Since there is a homotopy equivalence $\C^{m+1}\setminus \{0\} \sim S^{2m+1}$, it induces a homotopy equivalence $E\G \sim E\T$. This induces the desired isomorphism on cohomology.
\end{proof}

According to \cite[Subsection 2.8]{EG98}, we now explain how equivariant singular cohomology is related to the $G$-equivariant Chow group. Let $X$ be an $n$-dimensional complex algebraic variety with a left action of a complex algebraic group $G$. Let $K=\Z, \Q, \R$. 
The equivariant Borel--Moore homology $H^{G}_{BM,i}(X, K)$ is defined by
\[
H^{G}_{BM,i}(X, K)
 \coloneqq H_{BM,i+2\dim V-2g}(X_G, K)
\quad \text{where } X_G=X\times U /G,
\]
where $V$, $U$, and $g$ are as in Definition \ref{defn-equivChow}. Note that $\dim V-g=\dim X_G-\dim X$. As in Definition \ref{defn-equivChow}, if $V \setminus U$ has sufficiently large codimension, then this definition is independent of the choice of $U$ and $V$. The group $H^{G}_{BM,i}(X, K)$ is also called the $G$-equivariant locally finite homology in \cite[Section 4]{Inoue20}.

The $G$-equivariant cycle map between the homology groups is defined by
\begin{equation*}
cl_{G} \colon \CH_i^G(X)  \coloneqq  \CH_{i+\dim V - g}(X_G) \overset{cl}{\longrightarrow} H_{BM,2i+2\dim V -2g}(X_G, \Z) =H^{G}_{BM,2i}(X, \Z).
\end{equation*}
Here $cl \colon \CH_{i+\dim V - g}(X_G) \to H_{BM,2i+2\dim V -2g}(X_G, \Z)$ denotes the cycle map defined in \cite[Section 19.1]{Ful93}.

When $X$ is smooth, $X_G$ is also smooth. Thus, as in \cite[Subsection 2.8]{EG98}, the $G$-equivariant cycle map between the cohomology groups is defined by
\begin{equation}
\label{eq-defn-cycle-map}
cl^{G}\colon\ \CH_G^i(X) \cong \CH^G_{n -i}(X) \overset{cl_G}{\longrightarrow} %H_{BM,2n-2i}(X_G, \Z)
H^G_{BM,2n-2i}(X, \Z) \cong H_G^{2i}(X, \Z).
\end{equation}

%For the classifying space $EG\to BG$ of $G$, an open subset $U\subset V$ gives a topological approximation of $EG$ (see \cite[Subsection 2.8]{EG98}). Hence, for any $i$, by choosing $U$ and $V$ appropriately, we have $H^{2 i}(X_G,\Z)=H_G^{2i}(X, \Z)$. 
% from $X_G$ is also smooth,  \(H_{BM,i}(X_G, \Z) \cong H^{2\dim X_{G}-i}(X_G,\Z) = H_G^{2\dim X-i}(X, \Z)\).
% \(H_{BM,i}(X_G, \Z)\) is dual to \(H^{2\dim X_{G}-i}(X_G,\Z) = H_G^{2\dim X-i}(X, \Z)\).
%\begin{rem}
%Equivariant Borel--Moore homology can be defined even when $X$ is not a complex algebraic variety; it is called locally finite homology in \cite[Section 4]{Inoue20}.
%\end{rem}

\subsubsection{Relation between the equivariant cohomology group and the Cartan model}
When $X$ is a manifold, equivariant cohomology can be computed by differential forms using the Cartan model.
We recall the definition of the Cartan model following \cite[Subsection 4.2]{Inoue20}, taking $\hbar=1$. (For more details, see \cite[Chapter 4]{GS99} and \cite[Section 3]{Lib07})

Let $M$ be a connected orientable $C^\infty$ real manifold of dimension $m$ equipped with a smooth left action of a compact Lie group $T$.
Let $\mathfrak t$ be the Lie algebra of $T$, and let $\exp \colon \mathfrak t \to T$ be the exponential map. For $\xi\in\mathfrak t$, we define the real vector field $\xi_M \colon M \to T_{\R}M$ at $x \in M$ by
\[
(\xi_M)_x
\coloneqq \left.\frac{d}{dt}\right\rvert_{t=0}\bigl(\exp(t\xi) \cdot x\bigr).
\]
For any differential form $\omega$, define the action by $t \bullet \omega \coloneqq \theta(t^{-1})^{*} \omega$, where $\theta \colon T \to  \Diff(M)$ denotes the action. Set
\begin{equation}
C^{p,q}_T(M) \coloneqq (\Sym^{p}\mathfrak t^{\vee}\otimes \Omega^{q-p}(M))^{T}, 
\end{equation}
which  consists of $T$-invariant $\Sym^{p}\mathfrak t^{\vee}$-valued differential $(q-p)$-forms. 
By identifying elements of the symmetric product $\Sym^{p}\mathfrak t^{\vee}$ with homogeneous polynomial maps of degree $p$ on $\mathfrak t$, we write
$$\varphi=\sum_k P_k\otimes \varphi_k\in C^{p,q}_T(M),$$
where $P_k$ is a homogeneous polynomial of degree $p$. For $\xi\in\mathfrak t$, set $\varphi(\xi) \coloneqq \sum_k P_k(\xi)\cdot \varphi_k$. Then, for any $\xi\in\mathfrak t$ and $\varphi \in C^{p,q}_T(M)$, 
the differentials $d, \delta$ are defined by
\begin{align*}
&d\colon C^{p,q}_T(M)\to C^{p,q+1}_T(M)
\quad  (d\varphi)(\xi) \coloneqq d(\varphi(\xi)),
\\
&\delta \colon C^{p,q}_T(M)\to C^{p+1,q}_T(M)
\quad
(\delta\varphi)(\xi) \coloneqq \iota_{\xi_{M}}(\varphi(\xi)).
\end{align*}
With these differentials, $(C^{p,q}_T(M), d, \delta)$ becomes a double complex.

\begin{defn}[{\cite[Subsection 10.1]{GS99}, \cite[Subsection 3.3]{Lib07}, \cite[Subsection 4.2]{Inoue20}}]
\label{defn-gysimmap}
In the above setting, the \emph{Cartan model $H^{\bullet}_{\dR,T}(M)$ of equivariant cohomology} is the cohomology of the total complex $(\Omega_T^{\bullet}(M), d_T)$, which is defined by
$$
\Omega_T^{l}(M) \coloneqq \bigoplus_{p+q=l} C^{p,q}_T(M)
=\bigoplus_{2i+j=l}(\Sym^{i}\mathfrak t^{\vee}\otimes \Omega^{j}(M))^{T},
\quad 
d_T \coloneqq  d + \delta \colon \Omega_T^{l}(M) \to \Omega_T^{l+1}(M).
$$
% Assume that $p + q - m$ is a positive even integer, and set $K \coloneqq \frac{p+q-m}{2} \in \N$. 
Assume that $M$ is compact. For any $\varphi=\sum_k P_k\otimes \varphi_k\in C^{p,q}_T(M)$, where $P_k \in \Sym^{p}\mathfrak t^{\vee}$ and $\varphi_k \in \Omega^{q-p}(M)$, we define the integration map by
\[
\pi_*\varphi
 \coloneqq  \sum_{k} P_k \cdot \left(\int_M \varphi_k \right)  \in \Sym^{p} \mathfrak{t}^{\vee}.
% \coloneqq  \sum_i P_i  \cdot \int_M [\varphi_i]_{m}  \in \Sym(\mathfrak{t}^{\vee})
\]
Note that $\int_M \varphi_k=0$ unless $q-p=m$.
This induces a homomorphism on the Cartan model
\(
\pi_*\colon H_{\dR, T}^{\bullet}(M)\rightarrow H_{\dR, T}^{\bullet-m}(\mathrm{pt})=\Sym \, \mathfrak{t}^{\vee}.
\)
\end{defn}
Although the definition of $\pi_{*}$ is somewhat complicated,
 its value $(\pi_*\varphi)(\xi)$ at $\xi\in\mathfrak t$ can be written more simply as follows:
 %For $\varphi \in \Omega_T^{\bullet}(M)$, we have
\begin{equation}
\label{eq-cartalmodel-substitute}
(\pi_*\varphi)(\xi) =
 \int_M [\underbrace{\varphi(\xi)}_{\in {\Omega^{\bullet}(M)}^{T}}]_{m} 
 \in \mathbb R.
%\quad (\forall \xi\in\mathfrak t)
\end{equation}
Here $[\cdot]_{m}$ means taking the component of degree $m$.
The following theorem shows that equivariant cohomology can be computed by  using the Cartan model.
\begin{thm}[{\cite[Subsections 2.5, 4.2, and 10.1]{GS99}, \cite[Theorem 21]{Lib07}, \cite[Subsection 4.2]{Inoue20}}]
\label{thm-equiv-caltan-isom}
Let $M$ be an orientable and connected $m$-dimensional $C^\infty$ real manifold with a smooth left action of a compact Lie group $T$. Let $ET \to BT$ be the classifying space of $T$. Then there exists a natural ring isomorphism
\begin{equation}
\label{eq-defn-Psi}
\Psi_T \colon H^{\bullet}_T(M,\R)
\overset{\sim}{\to}H^{\bullet}_{\dR,T}(M).
\end{equation}
Under this isomorphism, the map
\[
\pi_*\colon H_{\dR, T}^{\bullet}(M)\rightarrow H_{\dR, T}^{\bullet-m}(\mathrm{pt})
\]
in Definition \ref{defn-gysimmap} is identified with the Gysin map
$$
\pi'_* \colon H^{\bullet}_T(M,\R)=H^{\bullet}(ET\times_T M, \R)
\to H^{\bullet-m}(BT, \R)=H^{\bullet-m}_T(\mathrm{pt},\R)
$$
associated with the projection $ET\times_T M \to ET\times_T \mathrm{pt}=BT$. 
We also denote this map $\pi'_* $ by $\pi_*$.
\end{thm}
%Indeed, this follows from the natural isomorphism $H^{\bullet}_T(M,\R)\cong H^{\bullet}_{\dR,T}(M)$.

\subsection{Equivariant intersection theory with weights}
\label{subsec-equivariant-intersection-theory}

In this subsection, we briefly review \cite{HL24}. To avoid repetition, throughout Subsection \ref{subsec-equivariant-intersection-theory}, unless otherwise stated, let $\T=(S^1)^{l}$ be a compact torus and $\mathfrak{t} \coloneqq \Lie(\T)$ be the Lie algebra.
Let $X$ be a compact K\"ahler manifold with a holomorphic $\T$-action, and let $\omega$ be a $\T$-invariant K\"ahler form. 
Moreover, we assume that this $\T$-action is Hamiltonian, that is, there exists a $\T$-equivariant $C^\infty$ map \( \mu\colon X\to \mathfrak t^{\vee}\), called the \emph{moment map},
such that
\[
d\mu_\xi= -\iota_{\xi_X}\omega
\qquad
\forall\,\xi\in\mathfrak t.
\]
Here $\xi_X$ is the vector field on $X$ generated by the action of $\exp(\xi) \in \T$, and we set
$$
\mu_\xi \coloneqq \langle\mu,\xi\rangle\colon 
X \to \R, \quad  x \mapsto \langle\mu(x), \xi\rangle.
$$
We denote by $\langle \bullet, \bullet\rangle$ the canonical pairing $\mathfrak t^{\vee} \otimes \mathfrak t \to \R$ and call $P=\mu(X) \subset \mathfrak t^{\vee}$ the \emph{moment polytope}.

\subsubsection{Weighted contraction}

\begin{defn}\cite[Definition 2.1]{HL24}
\label{defn-Emoment}
Let $E \to X$ be a $\T$-equivariant vector bundle, and let $h$ be a $\T$-invariant Hermitian metric. 
The \emph{moment map} $\Phi_h\in C^{\infty}(X, \mathfrak t^{\vee}\otimes \End(E))$ for $(E, h)$ is a $\T$-invariant $C^\infty$-section of $\mathfrak{t}^{\vee}\otimes \End(E)$ such that
\[
F_h(\bullet,\xi_X)=\nabla^h\langle \Phi_h,\xi\rangle
\quad \forall \xi \in \mathfrak{t}.
\]
Here $\nabla^h$ is the Chern connection of $h$, and $F_h=\bar{\partial} (h^{-1} \partial h)$ is the Chern curvature.
\end{defn}
By \cite[Proposition 2.2]{HL24}, there exists a canonically defined moment map $\Phi_h$ defined by
\begin{equation}
\label{eq-canonically-moment-map}
\langle \Phi_h,\xi\rangle
 \coloneqq 
\nabla^{h}_{\xi_X}-\mathcal L^{E}_{\xi_X} \quad \forall \xi \in \mathfrak{t}.
\end{equation}
Here $\mathcal L^{E}_{\xi_X} \colon E \to E$ denotes the Lie derivative; that is, for a local section $e \colon U \to E$, set
\[
\mathcal L^{E}_{\xi_X}(e)
\coloneqq
\left.\frac{d}{dt}\right\rvert_{t=0}\bigl(\exp(-t\xi) \cdot e\bigr).
\]
Notice that the action on local sections is given by
\(
(\tau \cdot e)(x) \coloneqq \tau(e(\tau^{-1} x))
\)
for any $x \in U$ and $\tau\in \T$.

\begin{defn}\cite[Definition 2.4]{HL24}
\label{defn-weighted-lambda} 
Let $E \to X$ be a $\T$-equivariant vector bundle, and let $h$ be a $\T$-invariant Hermitian metric. 
For a smooth positive function $v\colon P \to \R_{>0}$, 
we define the \emph{$v$-weighted contraction} by
\[
\Lambda_{\omega,v}(F_h+\Phi_h)
 \coloneqq v(\mu)\Lambda_\omega F_h+\langle \Phi_h, dv(\mu)\rangle
%\quad \text{with values in $End E$}.
\]
Here $\Lambda_\omega$ denotes the contraction defined by
\(
\Lambda_\omega F_h  \coloneqq  
\frac{n\,F_h\wedge \omega^{n-1}}{\omega^n}
\).
\end{defn}

\begin{defn}\cite[subsection 8.1]{HL24}
\label{defn-equiv-chern}
Let $E$ be a $\T$-equivariant vector bundle, and let $h$ be a $\T$-invariant Hermitian metric. We define the \emph{$\T$-equivariant Chern character} by
\begin{equation}
\label{eq-defn-equiv-chern}
\ch^{\T}(E)\; \coloneqq \;
\left\{
\underbrace{\tr
\!\Bigl(\exp\Bigl(\frac{\sqrt{-1}}{2\pi}(F_h +\Phi_h)\Bigr)\Bigr)}_{\text{even-degree $\T$-differential form}}
\right\}
 \underalign{\text{(Def. \ref{defn-gysimmap})}}{\in} 
 \widehat{H^{2 \bullet}_{\dR, \T}}(X).
\end{equation}
Here $\widehat{H^{2 \bullet}_{\dR, \T}}(X)$ denotes the formal completion of $H^{2 \bullet}_{\dR, \T}(X)$. This is independent of the choice of $h$. In particular,
\begin{align*}
\ch^{\T}_{1}(E)
& \coloneqq \
\left\{
\tr \Bigl(\frac{\sqrt{-1}}{2\pi}(F_h +\Phi_h)\Bigr)
\right\} \in H^{2}_{\dR, \T}(X)
\\
\ch^{\T}_{2}(E)
& \coloneqq 
\left\{
\frac{1}{2}\tr
\Bigl(-\frac{1}{4\pi^{2}}(F_h +\Phi_h)^{2}\Bigr)
\right\} \in H^{4}_{\dR, \T}(X).
\end{align*}
Also  the first and second equivariant Chern classes are defined by  $c^{\T}_{1}(E) \coloneqq \ch^{\T}_{1}(E)$ and 
$c^{\T}_{2}(E) \coloneqq \frac{1}{2}\ch^{\T}_{1}(E)^2 - \ch^{\T}_{2}(E)$.
\end{defn}

\subsubsection{Equivariant intersection theory via the Cartan model}

\begin{defn}[{\cite[Subsection 3.2]{Inoue20}, \cite[Subsubsection 4.1.1]{HL24}}]
\label{defn-intersection-converge}
Let $\alpha_{\T}\in H_{\dR, \T}^{2}(X)$ and let $\beta_{\T}\in H_{\dR, \T}^{2m}(X)$. Let $f\colon\mathbb{R}\to\mathbb{R}$ be a real function that admits a convergent power series expansion at the origin of the form
\(
f(x)= \sum_{k=0}^{\infty}\frac{a_k}{k!}x^k
\).
We define
\[
(\beta_{\T}\cdot f(\alpha_{\T}))_{\T}
 \coloneqq  \pi_{*}(\beta_{\T}\cdot f(\alpha_{\T})) 
=\sum_{k=0}^{\infty}
\frac{a_k}{k!}
\cdot \pi_{*}(\beta_{\T}\cdot {\alpha_{\T}}^{k})
%\underbrace{\pi_{*}(\beta_{\T}\cdot {\alpha_{\T}}^{k})}_{\in H_{\dR,T}^{2m + 2m -2n}(pt)}
\underalign{\text{(Def. \ref{defn-gysimmap})}}{\in} \widehat{\Sym\,\mathfrak{t}^{\vee}}.
\]
Here $\pi_*\colon H_{\dR, T}^{\bullet}(X)\rightarrow \Sym \, \mathfrak{t}^{\vee}$ is defined as in Definition \ref{defn-gysimmap}, and $\widehat{\Sym\,\mathfrak{t}^{\vee}}$ denotes the formal completion of $\Sym\, \mathfrak{t}^{\vee}$.
\end{defn}
By Definition \ref{defn-gysimmap}, we have
\[
\pi_{*}(\beta_{\T}\cdot {\alpha_{\T}}^{k})(\xi)
\underalign{\eqref{eq-cartalmodel-substitute}}{=}
 \int_X [\beta_{\T}(\xi) \wedge \alpha_{\T}(\xi)^{k}]_{2n} 
% \int_X [\underbrace{\beta_{\T}(\xi) \wedge \alpha_{\T}(\xi)^{k}}_{\in {\Omega^{\bullet}(X)}^{\T}}]_{2n} 
 \in \mathbb R.
\quad (\forall \xi\in\mathfrak t)
\]
The use of completion makes the definition somewhat cumbersome, but the following useful theorem is available.
\begin{thm}\cite[Theorem 1.1]{Inoue20}
If $f$ has a power series which converges on all of $\mathbb{R}$, then the formal power series $(\beta_{\T}\cdot f(\alpha_{\T}))_\T$ is compactly absolutely convergent.
In particular, $(\beta_{\T}\cdot f(\alpha_{\T}))_\T(\xi)$ converges and is a real number.
\end{thm}

\begin{rem}[{\cite[remark 4.1]{HL24}}]
\label{rem-intersection-converge}
Since the power series $f(\alpha_{\T})$ converges on all of $\mathfrak{t}$, it extends to an entire function on the complex vector space $\mathfrak{t}\otimes\mathbb{C}$. In particular, for $\xi\in\mathfrak{t}$, the value $(\beta_{\T}\cdot f(\alpha_{\T}))_\T(\sqrt{-1}\xi)$ is well-defined.
\end{rem}
Using this and the Fourier transform, we define the weighted equivariant intersection number for a general weight $v\colon P\to \R$ and equivariant K\"ahler class $\{\omega + \mu\}\in  H_{\dR, \T}^{2}(X)$.
\begin{defn}[{\cite[Definition 4.2]{HL24}}]
\label{defn-equiv-intersection-number}
Let $\alpha_{\T}=\{\omega + \mu\}\in  H_{\dR, \T}^{2}(X)$ be an equivariant K\"ahler class and let $\beta_{\T}\in  H_{\dR, \T}^{2m}(X)$.
For a smooth function $v\colon P\to\mathbb{R}$, choose a compactly supported smooth extension $\bar{v}$ to all of $\mathfrak{t}^{\vee}$ whose support contains $P$, and define its Fourier transform by
\[
\widehat{v}(\xi)\coloneqq\int_{ \mathfrak{t}^{\vee}}\bar{v} (y)e^{-\sqrt{-1}\langle y,\xi\rangle}\,\underline{d\mu}(y).
\]
Here $\underline{d\mu}$ denotes the Lebesgue measure on $\mathfrak{t}^{\vee}$.

We define the \emph{weighted equivariant intersection number with respect to $v$} by
\[
(\beta_{\T}\cdot v(\alpha_{\T})) \coloneqq \int_{\mathfrak{t}}(
\beta_{\T}\cdot e^{\alpha_{\T}})_{\T}(\sqrt{-1}\xi) \cdot \widehat{v}(\xi)\,\underline{d\xi} \in \R.
\]
Here $(\beta_{\T}\cdot e^{\alpha_{\T}})_{\T}(\sqrt{-1}\xi)$ is defined as in Definition \ref{defn-intersection-converge} and Remark \ref{rem-intersection-converge}, and $\underline{d\xi}$ denotes the Lebesgue measure on $\mathfrak{t}$ dual to $\underline{d\mu}$, divided by $(2\pi)^{\dim T}=(2\pi)^{l}$.
\end{defn}
By \cite[Lemma 4.3]{HL24}, the real number $(\beta_{\T}\cdot v(\alpha_{\T}))$ is well-defined, the integral converges, and it is independent of the choice of the extension $\bar{v}$ of $v$ from $P$ to $\mathfrak{t}^{\vee}$. Moreover, $(\beta_{\T}\cdot v(\alpha_{\T}))$ depends only on the equivariant cohomology classes involved, and not on the choice of equivariant representatives.
This definition is not convenient for computations, but the following useful result is available.
\begin{lem}\cite[Lemma 4.3]{HL24}
\label{lem-HL24-4.3}
Under the setting of Definition \ref{defn-equiv-intersection-number}, assume that $v(\mu)=\widetilde{v}(\langle \mu,\xi\rangle)$ holds for some real analytic function $\widetilde{v} \colon \R \to \R$ and some $\xi\in\mathfrak{t}$. Then, for any $\beta_{\T}\in H_{\dR, \T}^{2m}(X)$, 
\[
\underbrace{(\beta_{\T}\cdot v(\alpha_{\T}))}
_{\text{(Def. \ref{defn-equiv-intersection-number})}}
=\underbrace{(\beta_{\T}\cdot f^{(m)}(\alpha_{\T}))_{\T} (\xi)}
_{\text{(Def. \ref{defn-intersection-converge})}}\in\R.
\]
Here $f \colon \R \to \R$ is a real function satisfying $f^{(n)}(x)=\widetilde{v}(x)$, and $f^{(m)}(x)$ denotes the $m$-th derivative of $f$.
%, and $(\beta_{\T}\cdot f^{(l)}(\alpha_{\T}))_{\T}(\xi)$ is defined as in Definition \ref{defn-intersection-converge}.
\end{lem}
Finally, we observe that, for the first Chern class of a vector bundle, the intersection number can be expressed in terms of $v$-weighted contraction as in Definition \ref{defn-weighted-lambda}.

\begin{lem}\cite[Lemma 4.4]{HL24}
\label{lem-intersection-chern}
Let $\alpha_{\T}=\{\omega + \mu\}\in H_{\dR, \T}^{2}(X)$ be an equivariant K\"ahler class, and let $E$ be a $\T$-equivariant holomorphic vector bundle. For a smooth positive function $v\colon P\to\mathbb{R}_{>0}$, 
$$
( v(\alpha_{\T}))
=\displaystyle \int_X v(\mu) \omega^{[n]}
\quad
\quad \text{and} \quad
(c_1^{\T}(E)\cdot v(\alpha_{\T}))
=\int_X \tr\!\left(\frac{\sqrt{-1}}{2\pi}\Lambda_{\omega,v}(F_h+\Phi_h)\right)\omega^{[n]}.
$$
Here we set $\omega^{[n]} \coloneqq \frac{\omega^{n}}{n!}$.
\end{lem}

\begin{ex}
\label{ex-explicit-slope}
Assume that $v(\mu)=\widetilde{v}(\mu_\xi)$ holds for some real analytic function $\widetilde{v} \colon \R \to \R$ and  $\xi\in\mathfrak{t}$. 
Set $\mu_\xi \coloneqq \langle \mu,\xi\rangle \colon X \to \R$ and define $v'(\mu) \coloneqq \widetilde{v}^{(1)}(\mu_\xi)$. 
Then Lemma \ref{lem-intersection-chern} gives
$$
(c_1^\T(E)\cdot v(\alpha_{\T}))
=
\int_{X}v(\mu) 
\tr \Bigl(\frac{\sqrt{-1}}{2\pi}F_h\Bigr) \wedge \omega^{[n-1]}
+  v'(\mu)\tr \left(\frac{\sqrt{-1}}{2\pi}\langle \Phi_h,\xi\rangle\right) \omega^{[n]}.
$$
For example, if there exists $m \in \N$ such that $v(\mu) = \mu_{\xi}^m$, then
$$
( v(\alpha_{\T}))
=
\displaystyle \int_X \mu_\xi^{m} \omega^{[n]}, 
\quad 
(c_1^\T(E)\cdot v(\alpha_{\T}))
=
\int_{X}
\mu_{\xi}^m
\tr \Bigl(\frac{\sqrt{-1}}{2\pi}F_h\Bigr) \wedge \omega^{[n-1]}
+  m \mu_{\xi}^{m-1}\tr \left(\frac{\sqrt{-1}}{2\pi}\langle \Phi_h,\xi\rangle\right) \omega^{[n]}.
$$
Similarly, if $v(\mu) = e^{\mu_{\xi}}$, then
$$
( v(\alpha_{\T}))
=
\displaystyle \int_X  e^{\mu_{\xi}} \omega^{[n]}, 
\quad 
(c_1^\T(E)\cdot v(\alpha_{\T}))
=
\int_{X}
 e^{\mu_{\xi}}
\tr \Bigl(\frac{\sqrt{-1}}{2\pi}F_h\Bigr) \wedge \omega^{[n-1]}
+   e^{\mu_{\xi}} \tr \left(\frac{\sqrt{-1}}{2\pi}\langle \Phi_h,\xi\rangle\right) \omega^{[n]}.
$$
\end{ex}

\section{Weighted equivariant intersection theory and slope stability}
\label{sec-G-equivariant-intersection-stability}
In this section, we combine \cite{EG98} and \cite{HL24} to define weighted equivariant intersection numbers for coherent sheaves. We also extend the notion of slope stability in \cite{HL24} and prove basic facts such as the existence of the Harder--Narasimhan filtration.

For this purpose, we consider the following setup.

\begin{setup}
\label{setup-soliton-situation}
\begin{itemize}
\item Let $\G=\mathbb{G}_m^{l}$ be an algebraic torus, and let $\T\subset \G$ be the maximal compact torus, that is, $\T=(S^1)^{l}$.
\item Let $X$ be an $n$-dimensional smooth projective variety over $\C$ with a left  $\G$-action. Moreover, assume that $X$ admits a $\G$-equivariant ample line bundle.
%Let $(X,L)$ be an $n$-dimensional polarized smooth variety with a $\G$-action. In other words, $X$ is a smooth projective manifold over $\C$, and $L$ is a $\G$-equivariant ample line bundle.
\item Let $\omega$ be a $\T$-invariant K\"ahler metric. % such that $[\omega]=c_1(L)$.
\item Assume that $(X, \omega)$ has a Hamiltonian $\T$-action, and denote its moment map by $\mu \colon X \to \mathfrak{t}^{\vee}$ and its moment polytope by $P \coloneqq \mu(X) \subset \mathfrak{t}^{\vee}$.
\item Set the equivariant K\"ahler class $\alpha_{\T}  \coloneqq  \{\omega + \mu\} \in H^2_{\dR, \T}(X).$
\item Let $v\colon P \to \R_{>0}$ be a positive smooth function.
\end{itemize}
\end{setup}
In the above setup, the projectivity of $X$ is needed to make \cite{EG98} and \cite{HL24} compatible in the algebraic setting. Moreover, the smooth projectivity of $X$ is needed for the existence of finite locally free resolutions.

\subsection{Weighted equivariant intersection numbers of coherent sheaves}

\begin{thmdefn}(cf. \cite[Lemma 2.1]{Inoue20})
\label{thmdefn-chern-equivalence}
Assume Setup \ref{setup-soliton-situation}. Then there exists a natural homomorphism $\widetilde{\Psi} \colon \widehat{\CH_{\G}^{\bullet}}(X)_{\Q} \to \widehat{H^{2\bullet}_{\dR, \T}}(X)$ such that, for any $\G$-equivariant locally free sheaf $E$, we have
$$
\ch^{\T}(E)
=\widetilde{\Psi}
(\ch^{\G}(E))
\in \widehat{H^{2\bullet}_{\dR, \T}}(X).
$$
Here $\ch^{\G}\colon K_G^0(X)\to \widehat{\CH}_G^\bullet(X)_{\mathbb{Q}}$ is the map given by \eqref{eq-Gequiv-K}, and $\ch^{\T}(E) \in \widehat{H^{2\bullet}_{\dR, \T}}(X)$ is given in Definition \ref{defn-equiv-chern}. 

Using this property, for any $\G$-equivariant coherent sheaf $\mathcal{E}$, we define
\begin{equation}
\label{eq-equivalence-chernG-chernT}
\ch^{\T}(\mathcal{E}) \coloneqq \widetilde{\Psi}(\ch^{\G}(\mathcal{E}))\in \widehat{H^{2\bullet}_{\dR, \T}}(X).
\end{equation}
This is well-defined by Lemma \ref{lem-equiv-ch-coh}.
\end{thmdefn}
In the usual case, namely, when both $\G$ and $\T$ are trivial, this follows from the axiomatic definition of Chern classes in \cite[Chapter 2]{Kob14}. This result is probably already known in the above setting, but we could not find a precise reference. (For a closely related result, see \cite[Lemma 2.1]{Inoue20}). Although it follows from a combination of known facts, we include the proof for completeness.

\begin{proof}
For simplicity of notation, we omit the completion symbol $\widehat{\ }$ by restricting to a fixed degree.
First, following \cite[Subsection 5.4]{Lib07}, we recall the topological construction of the Chern character.
For $G=\G$ or $\T$, and for $K=\Q$ or $\R$,
let $K_G^0(X)$ be the Grothendieck group of $G$-equivariant vector bundles and set
\[
K_G^0(X)_K \coloneqq K_G^0(X)\otimes_{\Z}K.
\]
By Lemma \ref{lem-cartan-equivcoh}, the groups $\G$ and $\T$ have the same classifying space, namely,
\[
E\G=E\T=
\left(\lim_{m\to\infty}\bigl(\C^{m+1}\setminus\{0\}\bigr)\right)^l
\quad \text{and} \quad
B\G=B\T=
\left(\lim_{m\to\infty}\C\mathbb P^m \right)^l.
\]
Thus the topological $G$-equivariant Chern character map 
\(
\ch^{G,\mathrm{top}}_K\colon
K_G^0(X)_K
\rightarrow
H_G^{2\bullet}(X,K)
\)
is defined by 
\[
\ch^{G,\mathrm{top}}_K(E)
 \coloneqq 
\ch^{\mathrm{top}}(EG \times_G E)
\in H^{2\bullet}( EG \times_G X,K)
= H_G^{2\bullet}(X,K).
\]
%For a $G$-equivariant locally free sheaf $E\to X$, the quotient\(E\times EG/G \rightarrow EG \times_G X\)is an ordinary complex vector bundle. 
Here
\(
\ch^{\mathrm{top}}\colon
K^0(EG \times_G X)_K
\rightarrow
H^{2\bullet}(EG \times_G X,K)
\)
is the ordinary topological Chern character. With this notation, consider the following diagram:
\begin{equation*}
\xymatrix@C=35pt@R=25pt{
\CH_{\G}^{\bullet}(X)_{\Q}
\ar[r]^{cl^\G}_{\eqref{eq-defn-cycle-map}}
\ar@/^2.5pc/[rrrr]^{\widetilde{\Psi}}
%\ar@{}@<0ex>[dr]|{\circlearrowright}
&
H_{\G}^{2\bullet}(X, \Q) \ar[r]^{\otimes \R}
&
H_{\G}^{2\bullet}(X, \R)\ar[r]^{\res_{\T}^{\G}}_{\eqref{eq-rest-GT}}
&
H_{\T}^{2\bullet}(X, \R)\ar[r]^{\Psi_\T}_{\eqref{eq-defn-Psi}}
&
H^{2\bullet}_{\dR, \T}(X).
\\
K_\G^0(X)_{\Q}
\ar[u]^{\ch^\G}_{\eqref{eq-Gequiv-K}}
\ar@{=}[r]
&
K_\G^0(X)_{\Q} 
\ar[u]^{\ch^{\G, \mathrm{top}}_{\Q}} 
\ar@{^{(}->}[r]
&
K_\G^0(X)_{\R}
\ar[u]^{\ch^{\G, \mathrm{top}}_{\R}} 
\ar@{^{(}->}[r]
&  K_\T^0(X)_{\R} 
\ar@{=}[r]
\ar[u]^{\ch^{\T, \mathrm{top}}_{\R}} 
&
K_\T^0(X)_{\R} \ar[u]^{\ch^\T}_{\eqref{eq-defn-equiv-chern}}
\\
}
\end{equation*}
Thus it is enough to show that the four squares in this diagram are commutative. Once this is proved, we may define
\(
\widetilde{\Psi} \coloneqq \Psi_{\T}\circ \res_{\T}^{\G}\circ \otimes \R \circ cl^\G
\)
as above.

We first show that the first square commutes. Let $E$ be a $\G$-equivariant vector bundle. For each $i$, take $U$ and $V$ as in Definition \ref{defn-equivChow}. (In this setting, by Example \ref{ex-EG98-3.1}, take $V=(\C^{m+1})^l$ and $U=(\C^{m+1}\setminus\{0\})^l$ for sufficiently large $m\gg i$.) Set
\[
X_\G \coloneqq X\times U/\G,
\quad \text{and} \quad
E_\G \coloneqq E\times U/\G.
\]
Then the following statements hold.
\begin{itemize}
\item By the definition \eqref{eq-Gequiv-K} of the equivariant Chern character, we have
\(
\ch^\G(E)
=
\ch(E_\G)
\in
\CH^\bullet_\G(X)_\Q
\).
Thus  by \eqref{eq-defn-cycle-map}, the cycle map $cl^\G$ is defined by the ordinary cycle map on $X_\G$. Hence
\(
cl^\G\bigl(\ch^\G(E)\bigr)
=
cl\bigl(\ch(E_\G)\bigr).
\)
\item By \cite[Chapter~19]{Ful93}, the ordinary cycle map is compatible with the Chern character, and therefore
\(
cl\bigl(\ch(E_\G)\bigr)
=
\ch^{\mathrm{top}}(E_\G).
\)
\item By the definition of $\ch^{\G,\mathrm{top}}_\Q$, and since $U \to U/\G$ gives a finite-dimensional approximation of $E\G \to B\G$, 
\[
\ch^{\mathrm{top}}(E_\G)
=
\ch^{\mathrm{top}}(E\times U/\G)
=
\ch^{\mathrm{top}}(E\G\times_\G E)
=
\ch^{\G,\mathrm{top}}_\Q(E).
\]
\end{itemize}
Combining these equalities, we obtain
\(
cl^\G\bigl(\ch^\G(E)\bigr)
=
\ch^{\G,\mathrm{top}}_\Q(E)
\).
Thus the first square is commutative. 

The second square clearly commutes. The third square follows from Lemma \ref{lem-cartan-equivcoh}, since $\G$ and $\T$ have the same classifying space. The fourth square follows from \cite[Theorem 7.38]{BGV04} or \cite[Section 5]{Lib07}.
\end{proof}

\begin{defn}
\label{defn-weighted-intersection-coherent}
Assume Setup \ref{setup-soliton-situation}, and let $\mathcal{E}$ be a $\G$-equivariant
coherent sheaf. 
For any $p \in \N$, following Definition \ref{defn-equiv-intersection-number}, we define $(\ch^{\T}_p(\mathcal{E})\cdot v(\alpha_{\T})) \in \R$ by
$$
(\ch^{\T}_p(\mathcal{E})\cdot v(\alpha_{\T}))
 \coloneqq  \underbrace{\left(
\widetilde{\Psi}(\ch_{p}^\G(\mathcal{E}))
\cdot v(\alpha_{\T})\right)
}
%(\underbrace{\widetilde{\Psi}(\ch_{p}^\G(\mathcal{E}))}_{\in H^{2p}_{\dR, \T}(X)}\cdot v(\alpha_{\T}))}
_{\text{(Def. \ref{defn-equiv-intersection-number})}}
\in \R.
$$
Here $\widetilde{\Psi} \colon \widehat{\CH_{\G}^{\bullet}}(X)_{\Q} \to \widehat{H^{2\bullet}_{\dR, \T}}(X)$ is the map defined in Theorem \ref{thmdefn-chern-equivalence}.
\end{defn}
When $\mathcal{E}$ is a vector bundle, this definition agrees with the definition of the intersection number in \cite[Section 4]{HL24} by Theorem \ref{thmdefn-chern-equivalence}. As in Definition \ref{defn-chern-Bogomolov}, we also define $(c_1^{\T}(\mathcal{E})\cdot v(\alpha_{\T}))$, $(c_2^{\T}(\mathcal{E})\cdot v(\alpha_{\T}))$, and $(c_1^{\T}(\mathcal{E})^2 \cdot v(\alpha_{\T}))$ in the same way, and define the Bogomolov discriminant by
\begin{equation}
\label{eq-Bogomolov-T-discriminant}
(\Delta^{\T}(\mathcal{E})\cdot v(\alpha_{\T}))
 \coloneqq 
(\widetilde{\Psi}(\Delta^{\G}(\mathcal{E}))\cdot v(\alpha_{\T}))
\in \R.
\end{equation}

\begin{prop}
\label{prop-semipositive-slope}
Assume Setup \ref{setup-soliton-situation}. Let $\mathcal{E}$ be a $\G$-equivariant coherent sheaf. Assume that there exists an irreducible
$\G$-invariant cycle $D$ of codimension $1$ and a rational number $a \ge 0$ such that
$$c^{\G}_{1}(\mathcal{E}) \cap [X]_\G = a [D]_{\G} \in \CH^{\G}_{n-1}(X)_{\Q}.$$
Let $i \colon \widetilde{D} \to X$ be the composition of a $\G$-equivariant resolution of $D$ and the inclusion $D \hookrightarrow X$. Then 
$$
(c_1^{\T}(\mathcal{E})\cdot v(\alpha_{\T}))
=a (v(\alpha_{\T}\vert_{\widetilde{D}})) 
\in \R.
$$
In particular, $(c_1^{\T}(\mathcal{E})\cdot v(\alpha_{\T})) \ge 0$. Moreover, if equality holds, then $a=0$.
\end{prop}
The existence of a $\G$-equivariant resolution of $D$ follows from the existence of a functorial resolution of $D$ (for example, see also \cite[Proposition 9.1]{Kol07a}).
\begin{proof}
It is enough to prove the following claim.

\begin{claim}
\label{claim-projection-formula}
For every $m \in \N$, we have
$$(c_1^{\T}(\mathcal{E})\cdot \alpha_{\T}^m)_{\T}
=a (\alpha_{\T}^m\vert_{\widetilde{D}})_{\T} 
\underalign{\text{(Def. \ref{defn-intersection-converge})}}{\in} 
\Sym\,\mathfrak{t}^{\vee}
%H^{2m-2n+2}_{\dR, \T}(pt),
$$
where $\alpha_{\T}\vert_{\widetilde{D}} \coloneqq i^{*} \alpha_{\T} \in H^{2}_{\dR, \T}(\widetilde{D})$.
In particular, for every $\xi \in \mathfrak{t}$, we have
$$(c^{\T}_{1}(\mathcal{E}) \cdot e^{\alpha_{\T}})_{\T}(\sqrt{-1}\xi) = a (e^{\alpha_{\T}\vert_{\widetilde{D}}})_{\T}(\sqrt{-1}\xi) 
\underalign{\text{(Rem \ref{rem-intersection-converge})}}{\in} \C.
$$
\end{claim}

Assuming Claim \ref{claim-projection-formula}, the assertion is proved as follows. 
As in Definition \ref{defn-equiv-intersection-number}, take an extension $\bar{v} \colon \mathfrak {t}^{\vee} \to \R$ of $v \colon P \to \R_{>0}$, and set
$\widehat{v}(\xi) \coloneqq \int_{\mathfrak{t}^{\vee}}\bar{v} (y)e^{-\sqrt{-1}\langle y,\xi\rangle}\,\underline{d\mu}(y)$. 
Then we obtain 
\begin{align*}
(c_1^{\T}(\mathcal{E})\cdot v(\alpha_{\T}))
&\underalign{\text{(Def. \ref{defn-equiv-intersection-number})}}{=} 
\int_{\mathfrak{t}}(c^{\T}_{1}(\mathcal{E}) \cdot e^{\alpha_{\T}})_{\T}(\sqrt{-1}\xi) \cdot \widehat{v}(\xi)\,\underline{d\xi}
\\
&\underalign{\text{(Claim \ref{claim-projection-formula})}}{=} 
a \int_{\mathfrak{t}}  (e^{\alpha_{\T}\vert_{\widetilde{D}}})_{\T}(\sqrt{-1}\xi) 
\cdot \widehat{v}(\xi)\,\underline{d\xi}
\underalign{\text{(Def. \ref{defn-equiv-intersection-number})}}{=} 
a (v(\alpha_{\T}\vert_{\widetilde{D}})).
\end{align*}
This proves the equality. The remaining assertion follows from
$$
a (v(\alpha_{\T}\vert_{\widetilde{D}})) 
\underalign{\text{(Lem. \ref{lem-intersection-chern})}}{=} 
a  \int_{\widetilde{D}} v(\mu \circ i \vert_{\widetilde{D}}) \cdot
{(\omega\vert_{\widetilde{D}})}^{[n-1]}
\underalign{\text{(by $a\ge 0 \text{ \& } v>0$)}}{\ge} 0.
$$
Moreover, if equality holds, then $a=0$.
\end{proof}

\begin{proof}[Proof of Claim \ref{claim-projection-formula}]
Fix $m \in \N$. We first prove $(c_1^{\T}(\mathcal{E})\cdot \alpha_{\T}^m)_{\T}
=a (\alpha_{\T}^m\vert_{\widetilde{D}})_{\T}$. The following diagram is commutative:
\begin{equation}
\label{eq-commutative-1}
\xymatrix@C=40pt@R=20pt{
\CH^{\G}_{n-1}(\widetilde{D})_{\Q}\ar[d]^{i_{*}} 
&\CH_{\G}^{0}(\widetilde{D})_{\Q}\ar[d]^{i_{*}} 
\ar[l]_{\sim}^-{\cap [\widetilde{D}]_G}\ar@{}@<0ex>[dl]|{\circlearrowright}
 \\
\CH^{\G}_{n-1}(X)_{\Q}
&\CH_{\G}^{1}(X)_{\Q}
 \ar[l]_{\sim}^-{\cap [X]_G}
 \\
}
\end{equation}
Denote by $1_{\widetilde{D}} \in \CH_{\G}^{0}(\widetilde{D})_{\Q}$ the generator. Since $\widetilde{D}$ is irreducible, it follows from $c^{\G}_{1}(\mathcal{E}) \cap [X]_G = a [D]_{\G} \in \CH^{\G}_{n-1}(X)_{\Q}$ that
\begin{equation}
\label{eq-torsion-1}
c^{\G}_{1}(\mathcal{E}) 
=
i_{*}(a \cdot 1_{\widetilde{D}}) \in \CH_{\G}^{1}(X)_{\Q}.
\end{equation}
Moreover, by the naturality in Theorem \ref{thmdefn-chern-equivalence}, the following diagram is commutative:
\begin{equation}
\label{eq-commutative-2}
\xymatrix@C=40pt@R=20pt{
\CH_{\G}^{0}(\widetilde{D})_{\Q} \ar[d]^{i_{*}} \ar[r]^{\widetilde{\Psi}_D} 
&H^{0}_{\dR, \T}(\widetilde{D}) \ar[d]^{i_{*}} 
\ar@{}@<0ex>[dl]|{\circlearrowright}
 \\
\CH_{\G}^{1}(X)_{\Q}\ar[r]^{\widetilde{\Psi}_X} 
&H^{2}_{\dR, \T}(X) 
 \\
}
\end{equation}
Here the push-forward is defined because $\widetilde{D}$ and $X$ are smooth. Finally, by the naturality of the Cartan model, the following diagram is commutative:
\begin{equation}
\label{eq-commutative-3}
\xymatrix@C=40pt@R=20pt{
H^{2m}_{\dR, \T}(\widetilde{D}) \ar[d]^{i_{*}} 
\ar[r]^-{{\pi_{\widetilde{D}}}_{*}} 
&H^{2m-2n+2}_{\dR, \T}(\mathrm{pt}) 
\ar@{}@<-2ex>[dl]|{\circlearrowright}
 \\
H^{2m + 2}_{\dR, \T}(X) \ar[ru]_-{{\pi_X}_{*}} 
&
 \\
}
\end{equation}
Therefore we compute as follows:
\begin{align*}
(c_1^{\T}(\mathcal{E})\cdot \alpha_{\T}^m)_{\T}
&\underalign{\text{(Def. \ref{defn-intersection-converge} \& Thm. \ref{thmdefn-chern-equivalence})}}{=} 
{\pi_X}_{*}(\widetilde{\Psi}_X(c^{\G}_{1}(\mathcal{E}))\cdot \alpha_{\T}^{m}) 
\underalign{\eqref{eq-torsion-1}}{=} 
{\pi_X}_{*}(\widetilde{\Psi}_Xi_{*}(a \cdot 1_{\widetilde{D}}) \cdot \alpha_{\T}^{m})
 \\
&\underalign{\eqref{eq-commutative-2}}{=} 
{\pi_X}_{*}(i_{*}\widetilde{\Psi}_D(a \cdot 1_{\widetilde{D}}) \cdot \alpha_{\T}^{m})
\underalign{\text{(Projection formula)}}{=} 
{\pi_X}_{*}i_{*}(\widetilde{\Psi}_D(a \cdot 1_{\widetilde{D}}) \cdot i^{*}\alpha_{\T}^{m}) \\
&\underalign{\eqref{eq-commutative-3}}{=} 
{\pi_{\widetilde{D}}}_{*}(\widetilde{\Psi}_D(a \cdot 1_{\widetilde{D}}) \cdot i^{*}\alpha_{\T}^{m}) 
\underalign{ 
\text{(Def. \ref{defn-intersection-converge} \& Thm. \ref{thmdefn-chern-equivalence})}}{=} 
a (\alpha_{\T}^m\vert_{\widetilde{D}})_{\T}.
\end{align*}
%In the last equality, we use \(i^{*}\alpha_{\T}=\alpha_{\T}\vert_{\widetilde{D}}\). 
This proves the first assertion. Also we have
\begin{align*}
(c^{\T}_{1}(\mathcal{E}) \cdot e^{\alpha_{\T}})_{\T} 
\underalign{ \text{(Def. \ref{defn-intersection-converge})}}{=} 
\sum_{k=0}^{\infty} \frac{1}{k!} 
(c_1^{\T}(\mathcal{E})\cdot \alpha_{\T}^k)_{\T}
\underalign{ \text{}}{=} 
\sum_{k=0}^{\infty} \frac{1}{k!} 
a (\alpha_{\T}^k\vert_{\widetilde{D}})_{\T}
\underalign{ \text{(Def. \ref{defn-intersection-converge})}}{=} 
a(e^{\alpha_{\T}\vert_{\widetilde{D}}})_{\T}.
\end{align*}
Thus, after substituting $\sqrt{-1}\xi$, the latter assertion follows.
\end{proof}

\begin{rem}
\label{rem-intersection-minimum}
Under the setting of Proposition \ref{prop-semipositive-slope}, set
\(
m \coloneqq \min_{p\in P}v(p)>0.
\)
%Such a minimum exists because the moment polytope $P=\mu(X)$ is compact. 
Then
$$
(c_1^{\T}(\mathcal{O}_{X}(D))\cdot v(\alpha_{\T}))
%\underalign{\text{(Prop. \ref{prop-semipositive-slope})}}{=}(v(\alpha_{\T}\vert_{\widetilde{D}})) 
\underalign{\raisebox{-1.0ex}{\tiny\text{(Lem. \ref{lem-intersection-chern}  \& Prop. \ref{prop-semipositive-slope})}}}{=} 
\int_{\widetilde{D}} v(\mu \circ i \vert_{\widetilde{D}}) \cdot
{(\omega\vert_{\widetilde{D}})}^{[n-1]}
\ge 
m (c_1(\mathcal{O}_{X}(D)) \cdot \{ \omega\}^{[n-1]}).
$$
%Note that $c_1(aD)$ on the right-hand side denotes the ordinary first Chern class.
\end{rem}

\begin{cor}
\label{cor-semipositive-slope}
Assume Setup \ref{setup-soliton-situation}. 
Let $\mathcal{E}$ be a $\G$-equivariant coherent sheaf.
\begin{enumerate}
\item If $\mathcal{E}$ is a torsion sheaf, then $(c_1^{\T}(\mathcal{E})\cdot v(\alpha_{\T}))\ge 0$. Moreover, if equality holds, then $\Supp(\mathcal{E})$ has codimension at least $2$.
\item If $\mathcal{E}$ is a torsion-free sheaf and there exists an irreducible
$\G$-invariant effective divisor $D$  such that $\det \mathcal{E} \cong \mathcal{O}_X(D)$, then $(c_1^{\T}(\mathcal{E})\cdot v(\alpha_{\T}))\ge 0$. Moreover, if equality holds, then $D=0$.
\end{enumerate}
\end{cor}
\begin{proof}
For (1), let $D_1, \ldots, D_k$ be the irreducible $\G$-invariant effective divisors contained in $\Supp(\mathcal{E})$. By Lemma \ref{lem-torsion-eff-cycle}, there exist nonnegative rational numbers $a_1, \ldots, a_k \ge 0$ such that
$$
c^{\G}_{1}(\mathcal{E}) \cap [X]_{\G} =\sum_{i=1}^{k} a_i [D_i]_{\G} \in \CH^{\G}_{n-1}(X)_{\Q}.
$$
Since the assumptions of Proposition \ref{prop-semipositive-slope} are satisfied, the claim follows. 

The proof of (2) is similar. Note that
\(
c^{\G}_{1}(\mathcal{E}) = c^{\G}_{1}(\det \mathcal{E}) = c^{\G}_{1}(\mathcal{O}_X(D))
\) holds.
This follows from the definition \eqref{eq-Gequiv-K} when $\mathcal{E}$ is locally free, and in general it follows by taking a finite locally free resolution.
\end{proof}

A similar statement holds for higher Chern classes. 
Here, we state the result only for the second Chern character.
%Since we use only the second Chern character, we state and prove only the following form.
\begin{prop}
\label{prop-torsion-2ndchern}
Assume Setup \ref{setup-soliton-situation}. Let $\mathcal{E}$ be a $\G$-equivariant torsion coherent sheaf, and set $Z  \coloneqq  \Supp(\mathcal{E})$. If $Z$ has codimension at least $2$, then
$$
(\ch^{\T}_2(\mathcal{E})\cdot v(\alpha_{\T})) \ge 0.
$$
Moreover, if equality holds, then $Z$ has codimension at least $3$.
\end{prop}
\begin{proof}
Let $Z_1, \ldots, Z_k \subset Z$ be the $\G$-invariant irreducible components of $Z$ of dimension $n-2$. As in Lemma \ref{lem-torsion-eff-cycle}, set $a_i\coloneqq \length_{\mathcal{O}_{X,\eta_i}}(\mathcal{E}_{\eta_i})$. 
Then 
$$
\ch_{2}^{\G}(\mathcal{E}) \cap [X]_{\G} 
\quad
\underalign{\text{(Lem. \ref{lem-torsion-eff-cycle})}}{=}
\quad
\sum_{i=1}^{k}a_i [Z_i]_{\G} .
$$
Thus, by the same argument as in Proposition \ref{prop-semipositive-slope} and Corollary \ref{cor-semipositive-slope}, taking a $\G$-equivariant resolution $\widetilde{Z}_i \to Z_i$,
$$
(\ch^{\T}_2(\mathcal{E})\cdot v(\alpha_{\T}))
\underalign{\text{(cf. Prop. \ref{prop-semipositive-slope})}}{=}
\quad
\sum_{i=1}^{k}a_i (v(\alpha_{\T}\vert_{\widetilde{Z}_i}))
\underalign{\text{(Lem. \ref{lem-intersection-chern})}}{=} 
\quad
\sum_{i=1}^{k}a_i 
\int_{\widetilde{Z}_i}v(\mu\vert_{\widetilde{Z}_i}) 
\cdot
(\omega_{\widetilde{Z}_i})^{[n-2]}
\underalign{\text{(by $a_i\ge 0 \text{ \& } v>0$)}}{\ge} 0.
$$
Moreover, if equality holds, then $a_1 = \cdots =a_k=0$, and hence $Z$ has codimension at least $3$.
\end{proof}

\subsection{Weighted slope stability and basic properties}

Motivated by \cite[Definition 5.3]{HL24}, we define weighted slope stability. 
\begin{defn}[{cf. \cite[Definition 5.3]{HL24}}]
\label{defn-vslope-coh}
Assume Setup \ref{setup-soliton-situation}. Let $\mathcal{E}$ be a torsion-free $\G$-equivariant sheaf. We define the \emph{$v$-weighted slope} of $\mathcal{E}$ by
\[
\mu_v(\mathcal{E}) \coloneqq \frac{(c_1^{\T}(\mathcal{E})\cdot v(\alpha_{\T}))}{\rk(\mathcal{E})} \in \R.
\]

We say that $\mathcal{E}$ is \emph{$v$-stable} (resp. \emph{$v$-semistable}) if, for every $\G$-equivariant subsheaf $0 \subsetneq \mathcal{F} \subsetneq \mathcal{E}$ with $\rk(\mathcal{F}) \neq \rk(\mathcal{E})$, we have
$\mu_v(\mathcal{F}) < \mu_v(\mathcal{E})$ (resp. $\mu_v(\mathcal{F}) \le \mu_v(\mathcal{E})$).

\end{defn}
Thanks to Proposition \ref{prop-saturation} below, $\mathcal{E}$ is $v$-semistable if and only if $\mu_v(\mathcal{F}) \le \mu_v(\mathcal{E})$ holds for every $\G$-equivariant subsheaf $0 \subsetneq \mathcal{F} \subsetneq \mathcal{E}$; in other words, the rank condition is not needed.
\begin{rem}
\label{rem-equivalence-semistability}
We do not know whether the semistability condition in \cite[Definition 5.3]{HL24} is the same as our semistability condition. However, this issue does not affect the subsequent arguments.
\end{rem}

\begin{prop}[{cf. \cite[Proposition 2.12, Corollary 2.13]{GKP16a}}]
\label{prop-saturation}
Assume Setup \ref{setup-soliton-situation}. 
Let $\mathcal{E}$ be a torsion-free $\G$-equivariant sheaf and  $\mathcal{F} \subset \mathcal{E}$ be a $\G$-equivariant subsheaf.
\begin{enumerate}
\item If $\mathcal{F}$ and $\mathcal{E}$ have the same rank, then $\mu_v(\mathcal{F}) \le \mu_v(\mathcal{E})$.
\item If $\mathcal{F}^{sat}$ denotes the saturation of $\mathcal{F}$ in $\mathcal{E}$, then $\mu_v(\mathcal{F}) \le \mu_v(\mathcal{F}^{sat})$.
\end{enumerate}
\end{prop}
Here $\mathcal{F}^{sat} \coloneqq  \Ker{\mathcal{E} \to (\mathcal{E}/\mathcal{F})/\Tor}$ is the saturation of $\mathcal{F}$ in $\mathcal{E}$.
\begin{proof}
For (1), there exists a $\G$-invariant effective divisor $D$ such that $\det \mathcal{F}\cong \det \mathcal{E} \otimes \mathcal{O}_X(-D)$. Hence the assertion follows from Corollary  \ref{cor-semipositive-slope}. Assertion (2) follows immediately from (1).
\end{proof}

\begin{rem}
\label{rem-soliton-GKP16a}
Once Proposition \ref{prop-saturation} is established, by routine arguments, we also obtain the following properties. In the following statement, assume Setup \ref{setup-soliton-situation} and let $\mathcal{E}$ be any $\G$-equivariant torsion-free sheaf of positive rank $r$ on $X$.
\begin{enumerate}[label=$(\arabic*)$]
\item \label{prop-GKP16a-2.21}(cf. \cite[Definition 2.20, Proposition 2.21]{GKP16a})
Define the \emph{maximal slope} by
\[
\mu_{v}^{\max}(\mathcal{E})
 \coloneqq 
\sup\{\mu_{v}(\mathcal{F})\mid 0\neq \mathcal{F}\subseteq \mathcal{E}\ \text{is a $\G$-equivariant coherent subsheaf}\}.
\]
Then there exists a $\G$-equivariant ample line bundle $\mathcal{H}$ depending only on $\mathcal{E}$ such that
$$
\mu_{v}^{\max}(\mathcal{E}) \le r (c_1^{\T}(\mathcal{H})\cdot  v(\alpha_{\T})).
$$
%where $r \coloneqq \operatorname{rank}\,\mathcal{E}$. 
In particular, \(\mu_{v}^{\max}(\mathcal{E})<\infty\).
\item \label{prop-GKP16a-2.22}
(cf. \cite[Proposition 2.22, Corollary 2.24]{GKP16a})
%Let $\mathcal{E}$ be any $\G$-equivariant torsion-free coherent sheaf of positive rank on $X$. 
There exists a nonzero coherent subsheaf $\mathcal{F}\subseteq \mathcal{E}$ such that
\begin{itemize}
 \item the slope is maximal, that is,  $ \mu_{v}(\mathcal{F})=\mu_{v}^{\max}(\mathcal{E})$,
\item if $\mathcal{F}'\subseteq \mathcal{E}$ is any other subsheaf satisfying $\mu_{v}(\mathcal{F}')=\mu_{v}^{\max}(\mathcal{E})$, then $\mathcal{F}'\subseteq \mathcal{F}$.
\end{itemize}
This sheaf $\mathcal{F}$ is called the \emph{maximally destabilizing subsheaf}. It is  unique, semistable, and saturated in $\mathcal{E}$.
\item \label{prop-HN-existence} (cf. \cite[Corollary 2.26]{GKP16a}, \cite[Definition 1.3.2]{HL10})
%For any $\mathcal{E}$ be any $\G$-equivariant torsion-free coherent sheaf of positive rank on $X$. 
There exists a unique Harder--Narasimhan filtration, that is, a filtration
\[
0=\mathcal{E}_0\subsetneq \mathcal{E}_1\subsetneq \cdots \subsetneq \mathcal{E}_{s}=\mathcal{E}
\]
such that each quotient sheaf
\(
\mathcal{G}_i \coloneqq \mathcal{E}_i/\mathcal{E}_{i-1}
\)
is torsion-free and $v$-semistable, and the sequence of slopes $\mu_{v}(\mathcal{G}_i)$ is strictly decreasing. 
Moreover, $\mu_{v}^{\max}(\mathcal{E})=\mu_{v}(\mathcal{E}_1)$ and $\mu_{v}^{\min}(\mathcal{E})=\mu_{v}(\mathcal{G}_{s})$, where the \emph{minimum slope} is defined by
$$
\mu_{v}^{\min}(\mathcal{E}) \coloneqq 
\inf\{\mu_{v}(\mathcal{Q})\mid \mathcal{E}\twoheadrightarrow \mathcal{Q}\text{ nonzero torsion-free $\G$-equivariant quotient sheaf}\}.
$$
\item \label{prop-JH-existence} (cf. \cite[Corollary 2.27]{GKP16a})
 If $\mathcal{E}$ is $v$-semistable, then there exists a Jordan--Hölder filtration, that is, a filtration
\[
0=\mathcal{E}_0\subsetneq \mathcal{E}_1\subsetneq \cdots \subsetneq \mathcal{E}_{s}=\mathcal{E}
\]
such that each quotient sheaf
\(
\mathcal{G}_i \coloneqq \mathcal{E}_i/\mathcal{E}_{i-1}
\)
is torsion-free and $v$-stable, and satisfies
\(
\mu_{v}(\mathcal{G}_i)=\mu_{v}(\mathcal{E}).
\)
\item \label{prop-min-max-reflexive} (cf. \cite[Proposition 1.13.]{Cla17})
$ \mu_{v}^{\max}(\mathcal{E})=-\mu_{v}^{\min}(\mathcal{E}^{\vee})$ and $ \mu_{v}^{\max}(\mathcal{E})=\mu_{v}^{\max}(\mathcal{E}^{\vee\vee})$.
\end{enumerate}
Most of the arguments are taken almost directly from \cite{Kob14, GKP16a}, so we omit the proof. 
\end{rem}

\subsection{Openness of weighted slope stability}
We prove the following Grothendieck lemma.
\begin{thm}[{cf. \cite[Theorem 2.29]{GKP16a}}]
\label{thm-GKP16a-2.29}
Assume Setup \ref{setup-soliton-situation}. Let $\mathcal{E}$ be a $\G$-equivariant torsion-free sheaf on $X$, and let $c\in\mathbb{R}$. Then
\[
S \coloneqq \{\mu_{v}(\mathcal{F})\mid \mathcal{F}\subseteq\mathcal{E}\ \text{is a $\G$-equivariant subsheaf of positive rank with }\mu_{v}(\mathcal{F})\geq c\}\subseteq\R
\]
is a finite set.
\end{thm}

\begin{proof}
We write
\(
\deg_v(L) \coloneqq (c_1^\T(L)\cdot v(\alpha_\T))
\)
for any \(\G\)-equivariant line bundle \(L\). Take a \(\G\)-linearized very ample line bundle \(H\) such that \(\mathcal E^\vee\otimes H\) is globally generated, and set
\(
V \coloneqq H^0(X,\mathcal E^\vee\otimes H)^\vee,
\)
which is a finite-dimensional \(\G\)-module. By considering the \(\G\)-equivariant injection
\(
\mathcal E\otimes H^{-1}
\hookrightarrow
V\otimes\mathcal O_X,
\)
we may assume, after replacing \(c\), that \(\mathcal E=V\otimes\mathcal O_X\).

Let
\(
\mathcal F\subseteq V\otimes\mathcal O_X
\)
be a \(\G\)-equivariant subsheaf of positive rank \(s\). Then there exist a weight \(\chi\) of \(\bigwedge^sV\) and a nonzero \(\G\)-equivariant morphism
\[
\det\mathcal F
\longrightarrow
\mathbb C_\chi\otimes\mathcal O_X,
\]
where \(\mathbb C_\chi\) is the one-dimensional \(\G\)-module of weight \(\chi\). Thus, there exists a \(\G\)-invariant effective divisor \(D\) such that
\begin{equation}
\label{eq-Gisom-DF}
\det\mathcal F
\simeq
\mathcal O_X(-D)\otimes\mathbb C_\chi
\end{equation}
is a \(\G\)-equivariant isomorphism. Therefore, the assumption \(\mu_v(\mathcal F)\geq c\) gives
\[
0
\underalign{\text{(Cor. \ref{cor-semipositive-slope})}}{\leq}
\deg_v(\mathcal O_X(D))
\underalign{\eqref{eq-Gisom-DF}}{=}
-s\mu_v(\mathcal F)
+
\deg_v(\mathbb C_\chi\otimes\mathcal O_X)
\leq
-sc
+
\deg_v(\mathbb C_\chi\otimes\mathcal O_X).
\]
Since \(s\) and \(\chi\) appearing on the right-hand side take only finitely many values, it is enough to show that, for any fixed \(B>0\), the quantity \(\deg_v(\mathcal O_X(D))\) takes only finitely many values as \(D\) ranges over the \(\G\)-invariant effective divisors satisfying
\begin{equation}
\label{eq-bounded-weighted-degree}
0\leq\deg_v(\mathcal O_X(D))\leq B.
\end{equation}

From now on, we consider divisors \(D\) satisfying \eqref{eq-bounded-weighted-degree}. Set
\(
m \coloneqq \min_{q\in P}v(q)>0.
\)
Then Remark~\ref{rem-intersection-minimum} implies that
\(
D\cdot\omega^{[n-1]}
\leq
\frac{B}{m}.
\)
Thus, the numerical classes of such divisors \(D\) form a finite set; cf.\ \cite[Theorem 2.29]{GKP16a}. In particular, there exists an integer \(K>0\) such that
\begin{equation}
\label{eq-DH-negativeK}
D\cdot H^{n-1}\leq K
\end{equation}
for every divisor \(D\) satisfying \eqref{eq-bounded-weighted-degree}.

By the Borel fixed point theorem, we can take a fixed point \(p\in X^\G\).
%Then
%\begin{equation*}
%T_p^\vee X
%\simeq
%\bigoplus_{i=1}^n\mathbb C_{\nu_i}.
%\end{equation*}
Then \(\operatorname{mult}_p(D)\leq K\) holds.
Indeed, let
\(
\rho\colon\widetilde X=\operatorname{Bl}_pX\to X
\)
be the blow-up of \(X\) at \(p\), and let \(E_p\) be the exceptional divisor. Since \(H\) is very ample, \(\rho^*H-E_p\) is nef. The effectiveness of the strict transform of \(D\) therefore gives
\begin{align*}
0\leq
\bigl(\rho^*D-\operatorname{mult}_p(D)E_p\bigr)
\cdot(\rho^*H-E_p)^{n-1}
=
D\cdot H^{n-1}
-\operatorname{mult}_p(D)
\underalign{\eqref{eq-DH-negativeK}}{\le} K -\operatorname{mult}_p(D).
\end{align*}
%Hence, \(\operatorname{mult}_p(D)\leq K\).

Let
\(
\nu_1,\ldots,\nu_n
\)
be the weights of the natural isotropy action on the cotangent space \(T_p^\vee X\), and let \(\mu\) be the weight of
\(
\mathcal O_X(D)_p\otimes_{\mathcal O_{X,p}}\kappa(p).
\)
We will show that there exist \(a_i\in\mathbb Z_{\geq0}\) such that
\begin{equation}
\label{eq-mu-ai}
\mu=-\sum_{i=1}^na_i\nu_i
\quad \text{and} \quad
\sum_{i=1}^na_i\leq K.
\end{equation}
In particular, this implies that \(\mu\) takes only finitely many values. Choose a \(\G\)-stable affine open subset \(U\) containing \(p\), and \(\G\)-equivariant coordinates \(z_1,\ldots,z_n\) such that
\(
g\cdot z_i
=
\nu_i(g)z_i.
\)
Let \(f\) be a local equation of \(D\) at \(p\).
Then, by the argument in \cite[Proposition 3.17]{Mul25}, we can write
\[
f=\sum_{J\in\mathbb Z_{\geq0}^n}c_Jz^J \in \widehat{\mathcal O_{X,p}}\cong\C[[z_1,\ldots,z_n]],
\qquad
\text{where }
z^J \coloneqq z_1^{J_1}\cdots z_n^{J_n}.
\]
Take \((a_1,\ldots,a_n) \in\mathbb Z_{\geq0}^n\) with 
\(
\operatorname{mult}_pD
=
\sum_{i=1}^na_i
\)
and \(c_{(a_1,\ldots,a_n)}\neq0\). The argument in \cite[Proposition 3.17]{Mul25} gives
\[
\mu=-a_1\nu_1-\cdots-a_n\nu_n.
\]
Hence, \eqref{eq-mu-ai} follows from
\(
\operatorname{mult}_pD\leq K.
\)

Summarizing the discussion so far,
\begin{equation}
\label{eq-finiteness-weightc1}
\left\{
\left(c_1(\mathcal O_X(D)),i_p^*c_1^\T(\mathcal O_X(D))\right)
\mid
\text{$D$ is a \(\G\)-invariant effective divisor with }0\leq\deg_v(\mathcal O_X(D))\leq B
\right\}
\end{equation}
is a finite subset of
\(
H^2(X,\mathbb R)\oplus H_\T^2(\mathrm{pt},\mathbb R),
\)
where \(i_p\colon\{p\}\hookrightarrow X\) denotes the inclusion. Here, we identify \(H_\T^2(X,\mathbb R)\) with \(H_{\dR,\T}^2(X,\mathbb R)\) via Theorem \ref{thm-equiv-caltan-isom}.

By equivariant formality of the Hamiltonian \(\T\)-action (cf.\ \cite[Proposition 2.1]{GM10}), the sequence
\begin{equation*}
0
\longrightarrow
H_\T^2(\mathrm{pt},\mathbb R)
\overset{\pi^*}{\longrightarrow}
H_\T^2(X,\mathbb R)
\overset{\mathrm{forget}}{\longrightarrow}
H^2(X,\mathbb R)
\longrightarrow
0
\end{equation*}
is exact. Here,
\(
\pi\colon E\T\times_\T X\to B\T
\)
is the morphism appearing in Theorem \ref{thm-equiv-caltan-isom}. Since \(p\) is a fixed point, it induces a section of \(\pi\). Therefore,
\begin{equation*}
(\mathrm{forget},i_p^*)\colon
H_\T^2(X,\mathbb R)
\longrightarrow
H^2(X,\mathbb R)\oplus H_\T^2(\mathrm{pt},\mathbb R)
\end{equation*}
is injective. Thus, by \eqref{eq-finiteness-weightc1},
\(
c_1^\T(\mathcal O_X(D))\in H_\T^2(X,\mathbb R)
\)
takes only finitely many values, and hence so does
\(
\deg_v(\mathcal O_X(D)).
\)
\end{proof}

\begin{lem}
\label{lem-stable-uniform}
Assume Setup \ref{setup-soliton-situation}. 
Let $\mathcal{E}$ be a $\G$-equivariant torsion-free sheaf. If $\mathcal{E}$ is $v$-stable, then there exists $\varepsilon>0$ such that
$$
\mu_v(\mathcal{F})  < \mu_v(\mathcal{E}) - \varepsilon
$$
holds for every $\G$-equivariant subsheaf $0 \subsetneq \mathcal{F} \subsetneq \mathcal{E}$ with $\rk(\mathcal{F}) \neq \rk(\mathcal{E})$.
\end{lem}
\begin{proof}
By Theorem \ref{thm-GKP16a-2.29},
$$
S \coloneqq 
\{\mu_{v}(\mathcal{F}) \mid 0 \subsetneq \mathcal{F}\subsetneq \mathcal{E}\ \text{ with } \rk(\mathcal{F}) \neq \rk(\mathcal{E}) \text{ and }\mu_{v}(\mathcal{F})\geq \mu_v(\mathcal{E}) -1\}
\subseteq \R
$$
is finite. 
Since \(\mathcal{E}\) is \(v\)-stable, every element of \(S\) is strictly less than
\(\mu_v(\mathcal{E})\). Hence, we can choose an \(\varepsilon>0\) satisfying the desired inequality.
%Hence it suffices to choose the positive number $\varepsilon>0$ defined by
%$$
%\varepsilon  \coloneqq \frac{1}{2}\left(\mu_v(\mathcal{E}) -\max_{ \mu_{v}(\mathcal{F}) \in S}\{\mu_{v}(\mathcal{F}), \mu_v(\mathcal{E}) -1\}\right)
%$$
\end{proof}

\begin{lem}[{Openness of stability}]
\label{lem-openness-stab}
Assume Setup \ref{setup-soliton-situation}. 
Let \(v_m\colon P\to\R_{>0}\) be smooth positive functions converging to \(v\) in \(C^1(P)\) and satisfying \(v_m>v\) on \(P\) for every \(m\).
If $\mathcal{E}$ is a $v$-stable $\G$-equivariant torsion-free sheaf, then there exists $M \in \N$ such that  $\mathcal{E}$ is $v_m$-stable for every $m > M$. 
\end{lem}
\begin{proof}
Let $\varepsilon>0$ be as in Lemma \ref{lem-stable-uniform}, and set $u_m\coloneqq v_m-v>0$.
Choose a $\G$-linearized very ample line bundle $H$ and a finite-dimensional $\G$-module $V$ such that $\mathcal E\hookrightarrow H\otimes V$, as in the proof of Theorem \ref{thm-GKP16a-2.29}.
For every equivariant subsheaf $\mathcal F\subset\mathcal E$ of rank $q$, taking determinants gives a nonzero equivariant map $\det\mathcal F\to H^{\otimes q}\otimes\C_\chi$ for some weight $\chi$ of $\bigwedge^qV$.
The zero divisor of this map is effective, so Proposition \ref{prop-semipositive-slope} gives
\[
\mu_{u_m}(\mathcal F)\leq(c_1^\T(H)\cdot u_m(\alpha_\T))+
\frac1q(c_1^\T(\C_\chi\otimes\mathcal O_X)\cdot u_m(\alpha_\T)).
\]
There are only finitely many possible pairs $(q,\chi)$ as in Theorem \ref{thm-GKP16a-2.29}, and each expression on the right tends to zero by Lemma \ref{lem-intersection-chern}.
Since $\mu_{u_m}(\mathcal E)\to0$ as well, for all sufficiently large $m$ we have
\begin{equation}
\label{eq-u_m-ineq}
\mu_{u_m}(\mathcal F)<\frac{\varepsilon}{3}\quad\text{for every such $\mathcal F$},
\qquad |\mu_{u_m}(\mathcal E)|<\frac{\varepsilon}{3}.
\end{equation}
Thus, for every equivariant subsheaf $\mathcal F$ with $0<\rk\mathcal F<\rk\mathcal E$,
\[
\mu_{v_m}(\mathcal F)-\mu_{v_m}(\mathcal E)
\underalign{\text{(Lem. \ref{lem-stable-uniform} \& \eqref{eq-u_m-ineq})}}{<}
-\varepsilon+\frac{\varepsilon}{3}+\frac{\varepsilon}{3}<0.
\]
This proves the assertion.
\end{proof}

\section{Bogomolov--Gieseker inequality with polynomial weight}
\label{sec-BG-polynomial-weight}
We establish the Bogomolov--Gieseker inequality for $\G$-equivariant torsion-free sheaves in the case of the polynomial weight $v(\mu)=(c + b \mu_{\xi})^m$ by using the semi-simple fibration constructed in \cite[Section 5]{AJL23}  and \cite[Subsection 3.6]{HL24}. 

\subsection{Hodge index theorem and Bogomolov--Gieseker inequality}
In the next subsection \ref{subsec-proofofthm-degree}, we will prove the following theorem.

\begin{thm}
\label{thm-degree-maxslope-summary}
Under Setup \ref{setup-soliton-situation}, assume that there exists an integer lattice point $\mathbf{p} \in \mathfrak{t}$,  $c \in \R$, and  $m \in \N$ such that
$$
v(\mu)=\left(c + \langle \mu, \mathbf{p}\rangle\right)^{m} \colon X \to \R_{>0}
\quad 
\text{and}
\quad 
c + \langle \mu,\mathbf{p}\rangle>0 \text{ on $X$. }
$$
Then there exist a principal $\G$-bundle $P\to \C\mathbb{P}^m$ and a Kähler form $\omega_P$ on the smooth projective variety $X_P=P\times_{\G}X$ satisfying the following properties.
\begin{enumerate}
\item $(=$Lemma \ref{lem-degree-formula}$)$ For any $\G$-equivariant torsion-free sheaf $\mathcal{E}$ on $X$, there exists a torsion-free sheaf $\mathcal{E}_P$ on $X_P$ such that the following equalities hold:
\[
\{ \omega_P\}^{[n+m]}
=V \cdot ( v(\alpha_{\T})), 
\quad
c_1(\mathcal{E}_P) \cdot \{ \omega_P\}^{[n+m-1]}
=V \cdot (c_{1}^{\T}(\mathcal{E}) \cdot v(\alpha_{\T})), 
\]
\[
\ch_2(\mathcal{E}_P) \cdot \{ \omega_P \}^{[n+m-2]}
=V \cdot (\ch_{2}^{\T}(\mathcal{E}) \cdot v(\alpha_{\T})), 
\quad
 c_1(\mathcal{E}_P)^2 \cdot \{ \omega_P\}^{[n+m-2]}
=V \cdot (c_{1}^{\T}(\mathcal{E})^2 \cdot v(\alpha_{\T})).
\]
Here $V \coloneqq \frac{(2\pi)^m}{m!}$ and $\omega_P^{[k]} \coloneqq \frac{\omega_P^{k}}{k!}$ for any $k \in \N$.
\item $(=$Lemma \ref{lem-max-slope-comparison}$)$
For  the usual maximal slope $\mu^{\max}_{\omega_{P}^{[n+m-1]}}(\mathcal{E}_P)$ of $\mathcal{E}_P$ with respect to $\omega_{P}^{[n+m-1]}$, the following equality holds:
\[
\mu^{\max}_{\omega_{P}^{[n+m-1]}}(\mathcal{E}_P)
=
V\cdot \mu^{\max}_{v}(\mathcal{E}).
\]
\end{enumerate}
\end{thm}
Before proving Theorem \ref{thm-degree-maxslope-summary}, we first establish the following consequences, assuming Theorem \ref{thm-degree-maxslope-summary}.

\begin{thm}[Equivariant Hodge index theorem for polynomial weights]
\label{thm-equiv-Hodge-index}
Assume Setup \ref{setup-soliton-situation}. Let $\mathcal{E}$ be a $\G$-equivariant torsion-free sheaf, and assume that there exist $\xi\in\mathfrak{t}$, $c, b \in \R$, and $m \in \N$ such that
$v(\mu)= (c  + b \mu_\xi)^{m}$ and $c  + b \mu_\xi>0$ on $X$.
Then
\begin{equation}
\label{eq-equiv-Hodge-index}
(c_1^{\T}(\mathcal{E})^2 \cdot v(\alpha_{\T}))
\cdot 
( v(\alpha_{\T}))
\le
\frac{(n+m-1)}{n+m } (c_1^{\T}(\mathcal{E}) \cdot v(\alpha_{\T}))^2.
\end{equation}
\end{thm}

\begin{proof}
We may assume that $\xi \in \mathfrak{t}$ is a lattice point. Indeed, if $\xi$ is rational, this follows after multiplying $\xi$ by a positive integer, and the general case follows by approximating $\xi$ by rational points.
Similarly, we may assume that $b=1$. 
Using the notation of Theorem \ref{thm-degree-maxslope-summary} and the usual Hodge index theorem, we deduce that
\begin{align*}
&(c_1^{\T}(\mathcal{E})^2 \cdot v(\alpha_{\T}))\cdot (v(\alpha_{\T}))
\underalign{\text{(Thm. \ref{thm-degree-maxslope-summary})}}{=} 
\frac{1}{V^2} \cdot (c_1(\mathcal{E}_P)^2 \cdot \{ \omega_P \}^{[n+m-2]}) \cdot  ( \{ \omega_P \}^{[n+m]}) \\
&\underalign{\text{(by $\omega^{[n]} \coloneqq \frac{\omega^{n}}{n!}$)}}{=}  \quad
\frac{1}{V^2 (n+m-2)! (n+m)!}(c_1(\mathcal{E}_P)^2 \cdot \{ \omega_P\}^{n+m-2}) \cdot  ( \{ \omega_P\}^{n+m}) 
\\
&\underalign{\raisebox{-0.7ex}{\tiny \text{(by Hodge index Thm)}}}{\le} \quad
\frac{1}{V^2(n+m-2)! (n+m)!}(c_1(\mathcal{E}_P) \cdot \{ \omega_P\}^{n+m-1})^2
%\\
%&\underalign{}{=}
%\frac{(n+m-1)V^2}{n+m }(c_1(\mathcal{E}_P) \cdot \{ \omega_P\}^{[n+m-1]})^2
\underalign{\text{(Thm. \ref{thm-degree-maxslope-summary})}}{=} \frac{(n+m-1)}{n+m } (c_1^{\T}(\mathcal{E}) \cdot v(\alpha_{\T}))^2.
\end{align*}
\end{proof}

\begin{thm}[Bogomolov--Gieseker inequality with polynomial weight]
\label{thm-HL24-6.1-torsionfree}
Assume Setup \ref{setup-soliton-situation}. Let $\mathcal{E}$ be a $\G$-equivariant torsion-free sheaf of rank $r$, and assume that there exist $\xi\in\mathfrak{t}$, $c, b \in \R$, and $m \in \N$ such that
$v(\mu)= (c  + b \mu_\xi)^{m}$ and $(c  + b \mu_\xi)>0$ on $X$.
If $\mathcal{E}$ is $v$-semistable, then the Bogomolov--Gieseker inequality holds:
\[
\left( (2r c^{\T}_2(\mathcal{E}) - (r-1) c^{\T}_1(\mathcal{E})^2 )\cdot v(\alpha_{\T}) \right)
=  \left( (\ch^{\T}_{1}(\mathcal{E})^{2} - 2r \ch^{\T}_{2}(\mathcal{E}))\cdot v(\alpha_{\T}) \right)
\ge 0
\]
\end{thm}
\begin{proof}
We continue to use the notation of Theorem \ref{thm-degree-maxslope-summary}. The argument is different depending on whether $\xi$ is rational.

(1). Suppose that $\xi$ is rational. As in Theorem \ref{thm-equiv-Hodge-index}, after multiplying $\xi$ by a constant, we may assume that $\xi \in \mathfrak{t}$ is a lattice point and that $b=1$. Hence, by Theorem \ref{thm-degree-maxslope-summary}, $\mathcal{E}_P$ is $\omega_P^{[n+m-1]}$-semistable. Therefore, by the usual Bogomolov--Gieseker inequality for a torsion-free sheaf (cf. \cite[Corollary 4.7]{Miy87}), 
$$
\left((\ch_{1}^{\T}(\mathcal{E})^{2} - 2r \ch_{2}^{\T}(\mathcal{E}) )\cdot v(\alpha_{\T})\right)
\underalign{\text{(Thm. \ref{thm-degree-maxslope-summary})}}{=}
\frac{1}{V} \left(\ch_{1}(\mathcal{E}_P)^{2} - 2r \ch_{2}(\mathcal{E}_P) \right)\cdot \{ \omega_P\}^{[n+m-2]}
\underalign{\text{(by usual BG ineq.)}}{\ge}0.
$$

(2). Suppose that $\xi$ is not rational. First assume that $\mathcal{E}$ is $v$-stable. 
In this case, there exist a sequence $\{\xi_k\}_{k=1}^{\infty}$ of rational points in $\mathfrak{t}$ and sequences $\{c_k\}_{k=1}^{\infty}, \{b_k\}_{k=1}^{\infty} \subset \R$ such that
$$
v_k(\mu) \coloneqq \left(c_k+b_k\mu_{\xi_k}\right)^m
$$
satisfies \(v_k>v\) and \(v_k\to v\) in \(C^2(P)\).
 Lemma \ref{lem-openness-stab} implies that $\mathcal{E}$ is $v_k$-stable for all sufficiently large $k$. 
 Hence, by case (1), the inequality
$$
\left((2rc^{\T}_2(\mathcal{E}) - (r-1)c^{\T}_1(\mathcal{E})^2 ) \cdot v_k(\alpha_{\T}) \right)\ge0
$$
holds. Since \(v_k \to v\) in \(C^2(P)\), letting \(k\to\infty\) yields the desired inequality.

It remains to consider the case where $\mathcal{E}$ is $v$-semistable. By Remark \ref{rem-soliton-GKP16a} \ref{prop-JH-existence}, we take a Jordan--Hölder filtration. Namely, there exists a filtration
\[
0=\mathcal{E}_0\subsetneq \mathcal{E}_1\subsetneq \cdots \subsetneq \mathcal{E}_{s}=\mathcal{E}
\]
such that each quotient sheaf
\(
\mathcal{G}_i \coloneqq \mathcal{E}_i/\mathcal{E}_{i-1}
\)
is a torsion-free $v$-stable sheaf of rank $r_i$, and satisfies
\(
\mu_{v}(\mathcal{G}_i)=\mu_{v}(\mathcal{E})
\).
Moreover, by applying Theorem \ref{thm-equiv-Hodge-index}
to $\det \mathcal{G}_{i}^{\otimes r_j}\otimes (\det \mathcal{G}_{j}^{\vee})^{\otimes r_i}$ and dividing by $r_i^2 r_j^2$, 
\begin{equation}
\label{eq-JH-Hodge-index}
\left(\left( \frac{c^{\T}_1(\mathcal{G}_{i})}{r_i} - \frac{c^{\T}_1(\mathcal{G}_{j})}{r_j} \right)^2  \cdot  v(\alpha_{\T}) \right)
 \le \frac{1}{ v(\alpha_{\T})}
 \underbrace{ \left( \left( \frac{c^{\T}_1(\mathcal{G}_{i})}{r_i} - \frac{c^{\T}_1(\mathcal{G}_{j})}{r_j} \right) \cdot  v(\alpha_{\T}) \right)^2 }_{=(\mu_{v}(\mathcal{G}_i)-\mu_{v}(\mathcal{G}_j))^2}
 =0.
\end{equation}
Hence we obtain
\begin{equation}
\label{eq-semistable-stable-red}
\begin{aligned}
 &\left((2rc^{\T}_2(\mathcal{E}) - (r-1)c^{\T}_1(\mathcal{E})^2 )\cdot v(\alpha_{\T})\right) \underalign{\text{\eqref{eq-Bogomolov-T-discriminant}}}{=}(\widetilde{\Psi}(\Delta^{\G}(\mathcal{E})) \cdot v(\alpha_{\T}) ) \\
 &\underalign{\text{(Lem. \ref{lem-Langer-ineq-chow})}}{=}
\left( \widetilde{\Psi}\left( 
r\sum_{i=1}^{s} \frac{\Delta^{\G}(\mathcal{G}_{i})}{r_i} 
 - \sum_{1 \le i < j \le s} r_i r_j \left( \frac{c^{\G}_1(\mathcal{G}_{i})}{r_i} - \frac{c^{\G}_1(\mathcal{G}_{j})}{r_j} \right)^2 
 \right) \cdot  v(\alpha_{\T}) \right)\\
  &\underalign{\text{(Def. \ref{defn-weighted-intersection-coherent})}}{=}
\sum_{i=1}^{s}\frac{r}{r_i}  
\underbrace{ \left( (2r_i c^{\T}_2(\mathcal{G}_{i}) - (r_i-1)c^{\T}_1(\mathcal{G}_{i})^2 )
\cdot v(\alpha_{\T}) \right)}_{\ge 0 \text{ by the stable case}}
+ \sum_{1 \le i < j \le s} 
   \underbrace{ - r_i r_j \left(\left( \frac{c^{\T}_1(\mathcal{G}_{i})}{r_i} - \frac{c^{\T}_1(\mathcal{G}_{j})}{r_j} \right)^2  \cdot  v(\alpha_{\T}) \right)}_{\ge 0 \text{ by \eqref{eq-JH-Hodge-index}}}\\
   & \underalign{ }{\ge} 0.
\end{aligned}
\end{equation}
\end{proof}

\subsection{Proof of Theorem \ref{thm-degree-maxslope-summary}}
\label{subsec-proofofthm-degree}
We now prove Theorem \ref{thm-degree-maxslope-summary}. The proof uses the semi-simple fibration $X_Q \to B$ constructed in \cite[Section 5]{AJL23} and \cite[Subsection 3.6]{HL24}. In these papers, $X_Q=Q \times_{\T} X$ is constructed by using a principal $\T$-bundle $Q \to B$. However, since $Q$ is a real manifold, this construction is not directly compatible with the methods of complex geometry and algebraic geometry. We therefore construct a complex manifold $P$ with the structure of a principal $\G$-bundle $P \to B$ such that $X_P=P \times_{\G} X$ is biholomorphic to $X_Q$, and  establish the complex analogue of the constructions in \cite[Section 5]{AJL23} and \cite[Subsection 3.6]{HL24}.  

Throughout Subsection \ref{subsec-proofofthm-degree}, we always assume Setup \ref{setup-soliton-situation}. 
% and also use some facts from \cite[Subsection 3.6]{HL24}.
\subsubsection{Construction of $P$ and $Q$}
Fix $m \in \N$. 
Choose a Kähler form $\omega_{\C\mathbb{P}^m}$ on $\C\mathbb{P}^m$ such that
$\frac{1}{2 \pi}\{ \omega_{\C\mathbb{P}^m} \}\in H^2(\C\mathbb{P}^m, \Z)$ and
$\int_{\C\mathbb{P}^m} \left(\frac{1}{2 \pi}\omega_{\C\mathbb{P}^m}\right)^{m}=1$. This normalization gives
$$
\int_{\C\mathbb{P}^m} (\omega_{\C\mathbb{P}^m})^{[m]}
=\int_{\C\mathbb{P}^m} \frac{\omega_{\C\mathbb{P}^m}^{m}}{m!}=\frac{(2\pi)^m}{m!}%\underalign{\text{(Thm. \ref{thm-degree-maxslope-summary})}}{=}
\eqqcolon V.
$$
In the notation of \cite[Subsection 3.6]{HL24}, we take $N=1$, $\alpha=1$, and $B=\C\mathbb{P}^m$. 
%$q_\alpha=\id_{B}$,

Let $\mathbf{p}\in \Lambda\subset \mathfrak{t}$, where $\Lambda\subset \mathfrak{t}$ denotes the integer lattice of $\mathfrak{t}$. Take a basis $(\xi_a)_{a=1}^{l}\subset \mathfrak{t}$ consisting of the $S^1$-generators of $\mathfrak{t}$. With this convention, $\Lambda = \bigoplus_{a=1}^{l}\Z \xi_a$. Moreover, if the map $\mathfrak{t} \to \T$ is defined by $\xi_a\mapsto e^{2\pi \sqrt{-1} \xi_a}$, then $\T \cong \mathfrak{t}/\Lambda$.
 Let $c\in\mathbb{R}$ be a real number such that
$\langle \mu,\mathbf{p}\rangle+c>0$
on $X$. Set
\[
f(\mu) \coloneqq c + \langle \mu,\mathbf{p}\rangle\colon X\to\mathbb{R}_{>0}
\quad 
\text{and}
\quad
v(\mu) \coloneqq (c + \langle \mu,\mathbf{p}\rangle)^m=f(\mu)^m
\]
Notice that $v'(\mu)=m f(\mu)^{m-1}$.

\begin{lem}
\label{lem-existence-Q}
There exist a principal $\G$-bundle $\pi_{P}\colon P\to \C\mathbb{P}^m$ and a principal $\mathbb{T}$-bundle $\pi_{Q}\colon Q\to \C\mathbb{P}^m$ satisfying the following properties.
\begin{enumerate}
\item 
$
Q \subset P$ and $
P \underset{C^\infty}{\cong} Q \times_{\mathbb{T}}\G
$.
In other words, $Q$ is the reduction to $\mathbb{T}$ of the principal bundle with structure group $\G$.
\item 
There exists a $\mathfrak t_\C$-valued $1$-form $\eta^{\C}\in \Omega^1(P,\mathfrak t_\C)$ on $P$ such that 
\begin{equation}
\label{eq-deta-property}
d\eta^{\C}
=
\pi_P^*\omega_{\C\mathbb{P}^m}
\otimes \mathbf{p}
\in  \Omega^2(P,\mathfrak{t}).
\end{equation}
\end{enumerate}
\end{lem}

\begin{proof}
We recall the constructions in \cite[Section 5]{AJL23} and explain how they can be written in the case of $\C\mathbb{P}^m$. We identify $\T=\mathfrak{t}/\Lambda$.
Since $\mathbf{p} \in \Lambda=\bigoplus_{a=1}^{l}\Z \xi_a$, there exist
$\mathbf{p}_{a}\in \Z$ such that
$
\mathbf{p}
=\sum_{a=1}^{l} \mathbf{p}_{a} \xi_a
$.
Set
$$
\gamma_a  \coloneqq  
-\mathbf{p}_{a}  
\left( \frac{1}{2 \pi}[\omega_{\C\mathbb{P}^m}]\right)
\in H^2(\C\mathbb{P}^m, \Z), 
\quad 
L_a  \coloneqq  \mathcal{O}_{\C\mathbb{P}^m}(\mathbf{p}_{a})
\overset{\tau_a}{\longrightarrow} \C\mathbb{P}^m
\quad 
a=1, \ldots,l .
$$
Then the dual bundle $L_{a}^{\vee}$ satisfies
$c_1(L_{a}^{\vee})=
\gamma_a 
$.
Let $h_{\mathcal{O}_{\C\mathbb{P}^m}(1)}$ be a positive metric on
$\mathcal{O}_{\C\mathbb{P}^m}(1)$, and set
$h_a = h_{\mathcal{O}_{\C\mathbb{P}^m}(1)}^{\otimes \mathbf{p}_{a}}$.
Then $h_a^{\vee}$ is a metric on $L_{a}^{\vee}$ such that
\begin{equation}
\label{eq-def-h_a}
\frac{\sqrt{-1}}{2 \pi}F_{h_a^{\vee}}
=
-\mathbf{p}_{a} 
\left( \frac{1}{2 \pi}\omega_{\C\mathbb{P}^m}\right)
\in c_1(L_{a}^{\vee}).
\end{equation}
Set
$$
 L_{a}^{\vee, \times} \coloneqq  L_{a}^{\vee}\setminus \{0\}
 \quad \text{and} \quad 
Q_a  \coloneqq  \{ v \in L_{a}^{\vee} \mid \lvert v\rvert_{h_a^{\vee}}=1 \}
\quad 
a=1, \ldots,l .
%= \{ v \in \mathcal{O}_{\C\mathbb{P}^m}(-\mathbf{p}_{a}) \mid |v|_{h_{\mathrm{FS}}^{\otimes - \mathbf{p}_{a}}}=1 \}
$$
Then $L_{a}^{\vee, \times} \to \C\mathbb{P}^m$ is a principal $\C^{*}$-bundle, and $Q_a \to \C\mathbb{P}^m$ is a principal $S^1$-bundle. 
Thus we define
$$
P \coloneqq L_{1}^{\vee, \times} \times_{\C\mathbb{P}^m}\cdots \times_{\C\mathbb{P}^m} 
L_{l}^{\vee, \times}
\quad\text{and} \quad
%\pi_P : P \to \C\mathbb{P}^m
%\qquad
Q=Q_1 \times_{\C\mathbb{P}^m} \cdots \times_{\C\mathbb{P}^m} Q_{l}
%\quad \pi_{Q} : Q\to \C\mathbb{P}^m .
$$
Then $\pi_P \colon P \to \C\mathbb{P}^m$ is a principal $\G$-bundle, and $ \pi_{Q} \colon Q\to \C\mathbb{P}^m$ is a principal $\T$-bundle. 
Thus the first assertion is satisfied.

We next prove the second assertion.
Let $q_a \colon P \to L_{a}^{\vee}$ be the natural projection:
$$
\xymatrix@C=50pt@R=40pt{
P \coloneqq L_{1}^{\vee, \times} \times_{\C\mathbb{P}^m}\cdots \times_{\C\mathbb{P}^m} L_{l}^{\vee, \times}
\ar[r]^-{q_a}
\ar@/^2pc/[rr]^{\pi_P}
   & L_{a}^{\vee, \times} \ar[r]^-{\tau_a}
   &\C\mathbb{P}^m . \\
   }
$$
We construct the $\eta^\C \in \Omega^1(P, \mathfrak{t})$ on $P$ as follows. 
For each $a=1, \ldots, l$, let $(W, w)$ be a coordinate neighborhood with $W \subset \C\mathbb{P}^m$, and let $e_{a}^{\vee} \colon W \to L_{a}^{\vee}$ be a holomorphic local section. For $p_a \in L_{a}^{\vee}\vert_{W}$, take the coordinate $z_a$ defined by
\begin{equation}
\label{eq-pa-zaea}
p_a = z_a(p_a) \cdot e_{a}^{\vee}(\tau_a (p_a)).
\end{equation}
This gives coordinates $(z_a, w)$ on $L_{a}^{\vee}\vert_{W}$. The vector bundle $\tau_a ^{*}L_{a}^{\vee} \to L_{a}^{\vee}$ carries the metric
$ H_{a}^{\vee} \coloneqq \tau_a ^{*}h_a^{\vee}$. Using this metric, we define a form on $P$ by
\begin{equation}
\label{eq-etac-defn}
\eta^\C \coloneqq 
\sum_{a=1}^{l}
\frac{1}{\sqrt{-1}}
q_{a}^{*}\underbrace{(d \log z_a + \partial \log H_{a}^{\vee})}_{\text{$1$-form on $L_{a}^{\vee}$}}
 \otimes \xi_a
\in \Omega^{1}(P,\mathfrak t_\C).
\end{equation}
Then, by using $\mathbf{p}
=\sum_{a=1}^{l} \mathbf{p}_{a} \xi_a$, 
$$
d\eta^\C=
\sum_{a=1}^{l}
\frac{1}{\sqrt{-1}}
q_{a}^{*}\bar\partial \partial \log H_a^{\vee} \otimes \xi_a
=
\sum_{a=1}^{l}
\frac{1}{\sqrt{-1}}
%\underbrace{q_{a}^{*}\tau_a ^{*}}_{\pi_P^{*}}
\left(
q_{a}^{*}\tau_a ^{*}
F_{h_a^{\vee}}\right)\otimes \xi_a
\underalign{\eqref{eq-def-h_a}}{=}
%\sum_{a=1}^{l}\mathbf{p}_{a}
%\pi_P^{*}\omega_{\C\mathbb{P}^m}
%\otimes \xi_a
%\underalign{\text{by the definition of $ \mathbf{p}$}}{=}
\pi_P^{*}\omega_{\C\mathbb{P}^m}
\otimes \mathbf{p}.
$$
\end{proof}

\begin{rem}
\label{rem-eta-i}
Set $\eta \coloneqq \eta^{\C}\vert_{Q}$. Then $\eta$ is the same connection used in \cite[Subsection 3.6]{HL24}.
In particular,
$d\eta= {\pi_{Q}}^{*} \omega_{\C\mathbb{P}^m} \otimes  \mathbf{p}$.
Take $1$-forms $\eta_a$ on $Q$ such that $\eta=\sum_{a=1}^{l} \eta_a  \otimes \xi_a$. On each fiber of $Q\to\C\mathbb P^m$, 
$$
\eta_a=\eta^\C_a \vert_{Q}
=\frac{1}{\sqrt{-1}}
q_{a}^{*} d \log z_a .
$$
From this expression, for any $b \in \C\mathbb{P}^m$,
\begin{equation}
\label{eq-integral-Qb}
\int_{ \pi_{Q}^{-1}(b)}\eta_1 \wedge \cdots \wedge \eta_l
%=\int_{\T}d\theta_1 \wedge \cdots \wedge d \theta_l
= (2 \pi)^l .
\end{equation}
\end{rem}

\subsubsection{Construction of $X_P$ and $X_Q$}
\begin{lem}
\label{lem-XPXQ-construction}
Define a left $\G$-action on $P\times X$ by
\[
t\cdot(p,x)\coloneqq(p\cdot t^{-1},t\cdot x)
\qquad(x\in X,\ p\in P,\ t\in\G).
\]
Let $X_P \coloneqq P\times_{\G} X=(P\times X)/\G$ be the quotient by this $\G$-action, and let $\rho \colon P \times X \to X_P$ be the quotient map. Then the following assertions hold.
\begin{enumerate}
\item The associated fibration $\pi \colon X_P\to \C\mathbb{P}^m$ is a holomorphic submersion. More precisely, it is a holomorphic fiber bundle with fiber $X$.
\item For $X_Q \coloneqq Q \times_\T X$ in \cite[Subsection 3.6]{HL24}, 
$X_Q$ is biholomorphic to $X_P$.
\end{enumerate}
\end{lem}
\begin{proof}
Let $\C\mathbb{P}^m = \bigcup_{\lambda} U_{\lambda}$ be an open covering, and let $\rho_{\lambda \mu} \colon U_\lambda \cap U_\mu \to \G$ be the transition functions of $\pi \colon X_P\to \C\mathbb{P}^m$. Viewing $\G$ as a subgroup of $\Aut{X}$, the associated bundle $X_P=P \times_{\G} X$ is a holomorphic fiber bundle over $\C\mathbb{P}^m$ with local transition functions $\rho_{\lambda \mu} \colon U_\lambda \cap U_\mu\to \Aut{X}$.
For the second assertion, since
$P \underset{}{\cong} Q \times_{\mathbb{T}}\G$, the map
$$
Q\times_{\mathbb{T}} X 
\to 
P \times_{\G} X \cong Q \times_{\mathbb{T}}\G \times_{\G} X 
\quad 
[q,x] \mapsto [q,1,x]
$$
is induced. This map gives a biholomorphism.
\end{proof}
\begin{rem}
Since $X$ admits a $\G$-equivariant ample line bundle $L$, the line bundle $L_P$ constructed in Subsubsection \ref{subsubsec-slope-formula} is relatively ample for $X_P \to B$. Hence $X_P$ is a projective manifold.
\end{rem}

\begin{defn}
\label{defn-omegaP-kahler}
We take a Kähler form $\omega_{P}$ on $X_P$ such that
$$
\rho^{*}\omega_P\vert_{Q \times X}
=p_2^{*}\omega+\underbrace{\bigl(c + \langle p_2^{*}\mu,\mathbf{p}\rangle\bigr)}_{=p_2^{*}f(\mu)}\cdot p_1^{*}\pi_Q^*\omega_{\C\mathbb{P}^m}
+\langle p_2^{*} d\mu\wedge p_1^{*}\eta\rangle
\quad 
\text{ on $Q \times X$}.
$$
By \cite[Subsection 3.6]{HL24} and Lemma \ref{lem-XPXQ-construction}, such a form $\omega_{P}$ exists.
\end{defn}

\subsubsection{Construction of vector bundles and Hermitian metrics on $X_P$}
Let $\rho \colon P \times X \to X_P$ be the quotient map, and let $p_1 \colon P \times X \to P$ and $p_2 \colon P \times X \to X$ be the first and second projections, respectively.
 \begin{equation}
 \label{eq-diagram-XP}
\xymatrix@C=50pt@R=40pt{
 X &P \times X\ar[l]_-{p_2} \ar[r]^-{p_1} \ar[d]^{\rho}&P \ar[r]^-{\pi_{P} \text{ fiber $\G$} }&\C\mathbb{P}^m\\
  & X_P \cong X_Q \ar[rru]_-{\pi \text{ fiber $X$} }&  &
}
 \end{equation}
For simplicity of notation, we use the same notations for the quotient map
\(
\rho\colon Q\times X\rightarrow X_P
\)
and the projections
\(
p_1\colon Q\times X\to Q,p_2\colon Q\times X\to X
\)
obtained by restriction to $Q$.

Let $E\to X$ be a $\G$-equivariant holomorphic vector bundle. Then $p_2^{*}E \to P \times X$ is also a holomorphic vector bundle. We define a $\G$-action on $P\times E$ by
\[
t\cdot(p,e)\coloneqq(p\cdot t^{-1},t\cdot e)
\qquad(e\in E,\ p\in P,\ t\in\G).
\]

and set $E_P \coloneqq P\times_{\G}E$. Then $E_P\to X_P$ is a holomorphic vector bundle.

\begin{lem}
\label{lem-hP-defn}
Under the above notations, let $h$ be a $\mathbb{T}$-invariant metric on $E$. (Note that $h$ is not assumed to be $\G$-invariant.) Let $\nabla^{h}$ be the Chern connection of $h$, and let $\Phi_{h}$ be the $\End(E) \otimes \mathfrak{t}^{\vee}$-valued moment map of $h$.

Then there exists a Hermitian metric $\widetilde{h_P}$ on the vector bundle $p_2^{*}E$ over $P \times X$ satisfying the following properties.
\begin{enumerate}
\item $\widetilde{h_P}$ is $\G$-invariant. In particular, there exists a Hermitian metric $h_P$ on $E_P$ such that
$\widetilde{h_P}=\rho^{*}h_P$.
\item For the Chern connection $\nabla^{\widetilde{h_P}}$ of $\widetilde{h_P}$,
$$
\nabla^{\widetilde{h_P}}\vert_{Q \times X}
=
p_2^{*}\nabla^{h}
+\langle p_2^{*}\Phi_h\otimes p_1^{*}\eta \rangle 
\quad 
\text{ on $Q \times X$ }.
$$
In particular, setting $\Phi_h^{\xi_a} \coloneqq \langle \Phi_h, \xi_{a}\rangle \in \End(E)$, we have
$$
F_{\widetilde{h_P}}\vert_{Q \times X}
=
p_2^{*}F_h+
\langle p_2^{*}\Phi_h,\mathbf{p}\rangle  \cdot \pi_Q^*\omega_{\C\mathbb{P}^m} 
+\sum_{a=1}^l p_2^{*}\nabla^{h} \cdot p_2^{*}\Phi_h^{\xi_a}\wedge p_1^{*}\eta_a
+\sum_{a,b=1}^l p_2^{*}(\Phi_h^{\xi_a}\Phi_h^{\xi_b})\otimes p_1^{*}(\eta_a\wedge \eta_b).
$$
%$$
%\xymatrix@C=40pt@R=40pt{
 %X &Q \times X\ar[l]_{p_2} \ar[r]^{p_1} \ar[d]^{\rho}&Q \ar[r]^{\pi_{Q}}_{\text{fiber $\G$}}&\C\mathbb{P}^m\\
 % & X_Q=X_P \ar[rru]_{\pi}^{\text{fiber $X$}}&  &
%}
%$$
\end{enumerate}
\end{lem}
The last assertion about $F_{\widetilde{h_P}}\vert_{Q \times X}$ follows from \cite[Subsection 3.6]{HL24}.
%since we can apply the connection $\nabla$ there to the Chern connection $\nabla^h$ of $h$, then $\nabla^Q$ agrees with $\nabla^{\widetilde{h_P}}|_{Q \times X}$. 
This lemma is needed in the proof of Lemma \ref{lem-max-slope-comparison} in order to take the second fundamental form of a subsheaf of $p_2^*E$ with respect to the metric $\widetilde{h_P}$. %We also use here the fact that $P\times X$ is a complex manifold.

\begin{proof}
(1). 
We define a Hermitian metric $\widetilde h_P$ on $p_2^*E\to P\times X$. Notice that every point $p\in P$ can be written uniquely as
\[
p=q \cdot g,\qquad q=(q_1,\ldots,q_l)\in Q,\quad g=(g_1,\ldots,g_l)\in(\mathbb R_{>0})^l\subset \G.
\]
Indeed, for
\(
p=(p_1, \ldots, p_l) \in P \coloneqq L_{1}^{\vee, \times} \times_{\C\mathbb{P}^m}\cdots \times_{\C\mathbb{P}^m} L_{l}^{\vee, \times},
\)
set
$$
g \coloneqq (\lvert p_1\rvert_{h_{1}^{\vee}}, \ldots, \lvert p_l\rvert_{h_{l}^{\vee}}) \in (\mathbb R_{>0})^l,
\qquad
q \coloneqq \left(\frac{p_1}{\lvert p_1\rvert_{h_{1}^{\vee}}}, \ldots, \frac{p_l}{\lvert p_l\rvert_{h_{l}^{\vee}}}\right)\in Q.
$$
Then $p=q \cdot g$, and this expression is unique.
For $(p, x) \in P \times X$ and $e_1,e_2\in (p_2^*E)_{(p,x)}=E_x$, define
\begin{equation}
\label{eq-defn-widetildeh}
\widetilde h_{P,(p,x)}(e_1,e_2) \coloneqq h_{gx}(ge_1,ge_2).
\end{equation}
Then the $\mathbb T$-invariance of $h$ implies that $\widetilde h_P$ is invariant under the $\G$-action. In particular, it induces a Hermitian metric $h_P$ on $E_P$.

(2). We compute the Chern connection $\nabla^{{\widetilde h}_P}$ of $\widetilde h_P$. 
%This can be done locally. 
Let $e=(e_1,\ldots,e_r)$ be a local frame of $E$, and let
\(
H(x) \coloneqq (h_x(e_i,e_j))
\)
be the metric matrix of $h$. 
Write the $\G$-action as
\(
g\cdot e(x)=e(gx)G(g,x)
\)
using a matrix $G(g,x)$. Then \eqref{eq-defn-widetildeh} yields
\begin{equation}
\label{eq-wideH-GHG}
{\widetilde H}(p,x)\underalign{\eqref{eq-defn-widetildeh}}{=}\overline{{}^tG(g,x)}H(gx)G(g,x).
\end{equation}
Since $P \times X$ is a complex manifold, using the decomposition $\partial=\partial_X+\partial_P$, we have
$$
\nabla^{{\widetilde h}_P}
=
(d + {\widetilde H}^{-1}\partial_X{\widetilde H})
 + {\widetilde H}^{-1}\partial_P{\widetilde H}.
$$
By definition, $\widetilde H$ is invariant in the $Q$-direction, and therefore
$\widetilde H\vert_{Q\times X}
=p_2^*H$.
It follows that
$$
(d + {\widetilde H}^{-1}\partial_X{\widetilde H})\vert_{Q\times X}
=
p_2^{*}\nabla^{h}.
$$
Thus it remains to show that
$({\widetilde H}^{-1}\partial_P{\widetilde H})\vert_{Q \times X}= \langle \Phi_h \otimes \eta \rangle $. 
As in the proof of Lemma \ref{lem-existence-Q}, take $p_a\in L_{a}^{\vee, \times}$, $z_a$, $H_a$, and $e_{a}^{\vee}$. Set $t_a\coloneqq\log g_a$. Since $p_a=z_a e_a^{\vee}$ by \eqref{eq-pa-zaea}, we have
\(
g_a^2=\lvert z_a\rvert^2 H_a^{\vee},
\)
and hence
\[
\partial_P t_a
=
\partial_P\log g_a
\qquad
\underalign{\text{(by $g_a^2=\lvert z_a\rvert^2 H_a^{\vee}$)}}{=}
\qquad
\frac{1}{2}q_a^*
\left(d\log z_a+\partial\log H_a^{\vee}\right)
\underalign{\eqref{eq-etac-defn}}{=}
\frac{\sqrt{-1}}{2}\eta_a^\C.
\]
Here $ \eta_a^\C$ is defined by
$
\eta^\C
=
\sum_{a=1}^{l}
\eta_a^\C
 \otimes \xi_a
$.
Set $\Phi_h^{\xi_a} \coloneqq \langle \Phi_h, \xi_a\rangle$. Then, by \eqref{eq-canonically-moment-map}, the moment map is canonically given by
\(
\Phi_h^{\xi_a}
=
\nabla^h_{\xi_{a,X}}-\mathcal L^E_{\xi_{a, X}}
\)
and a direct calculation using \eqref{eq-wideH-GHG} yields
$$
\Phi_h^{\xi_a}
=
\frac{\sqrt{-1}}{2}{\widetilde H}^{-1}
\frac{\partial{\widetilde H}}{\partial t_a}
\Bigr\vert_{Q \times X}.
%=\frac{\sqrt{-1}}{2}{\widetilde H}^{-1}\frac{\partial{\widetilde H}}{\partial\log g_a}\vert_{Q \times X}.
$$
Since $\widetilde H$ is invariant in the $Q$-direction, 
\[
{\widetilde H}^{-1}\partial_P{\widetilde H}
\Bigr\vert_{Q \times X}
=
\sum_{a=1}^{l}
\underbrace{
{\widetilde H}^{-1}\frac{\partial{\widetilde H}}{\partial t_a }
\Bigr\vert_{Q \times X}
}_{=\frac{2}{\sqrt{-1}}\Phi_h^{\xi_a}}
\cdot 
\underbrace{
\partial_P t_a\vert_{Q \times X}}_{=\frac{\sqrt{-1}}{2 }  \eta_a^\C\vert_{Q \times X}}
%=
%\sum_{a=1}^{l}
%\underbrace{
%{\widetilde H}^{-1}\frac{\partial{\widetilde H}}{\partial\log g_a }
%\vert_{Q \times X}}_{=\frac{2}{\sqrt{-1}}\Phi_h^{\xi_a}}
%\cdot 
%\underbrace{
%\partial_P\log g_a\vert_{Q \times X}}_{=\frac{\sqrt{-1}}{2 }  \eta_a^\C\vert_{Q \times X}}
=
\sum_{a=1}^{l}\Phi_h^{\xi_a} \cdot   \eta_a^\C\vert_{Q \times X}
\underalign{\text{(Lem. \ref{lem-existence-Q})}}{=}
\langle \Phi_h \otimes \eta \rangle .
\]
\end{proof}

\subsubsection{The slope formula}
\label{subsubsec-slope-formula}
Let $\mathcal{E}$ be a $\G$-equivariant torsion-free sheaf on $X$. Then $p_2^{*}\mathcal{E}$ is also a torsion-free sheaf on $P \times X$. 
Since $p_2^{*}\mathcal{E}$ is also $\G$-equivariant, there exists a torsion-free sheaf $\mathcal{E}_P$ on $X_P$ such that
$p_2^{*}\mathcal{E} \cong \rho^{*}\mathcal{E}_P$.

\begin{lem}
\label{lem-degree-formula}
%Let $\mathcal{E}$ be a $\G$-equivariant torsion-free sheaf on $X$, and let $\mathcal{E}_P$ be the torsion-free sheaf on $X_P$ constructed above. 
Under the above notations, 
$$
c_1(\mathcal{E}_P) \cdot \{ \omega_P\}^{[n+m-1]}
=V\cdot (c_{1}^{\T}(\mathcal{E}) \cdot v(\alpha_{\T})).
$$
%here $V \coloneqq \int_{\C\mathbb{P}^m}\omega_{\C\mathbb{P}^m}^{[m]}$.
Similarly, 
$
\ch_2(\mathcal{E}_P) \{ \omega_P \}^{[n+m-2]}
=V\cdot (\ch_{2}^{\T}(\mathcal{E}) \cdot v(\alpha_{\T}))$, $
c_1(\mathcal{E}_P)^2 \{ \omega_P\}^{[n+m-2]}
=V\cdot (c_{1}^{\T}(\mathcal E)^2 v(\alpha_{\T}))$, and $\{ \omega_P\}^{[n+m]}=V\cdot v(\alpha_{\T})
$ hold.
\end{lem}
The equality
$c_1(\mathcal{E}_P) \cdot \{ \omega_P\}^{[n+m-1]}
=V(c_{1}^{\T}(\mathcal{E}) \cdot v(\alpha_{\T}))$
is proved in the same way as in \cite[Subsection 3.6]{HL24}. 
However, since we use this fact later, we give a proof.
\begin{proof}
We first treat the case where $\mathcal{E}$ is a vector bundle, and prove $c_1(\mathcal{E}_P) \cdot \{ \omega_P\}^{[n+m-1]}
=V \cdot (c_{1}^{\T}(\mathcal{E}) \cdot v(\alpha_{\T}))$.
Let $\eta^{\wedge, l} \coloneqq \eta_1\wedge\cdots\wedge\eta_l$ be the $l$-form on $Q$. Using the biholomorphism $X_P \cong X_Q$, for any $(n+m, n+m)$-form $\Omega$ on $X_P$, 
\begin{equation}
\label{eq-integral-XQ1}
\int_{X_P}\Omega
=
\frac{1}{(2\pi)^l}
\int_{Q\times X}
\rho^*\Omega \wedge p_1^*\eta^{\wedge, l}.
\end{equation}
We apply this formula to
\(
\Omega  \coloneqq  \frac{\sqrt{-1}}{2\pi}\tr F_{h_P} \wedge \omega_P^{[n+m-1]}.
\)
For this purpose, we compute
\begin{equation}
\label{eq-rho-Omega}
\begin{aligned}
&\rho^*\Omega \wedge p_1^*\eta^{\wedge, l}
\underalign{\text{(Lem. \ref{lem-hP-defn})}}{=}
\frac{\sqrt{-1}}{2\pi} \left(\tr F_{\widetilde{h_P}}\vert_{Q \times X}\right)
\wedge \rho^*\omega_P^{[n+m-1]} \wedge p_1^*\eta^{\wedge, l}
\\
&\underalign{\raisebox{-1.0ex}{\tiny\text{(Def. \ref{defn-omegaP-kahler} \& Lem. \ref{lem-hP-defn})}}}{=}
\frac{\sqrt{-1}}{2\pi}\left(p_2^{*}\tr F_h+
\tr\langle p_2^{*}\Phi_h,\mathbf{p}\rangle
\cdot \pi_Q^*\omega_{\C\mathbb{P}^m} \right)
\wedge
\left(p_2^{*}\omega+p_2^{*}f(\mu)\cdot p_1^{*}\pi_Q^*\omega_{\C\mathbb{P}^m}
\right)^{[n+m-1]}
\wedge
p_1^*\eta^{\wedge, l}
\\
&\underalign{}{=}
\underbrace{\frac{\sqrt{-1}}{2\pi}
\left( p_2^{*}f(\mu)^m \cdot p_2^{*}\tr F_h \wedge p_2^{*}\omega^{[n-1]}
+ m p_2^{*}f(\mu)^{m-1} \cdot \tr\langle p_2^{*}\Phi_h,\mathbf{p}\rangle \cdot p_2^{*}\omega^{[n]}
 \right)
 }_{=\frac{\sqrt{-1}}{2\pi}p_2^{*}\tr\Lambda_{\omega,v}(F_h+\Phi_h)\omega^{[n]}
  \text{ by Ex. \ref{ex-explicit-slope}}}\wedge
\underbrace{\pi_Q^*\omega_{\C\mathbb{P}^m}^{[m]}
\wedge
p_1^*\eta^{\wedge, l}}_{\text{pullback of a form on $Q$}}
\\
\end{aligned}
\end{equation}
Fiber integration along $Q \to \C\mathbb{P}^m$ gives
\begin{equation}
\label{eq-int-QpiB}
\frac{1}{(2\pi)^l}
\int_{Q}\pi_Q^*\omega_{\C\mathbb{P}^m}^{[m]}
\wedge
\eta^{\wedge, l}
\underalign{\eqref{eq-integral-Qb}}{=}
 \int_{\C\mathbb{P}^m}\omega_{\C\mathbb{P}^m}^{[m]}=V.
\end{equation}
Thus, by Fubini's theorem, 
\[
\begin{aligned}
c_1(\mathcal{E}_P) \cdot \{\omega_P\}^{[n+m-1]}
&\underalign{}{=}\int_{X_P}\frac{\sqrt{-1}}{2\pi}\tr F_{h_P} \wedge \omega_P^{[n+m-1]}
\\
&\underalign{\eqref{eq-integral-XQ1} \text{ \& } \eqref{eq-rho-Omega} }{=}
V \cdot \int_{X}\frac{\sqrt{-1}}{2\pi}\tr\Lambda_{\omega,v}(F_h+\Phi_h)\omega^{[n]}
%&\underalign{\text{}}{=}
%V\int_{X}\frac{\sqrt{-1}}{2\pi}
%\left(f(\mu)^{m}\mathrm{tr}F_h \wedge \omega^{[n+m-1]}
%+ m f(\mu)^{m-1}\mathrm{tr}\langle \Phi_h,\mathbf{p}\rangle \omega^{[n+m]} \right) \\
 \underalign{\text{(Lem. \ref{lem-intersection-chern})}}{=}
V(c_{1}^{\T}(\mathcal{E}) \cdot v(\alpha_{\T})).
\end{aligned}
\]
Thus, the desired equality holds.

We next consider the case where $\mathcal{E}$ is torsion-free. Since $X$ is smooth and projective, $\mathcal E$ admits a $\G$-equivariant finite locally free resolution
$$
0\longrightarrow E_N \longrightarrow\cdots\longrightarrow E_1\longrightarrow E_0\longrightarrow\mathcal E\longrightarrow 0.
$$
This induces a finite locally free resolution on $X_P$:
$$
0\longrightarrow E_{N, P} \longrightarrow\cdots \longrightarrow E_{1, P} \longrightarrow E_{0, P} \longrightarrow {\mathcal E}_P \longrightarrow 0.
$$
Therefore, by the additivity of the Chern character,
$$
\ch^{\mathbb T}(\mathcal E)
=
\sum_{i=0}^N(-1)^i\ch^{\mathbb T}(E_i)
\quad \text{and} \quad 
\ch(\mathcal E_P)
=
\sum_{i=0}^N(-1)^i\ch(E_{i, P}),
$$
Thus the result follows from the vector bundle case.

The other equalities are proved in the same way. For reference, we record the corresponding computation for $\ch_2(\mathcal{E}_P)$:
\begin{align*}
&\ch_2(\mathcal{E}_P) \cdot \{ \omega_P\}^{[n+m-2]}
\\
&=V 
\cdot \frac{-1}{8 \pi^2} 
\int_{X}\left( 
f(\mu)^{m}  \tr_{} F_{h}^2 \omega^{[n-2]}
 + 
 2m f(\mu)^{m-1}  \tr(\langle\Phi_h,\mathbf p\rangle F_h)
  \omega^{[n-1]}
 +
m (m-1) f(\mu)^{m-2} \tr_{}  \langle \Phi_h, \mathbf{p}\rangle^2 \omega^{[n]}  
\right)
\\
&=V \cdot (\ch_{2}^{\T}(\mathcal E) \cdot v(\alpha_{\T})).
\end{align*}
\end{proof}

\subsubsection{Comparison of the maximal slopes of \(\mathcal{E}\) and \(\mathcal{E}_P\)}
\begin{lem}
\label{lem-max-slope-comparison}
Let $\mathcal{E}$ be a $\G$-equivariant torsion-free sheaf, and let $\mathcal{E}_P$ be the torsion-free sheaf on $X_P$ constructed above. Then 
$$
\mu^{\max}_{\omega_{P}^{[n+m-1]}}(\mathcal{E}_P)
=
V \cdot \mu^{\max}_{v}(\mathcal{E}).
$$
\end{lem}
\begin{proof}
By Remark \ref{rem-soliton-GKP16a}\ref{prop-min-max-reflexive},
we may assume that $\mathcal E$ is reflexive.
Let $W\subset X$ be a $\G$-invariant open subset on which $\mathcal E$ is locally free and such that $\codim(X\setminus W)\geq2$.
Set $E\coloneqq\mathcal E|_W$. Choose a $\G$-equivariant embedding $\mathcal E\hookrightarrow B$ into a vector bundle with a smooth $\T$-invariant Hermitian metric, as in the proof of Theorem \ref{thm-GKP16a-2.29}.
After shrinking $W$ with $\codim(X\setminus W)\geq2$, let $h$ be the induced metric on $E$.
(This choice ensures that the induced metric on $\mathcal G|_{U_P}$ comes from the smooth metric on $B_P$, so the first equality in \eqref{eq-degree-pullback-slice-start} follows from \cite[Section~5.8]{Kob14}.)
Then the vector bundle $p_{2}^{*}E$ on $P \times W$ admits a Hermitian metric $\widetilde{h}_{P}$ as in Lemma \ref{lem-hP-defn}.

Let $W_P \coloneqq \rho(P \times W) \subset X_P$, and let $E_P$ be the quotient bundle of $p_{2}^{*}E$. Then $E_{P}=\mathcal{E}_P\vert_{W_P}$. By Lemma \ref{lem-hP-defn}, $E_P$ admits a Hermitian metric $h_P$ such that $\widetilde{h}_{P}=\rho^*h_P$. Moreover, $\codim X_P\setminus W_P \ge 2$.
\[
\xymatrix@C=50pt@R=40pt{
 P \times W\ar@{^{(}->}[r]_{}  \ar[d]^{\rho}
 &P \times X \ar[r]^-{p_1} \ar[d]^{\rho}&P \ar[r]^{\pi_{P} \text{ fiber $\G$} }_{}&\C\mathbb{P}^m\\
 W_P\ar@{^{(}->}[r]_{}  & X_P\ar[rru]_{\pi \text{ fiber $X$}}^{}&  &
}
\]

We first prove
\(
\mu^{\max}_{\omega_P^{[n+m-1]}}(\mathcal{E}_P)
\le
V \cdot \mu^{\max}_{v}(\mathcal{E})
\).
Let $\mathcal{G}\subset\mathcal{E}_P$ be the maximal destabilizing sheaf with respect to $\omega_P^{[n+m-1]}$. In particular, $\mathcal{E}_P/\mathcal{G}$ is torsion-free. Set $r \coloneqq \operatorname{rk}\mathcal G$. It is enough to show that
\begin{equation*}
%\label{eq-maximal-dest-slope}
\mu_{\omega_P^{[n+m-1]}}(\mathcal{G})
\le
V \cdot \mu^{\max}_{v}(\mathcal{E}).
\end{equation*}

We define a $\G$-action on $X_P=P\times_{\G}X$ by
\(
\tau\cdot[p,x] \coloneqq [p,\tau x] 
\)
for $\tau\in\G$. Since $\G$ is commutative, this is well-defined. Similarly, $\mathcal E_P$ carries a $\G$-action. 
By the uniqueness of the maximal destabilizing sheaf and the fact that $\tau\{\omega_P\}=\{\omega_P\}$ for all $\tau\in\G$, we obtain $\tau \mathcal G = \mathcal G$. This implies that $\mathcal G$ is $\G$-equivariant.

Choose an open subset $U_P \subset W_P$ such that
$\codim X_P\setminus U_P \ge 2$ and
$ \mathcal{G}\vert_{U_P} \subset  \mathcal{E}_P\vert_{U_P}$
is a subbundle of vector bundles. Since
$\mathcal{E}_P\vert_{U_P}=E_P\vert_{U_P}$,
the metric $h_P$ induces a submetric $h_{P,\mathcal{G}}$ on $\mathcal{G}\vert_{U_P}$. Set
$\widehat U \coloneqq \rho^{-1}(U_P)\cap (Q\times W)$. By Lemma \ref{lem-degree-formula},
\begin{equation}
\label{eq-degree-pullback-slice-start}
\begin{aligned}
%c_1(\mathcal{G}) \cdot \{ \omega_P\}^{[n+m-1]}
&\mu_{\omega_P^{[n+m-1]}}(\mathcal{G})
\qquad
\underalign{\text{\cite[Rem. 5.8.5]{Kob14} }}{=}
\qquad
\frac{1}{r}\int_{U_P} \frac{\sqrt{-1}}{2\pi}
\tr \left(F_{h_{P,\mathcal{G}}}\right) \wedge  \omega_P^{[n+m-1]}
\\
&\underalign{\eqref{eq-integral-XQ1}}{=}
\frac{1}{r} \frac{1}{(2\pi)^l}\int_{\widehat U}
\frac{\sqrt{-1}}{2\pi}
\tr\left(\rho^*F_{h_{P,\mathcal G}}\right)
\wedge
\rho^*\omega_P^{[n+m-1]}
\wedge
p_1^*\eta^{\wedge, l}.
\end{aligned}
\end{equation}

On $\rho^{-1}(U_P)\subset P\times W$, the pullback
\(
\rho^*\mathcal G \subset p_2^*E
\)
is a holomorphic subbundle, and by Lemma \ref{lem-hP-defn}, $p_2^*E$ carries the Hermitian metric $\rho^{*}h_P$.
%$\widetilde h_P=\rho^{*}h_P$. 
Take the $\operatorname{Hom}(\rho^*\mathcal G, {\rho^*\mathcal G}^{\perp}\bigr)$-valued $(1,0)$-form
\(
\widehat b
\)
which is the second fundamental form of the holomorphic subbundle
\(
\rho^*\mathcal G \subset p_2^*E
\)
with respect to the metric $\rho^{*}h_P$. If $\widehat\pi$ denotes the orthogonal projection onto $\rho^*\mathcal G$ with respect to $\rho^{*}h_P$, 
%$\widetilde h_P=\rho^*h_P$, 
then \cite[Lemma 11.2]{Dem12} gives
\begin{equation}
\label{eq-dem11-2nd}
\rho^*F_{h_{P,\mathcal G}}
=
\widehat\pi \rho^{*}F_{h_P}
+
 \widehat b^* \wedge \widehat b
 \quad 
 \text{ on $\rho^{-1}( U_P)\subset P\times X$ }.
\end{equation}
Let $Z \subset \C\mathbb{P}^m$ be a local open set such that $P\vert_{Z} \cong Z \times \G$. Then, locally, $\widehat b$ may be regarded as a form on $Z \times \G \times W$. With respect to this decomposition, write
\begin{equation}
\label{eq-2nd-decomposition}
\widehat b
= \underbrace{\widehat b_{\C\mathbb{P}^m}}_{\text{component in  $\C\mathbb{P}^m$}} + \widehat b_{P} + \underbrace{\widehat b_{X} }_{\text{component in $X$}}
 \quad 
 \text{ on $Z \times \G \times W \subset P\times X$ }.
\end{equation}

We now calculate $\left.\left(\rho^*F_{h_{P,\mathcal G}}\wedge\rho^*\omega_P^{[n+m-1]}\right)\right\vert_{Q \times X} \wedge p_1^*\eta^{\wedge, l}$. 
Recall that a form containing $p_1^*\eta_a$ vanishes after taking a wedge with
\(
p_1^*\eta^{\wedge, l}
\).
To express this simply, we introduce the following notation.
\begin{notation}
\label{notation-eta}
For two forms $\tau_1, \tau_2$ on $Q \times X$, we define the equivalence relation by
$$
\tau_1
\equiv 
\tau_2
%\mod p_1^*\eta^{\wedge, l}
\overset{\mathrm{def}}{\iff}
(\tau_1-\tau_2) \wedge  p_1^*\eta^{\wedge, l}=0.
$$
\end{notation}
Using this notation, 
\[
\rho^*\omega_P\vert_{Q \times X}
\underalign{\text{(Def. \ref{defn-omegaP-kahler} \& No. \ref{notation-eta})}}{\equiv}
p_2^*\omega
+
p_2^*f(\mu) \cdot p_1^*\pi_Q^*\omega_{\C\mathbb{P}^m}, 
\quad
\rho^*F_{h_P}\vert_{Q \times X}
\underalign{\text{(Lem. \ref{lem-hP-defn} \& No. \ref{notation-eta})}}{\equiv}
p_2^*F_h+
\langle p_2^*\Phi_h,\mathbf{p}\rangle \cdot p_1^*\pi_Q^*\omega_{\C\mathbb{P}^m}  .
\]
It follows that
\begin{equation}
\label{eq-Tseparated}
\resizebox{0.97\linewidth}{!}{$\displaystyle\begin{aligned}
&
\left. \left(\rho^*F_{h_{P,\mathcal G}}\wedge\rho^*\omega_P^{[n+m-1]}\right)\right\vert_{Q \times X}
\\
&\underalign{\raisebox{-2.0ex}{\tiny\text{No. \ref{notation-eta} \& \eqref{eq-2nd-decomposition}}}}{\equiv}
\underbrace
{\left( 
p_2^*f(\mu)^m \left. \left( \widehat\pi p_2^*F_h+\widehat b_X^* \wedge \widehat b_X \right) \right\vert_{Q \times X} \cdot
 p_2^*\omega^{[n-1]}
+
m p_2^*f(\mu)^{m-1}
\langle \widehat\pi p_2^*\Phi_h|_{\rho^*\mathcal G},\mathbf p\rangle \cdot
 p_2^*\omega^{[n]}
\right)
}_{\eqqcolon T}
\wedge p_1^*\pi_Q^*\omega_{\C\mathbb{P}^m}^{[m]}
\\
& \qquad + 
p_2^*f(\mu)^{m-1}
 \widehat b_{\C\mathbb{P}^m}^* \wedge \widehat b_{\C\mathbb{P}^m}
\wedge  p_2^*\omega^{[n]} 
\wedge p_1^*\pi_Q^*\omega_{\C\mathbb{P}^m}^{[m-1]} .
\end{aligned}$}
\end{equation}

Thus
\begin{equation}
\label{eq-degree-negative}
\resizebox{0.97\linewidth}{!}{$\displaystyle\begin{aligned}
&\mu_{\omega_P^{[n+m-1]}}(\mathcal{G})
\quad
\underalign{\eqref{eq-degree-pullback-slice-start} \text{ \& } \eqref{eq-Tseparated}}{=}
\quad
\frac{1}{r}\frac{1}{(2\pi)^l}\int_{\widehat U}
\frac{\sqrt{-1}}{2\pi}
\tr T
\wedge p_1^*\pi_Q^*\omega_{\C\mathbb{P}^m}^{[m]}
\wedge
p_1^*\eta^{\wedge, l}
\\&+
\underbrace{
\frac{1}{r}\frac{1}{(2\pi)^l}\int_{\widehat U}
\frac{\sqrt{-1}}{2\pi}
p_2^*f(\mu)^{m-1}
\tr \left(  \widehat b_{\C\mathbb{P}^m}^* \wedge \widehat b_{\C\mathbb{P}^m} \right)
\wedge  p_2^*\omega^{[n]} 
\wedge p_1^*\pi_Q^*\omega_{\C\mathbb{P}^m}^{[m-1]}
\wedge
p_1^*\eta^{\wedge, l}
}_{\le 0 \text{ (cf. \cite[Cor. 11.8]{Dem12}})} .
\end{aligned}$}
\end{equation}

About the first term on the right-hand side,  
we want to decompose the integral over $\widehat U$
 into an integral over $Q$ and over $X$. For this purpose, for $q\in Q$, set
\[
\widehat U_q \coloneqq \{x\in W\mid (q,x)\in\widehat U\} \subset X .
\]
We now compute $\frac{\sqrt{-1}}{2\pi}\tr T\vert_{\{q\} \times \widehat U_{q}}$.
Fix a general point $q\in Q$ with  $\codim (X\setminus \widehat U_q) \ge 2$.
Then
\[
\rho^*\mathcal G\vert_{\{q\}\times\widehat U_q}
\subset
p_2^*E\vert_{\{q\}\times\widehat U_q}
\simeq
E\vert_{\widehat U_q}
\]
is a holomorphic subbundle. 
Using the reflexivity of $\mathcal E$, the sheaf $\rho^*\mathcal G\vert_{\{q\}\times\widehat U_q}$ extends to $X$. Denote this extension by
\(
\mathcal G_q\subset\mathcal E
\).
Then
\(
\operatorname{rk}\mathcal G_q=r
\)
and $\mathcal G_q$ is a $\G$-equivariant torsion-free subsheaf on $X$, since $\mathcal G$ is equivariant with respect to the $\G$-action on $X_P$. 

On the slice $\{q\}\times\widehat U_q$, the restriction
\(\widehat\pi\vert_{\{q\}\times\widehat U_q}
%,\qquadb_q \coloneqq \widehat b_X|_{\{q\}\times\widehat U_q}
\)
is the orthogonal projection $E\vert_{\widehat U_q}\to\mathcal G_q\vert_{\widehat U_q}$ with respect to $h = \widetilde{h}_P\vert_{\{q\}\times\widehat U_q}$, and $\widehat b_X\vert_{\{q\}\times\widehat U_q}$ is its second fundamental form. 
In particular, setting $h_{\mathcal G_q} \coloneqq h\vert_{\mathcal G_q}$, by \cite[Lemma 11.2]{Dem12} again, 
\[
F_{h_{\mathcal G_q}}
\underalign{}{=}
\widehat\pi\vert_{\{q\}\times\widehat U_q} F_h +
\widehat b_X^*\vert_{\{q\}\times\widehat U_q}
\wedge 
\widehat b_X \vert_{\{q\}\times\widehat U_q} 
\quad
\text{and}
\quad
\Phi_{h_{\mathcal G_q}}
=
\widehat\pi\vert_{\{q\}\times\widehat U_q} \Phi_h\vert_{\mathcal G_q}.
\]
Therefore,
\begin{equation}
\label{eq-T-slice-formula}
\frac{\sqrt{-1}}{2\pi}\tr T\vert _{\{q\}\times\widehat U_q}
=
\underbrace{
f(\mu)^m
\frac{\sqrt{-1}}{2\pi}\tr(F_{h_{\mathcal G_q}})\wedge\omega^{[n-1]}
+
mf(\mu)^{m-1}
\frac{\sqrt{-1}}{2\pi}\tr
\langle\Phi_{h_{\mathcal G_q}},\mathbf{p}\rangle\omega^{[n]}
}_{=\frac{\sqrt{-1}}{2\pi}\tr\Lambda_{\omega,v}(F_{h_{\mathcal G_q}}+\Phi_{h_{\mathcal G_q}})\omega^{[n]}  \text{ by Ex. \ref{ex-explicit-slope}} }.
\end{equation}

Using this, together with \(\operatorname{rk}\mathcal G_q=r\), the local freeness of $\mathcal G_q$ on $\widehat U_q$, the fact that $X\setminus \widehat U_q$ has codimension at least $2$, and \cite[Lemma 5.5]{HL24}, 
\begin{align*}
\mu_{\omega_P^{[n+m-1]}}(\mathcal{G})
&\underalign{\eqref{eq-degree-negative}}{\le}
\frac{1}{(2\pi)^l}\int_{q \in Q} 
\underbrace{\left(\frac{1}{r} \int_{\widehat U_q}
\frac{\sqrt{-1}}{2\pi}\tr T\vert _{\{q\}\times\widehat U_q}
\right)
}_{ =\mu_{v}(\mathcal{G}_q) \text{ by \eqref{eq-T-slice-formula} \& \cite[Lemma 5.5]{HL24}}}
\pi_Q^*\omega_{\C\mathbb{P}^m}^{[m]}
\wedge
p_1^*\eta^{\wedge, l}
\\
&\underalign{}{=}
\frac{1}{(2\pi)^l}
\int_{q \in Q}
\underbrace{\mu_{v}(\mathcal{G}_q)}_{\le \mu_{v}^{\max}(\mathcal{E}) \text{ by $\mathcal G_q\subset \mathcal E$}}
\pi_Q^*\omega_{\C\mathbb{P}^m}^{[m]}
\wedge
\eta^{\wedge, l}
\\
&\underalign{}{\le}
\mu_{v}^{\max}(\mathcal{E})
\cdot 
\underbrace{\frac{1}{(2\pi)^l}
\int_{Q}
\pi_Q^*\omega_{\C\mathbb{P}^m}^{[m]}
\wedge
\eta^{\wedge, l}}_{=V \text{ by \eqref{eq-int-QpiB}}}
= V \cdot \mu_{v}^{\max}(\mathcal{E}).
\end{align*}
Thus, we obtain the desired inequality $\mu^{\max}_{\omega_P^{[n+m-1]}}(\mathcal{E}_P)
\le
V \cdot \mu^{\max}_{v}(\mathcal{E})$.
The reverse inequality
$\mu^{\max}_{\omega_P^{[n+m-1]}}(\mathcal{E}_P)
\ge V \cdot \mu^{\max}_{v}(\mathcal{E})$
follows immediately from Lemma \ref{lem-degree-formula}.
%, \cite[Lemma 5.5]{HL24}, and \cite[Remark 5.8.5]{Kob14}.
\end{proof}

\section{Equivariant Langer's inequality}
\label{sec-BG-soliton}
In this section, we extend Langer's inequality, which is needed for the proof of the Miyaoka--Yau inequality, to the cases of polynomial weights and solitons.

\subsection{Hodge index theorem and Bogomolov--Gieseker inequality for solitons}
Using the elementary limit \(\left(1+\frac{x}{m}\right)^m \to e^x\), we derive the result for solitons, corresponding to exponential weights, from the result for polynomial weights. 
A similar idea was used in \cite{ALL25} to construct a Calabi--Yau metric on the canonical cone over the product of a Fano manifold admitting a K\"ahler--Ricci soliton and a projective space of sufficiently large dimension.

%A similar idea was used in \cite{ALL25} to prove the existence of a Ricci-flat metric on the product of a Fano manifold admitting a K\"ahler--Ricci soliton with a projective space of sufficiently large dimension.

\begin{cor}
\label{cor-equiv-Hodge-index}
Assume Setup \ref{setup-soliton-situation}. Let $\mathcal{E}$ be a $\G$-equivariant torsion-free sheaf. 
If there exists $\xi\in\mathfrak{t}$ such that
$v(\mu)= e^{\mu_\xi}$, then
$$
(c_1^{\T}(\mathcal{E})^2 \cdot v(\alpha_{\T}))
\cdot 
( v(\alpha_{\T}))
\le(c_1^{\T}(\mathcal{E}) \cdot v(\alpha_{\T}))^2.
$$
\end{cor}

\begin{proof}
Since $X$ is compact, $P=\mu(X) \subset\mathfrak{t}^{\vee}$ is also compact. Thus there exists $R>0$ such that $\mu_\xi(X) \subset [-R, R]$. Set 
$$
v_m(\mu) \coloneqq \exp\left( \frac{R^2 + 1}{m}\right)
\cdot \left(1 + \frac{\mu_\xi}{m}\right)^m.
$$
Then \(v_m \colon P \to \R_{>0}\) and \(v_m > v\) for every \(m > 2R\). Moreover, \(v_m \to v\) in \(C^2(P)\).
Applying Theorem \ref{thm-equiv-Hodge-index} to $v_m$, we deduce that
$$
(c_1^{\T}(\mathcal{E})^2 \cdot v_m(\alpha_{\T}))
\cdot 
( v_m(\alpha_{\T}))
\le
\frac{(n+m-1)}{n+m } (c_1^{\T}(\mathcal{E}) \cdot v_m(\alpha_{\T}))^2
\le(c_1^{\T}(\mathcal{E}) \cdot v_m(\alpha_{\T}))^2.
$$
Letting $m \to \infty$ gives the desired inequality.
%Thus, by the argument in Lemma \ref{lem-openness-stab}, letting $m \to \infty$ gives the desired inequality.
\end{proof}

The following corollary follows by applying the same argument as in Theorem \ref{thm-HL24-6.1-torsionfree} to $v_m$ in Corollary \ref{cor-equiv-Hodge-index}, hence we omit the proof.
\begin{cor}
\label{cor-equiv-BG-soliton-case}
Assume Setup \ref{setup-soliton-situation}. Let $\mathcal{E}$ be a $\G$-equivariant torsion-free sheaf of rank $r$. 
Assume that there exists $\xi\in\mathfrak{t}$ such that
$v(\mu)= e^{\mu_\xi}$.
If $\mathcal{E}$ is $v$-semistable, then the Bogomolov--Gieseker inequality holds:
\[
%(\Delta^{\T}(E)\cdot v(\alpha_{\T}))_{\T}(\xi)=  
\left((2rc^{\T}_2(\mathcal{E}) - (r-1)c^{\T}_1(\mathcal{E})^2 ) \cdot v(\alpha_{\T}) \right)
\ge 0
\]
\end{cor}

\subsection{Equivariant Langer's inequality}

\begin{thm}[Equivariant Langer's inequality]
\label{thm-langer-ineq-equiv}
Assume Setup \ref{setup-soliton-situation}.
Then the following assertions hold for any $\G$-equivariant torsion-free sheaf $\mathcal{E}$ of rank $r$.
\begin{enumerate}
\item If there exist $\xi\in\mathfrak{t}$, $c, b \in \R$, and $m \in \N$ such that
$v(\mu)= (c  + b \mu_\xi)^{m}$ and $c  + b \mu_\xi>0$ on $X$, 
then
\begin{align*}
&\left((2rc^{\T}_2(\mathcal{E}) - (r-1)c^{\T}_1(\mathcal{E})^2 ) \cdot v(\alpha_{\T}) \right)
\ge -\frac{(n+m-1)}{n+m}\frac{r^2}{(v (\alpha_{\T}))} 
(\mu^{\max}_{v}(\mathcal{E}) - \mu_{v}(\mathcal{E}) )
 (\mu_{v}(\mathcal{E})  - \mu^{\min}_{v}(\mathcal{E})).
\end{align*}
\item If there exists $\xi\in\mathfrak{t}$ such that
$v(\mu) = e^{\mu_\xi}$
then
\begin{align*}
&\left((2rc^{\T}_2(\mathcal{E}) - (r-1)c^{\T}_1(\mathcal{E})^2 ) \cdot v(\alpha_{\T}) \right)
\ge - \frac{r^2}{(v (\alpha_{\T}))} 
(\mu^{\max}_{v}(\mathcal{E}) - \mu_{v}(\mathcal{E}) )
 (\mu_{v}(\mathcal{E})  - \mu^{\min}_{v}(\mathcal{E})).
\end{align*}
\end{enumerate}
\end{thm}

\begin{proof}
We prove only (1); the proof of (2) is similar. By Remark \ref{rem-soliton-GKP16a} \ref{prop-HN-existence}, take the Harder--Narasimhan filtration
\[
0=\mathcal{E}_0\subsetneq \mathcal{E}_1\subsetneq \cdots \subsetneq \mathcal{E}_{s}=\mathcal{E}
\]
such that each quotient sheaf
\(
\mathcal{G}_i \coloneqq \mathcal{E}_i/\mathcal{E}_{i-1}
\)
is torsion-free and $v$-semistable of rank $r_i$, and the sequence of slopes satisfies
\begin{equation}
\label{eq-HNfiltration-mimmax}
\mu_{v}^{\max}(\mathcal{E})=\mu_v(\mathcal{G}_1)
>\mu_v(\mathcal{G}_2) > \cdots > 
\mu_v(\mathcal{G}_{s}) =\mu_{v}^{\min}(\mathcal{E}).
\end{equation}
Then it follows that
\begin{equation}
\label{eq-Langer-ineq}
\begin{aligned}
 &\left(\left(2rc^{\T}_2(\mathcal{E}) - (r-1)c^{\T}_1(\mathcal{E})^2 \right) \cdot v(\alpha_{\T}) \right)
  \\
% &\underalign{\text{(Thm. \ref{thmdefn-chern-equivalence})}}{=} \left(\widetilde{\Psi}  (\Delta^{\G}(E)) \cdot v(\alpha_{\T}) \right) 
 %\\
 %&\underalign{\text{( Lem. \ref{lem-Langer-ineq-chow})}}{=}
 %\left( \widetilde{\Psi}
 %\left( 
%r\sum_{i=1}^{l} \frac{\Delta^{\G}(\mathcal{G}_{i})}{r_i} 
 %- \sum_{1 \le i < j \le l} r_i r_j \left( \frac{c^G_1(\mathcal{G}_{i})}{r_i} - \frac{c^G_1(\mathcal{G}_{j})}{r_j} \right)^2 
 %\right) \cdot  v(\alpha_{\T}) \right)
  %\\
  %&\underalign{\text{(Thm. \ref{thmdefn-chern-equivalence})}}{=}
 &\underalign{\text{\eqref{eq-semistable-stable-red}}}{=}
\sum_{i=1}^{s}\frac{r}{r_i}  \underbrace{ \left(2r_ic^{\T}_2(\mathcal{G}_i) - (r_i-1)c^{\T}_1(\mathcal{G}_i)^2 ) \cdot v(\alpha_{\T}) \right)}_{\ge 0 \text{ by  Thm. \ref{thm-HL24-6.1-torsionfree}}}
   +\sum_{1 \le i < j \le s} 
 \underbrace{-r_i r_j 
\left(\left( \frac{c^{\T}_1(\mathcal{G}_{i})}{r_i} - \frac{c^{\T}_1(\mathcal{G}_{j})}{r_j} \right)^2  \cdot  v(\alpha_{\T}) \right)}_{\ge -r_i r_j \frac{(n+m-1)}{n+m} \frac{(\mu_{v}(\mathcal{G}_i)-\mu_{v}(\mathcal{G}_j))^2}{(v(\alpha_{\T}))} \text{ by Thm. \ref{thm-equiv-Hodge-index}}}
\\
  &\underalign{}{\ge}
   -  \frac{(n+m-1)}{n+m} \frac{1}{(v(\alpha_{\T}) )}
\sum_{1 \le i < j \le s} r_i r_j (\mu_{v}(\mathcal{G}_i)-\mu_{v}(\mathcal{G}_j) )^2
\\
 %&\underalign{\cite[Lemma 1.4]{Lan04}, \cite[Lemma 3.2]{IMM24}}{\ge}
 &\underalign{\text{\cite[Lemma 1.4]{Lan04}}}{\ge}
   - \frac{(n+m-1)}{n+m} \frac{r^2}{(v(\alpha_{\T}) )}
   \underbrace{(\mu_v(\mathcal{G}_1)- \mu_v(\mathcal{E}))(\mu_v(\mathcal{E})-\mu_v(\mathcal{G}_{s}))}_{= (\mu^{\max}_{v}(\mathcal{E}) - \mu_{v}(\mathcal{E}) )
 (\mu_{v}(\mathcal{E})  - \mu^{\min}_{v}(\mathcal{E})) \text{ by \eqref{eq-HNfiltration-mimmax}}}
    %\\
    %&\underalign{}{=}  
    %- \frac{(n+m-1)}{n+m} \frac{r^2}{(v(\alpha_{\T}) )} (\mu^{\max}_{v}(\mathcal{E}) - \mu_{v}(\mathcal{E}) )
 %(\mu_{v}(\mathcal{E})  - \mu^{\min}_{v}(\mathcal{E})).
   \end{aligned}
   \end{equation}
Thus we obtain the desired inequality.
\end{proof}

\subsection{Outlook }
We record consequences of the weighted Hodge index and Bogomolov--Gieseker inequalities, assuming that they hold for a fixed weight $v$.

\begin{assup}
\label{assump-equiv-v-hodge}
Assume Setup \ref{setup-soliton-situation}, and fix a weight $v\colon P\to\R_{>0}$.
There exists a constant $C_v>0$ such that
\[
(c_1^\T(L)^2\cdot v(\alpha_\T))(v(\alpha_\T))
\leq C_v(c_1^\T(L)\cdot v(\alpha_\T))^2
\]
for every $\G$-linearized line bundle $L$ on $X$.
Moreover, if
\[
(c_1^\T(L)\cdot v(\alpha_\T))=0,
\qquad (c_1^\T(L)^2\cdot v(\alpha_\T))=0,
\]
then $c_1(L)=0$ in $H^2(X,\R)$.
\end{assup}

\begin{assup}
\label{assump-equiv-v-BG}
Assume Setup \ref{setup-soliton-situation}, and fix a weight $v\colon P\to\R_{>0}$.
For every $v$-stable $\G$-equivariant reflexive sheaf $\mathcal E$ of rank $r$, the Bogomolov--Gieseker inequality holds:
\[
\left((2rc_2^\T(\mathcal E)-(r-1)c_1^\T(\mathcal E)^2)\cdot v(\alpha_\T)\right)\geq0.
\]
Moreover, if equality holds, then $\mathcal E$ is locally free and projectively Hermitian flat.
\end{assup}

Under these assumptions, we obtain the following.
\begin{prop}
\label{prop-num-proj-flat}
Assume Setup \ref{setup-soliton-situation}. Let $v \colon P \to \R_{>0}$ be a weight. If Assumptions \ref{assump-equiv-v-hodge} and \ref{assump-equiv-v-BG} hold for the weight $v$, then the following statements hold.
\begin{enumerate}
\item If a torsion-free sheaf $\mathcal{E}$ of rank $r$ is $v$-semistable, then the Bogomolov--Gieseker inequality holds:
$$
 \left((2rc^{\T}_2(\mathcal{E}) - (r-1)c^{\T}_1(\mathcal{E})^2 ) \cdot v(\alpha_{\T}) \right) 
\ge 0
 $$
Moreover, if $\mathcal{E}$ is reflexive and the equality holds, then $\mathcal{E}$ is locally free and projectively numerically flat in the sense of \cite{LOY24}. In particular, the $\Q$-twisted sheaf $\mathcal{E} \langle \frac{\det \mathcal{E}^{\vee}}{r}\rangle$ is nef.
\item For any torsion-free sheaf $\mathcal{E}$ of rank $r$, Langer's inequality holds:
\begin{align*}
&\left((2rc^{\T}_2(\mathcal{E}) - (r-1)c^{\T}_1(\mathcal{E})^2 ) \cdot v(\alpha_{\T}) \right)
\ge - C_v \frac{r^2}{(v (\alpha_{\T}))} 
(\mu^{\max}_{v}(\mathcal{E}) - \mu_{v}(\mathcal{E}) )
 (\mu_{v}(\mathcal{E})  - \mu^{\min}_{v}(\mathcal{E})).
\end{align*}
Moreover, if $\mathcal{E}$ is reflexive and the equality holds, then $\mathcal{E}$ is locally free and the $v$-Harder--Narasimhan filtration of $\mathcal{E}$ has length at most $2$.
\end{enumerate}
\end{prop}
\begin{proof}
The proof is almost the same as that in \cite[Chapter 4, Theorem 4.1]{Nak04}.
(1)
By Remark \ref{rem-soliton-GKP16a}\ref{prop-JH-existence}, take a Jordan--Hölder filtration, that is, a filtration
\[
0=\mathcal{E}_0\subsetneq \mathcal{E}_1\subsetneq \cdots \subsetneq \mathcal{E}_{s}=\mathcal{E}
\]
such that each quotient sheaf
\(
\mathcal{G}_i \coloneqq \mathcal{E}_i/\mathcal{E}_{i-1}
\)
is a $v$-stable torsion-free sheaf of rank $r_i$ and satisfies
\(
\mu_{v}(\mathcal{G}_i)=\mu_{v}(\mathcal{E})
\).
By Assumption \ref{assump-equiv-v-hodge}, applied to $(\det\mathcal G_i)^{\otimes r_j}\otimes(\det\mathcal G_j)^{\otimes(-r_i)}$, 
\begin{equation}
\label{eq-JH-Hodge-index-2}
\left(\left( \frac{c^{\T}_1(\mathcal{G}_{i})}{r_i} - \frac{c^{\T}_1(\mathcal{G}_{j})}{r_j} \right)^2  \cdot  v(\alpha_{\T}) \right)
 \underalign{\text{(Assump. \ref{assump-equiv-v-hodge})}}{\le} \frac{C_v}{(v(\alpha_{\T}) )}
 \underbrace{\left( \left( \frac{c^{\T}_1(\mathcal{G}_{i})}{r_i} - \frac{c^{\T}_1(\mathcal{G}_{j})}{r_j} \right) \cdot  v(\alpha_{\T}) \right)^2}_{=(\mu_{v}(\mathcal{G}_i)-\mu_{v}(\mathcal{G}_j))^2}
 =0.
 %\frac{1}{\alpha_{\T}^2 \cdot v(\alpha_{\T}) )_{\T}}
 %\left( \mu_{v}(\mathcal{G}_i) - \mu_{v}(\mathcal{G}_j) \right)=0
\end{equation}

The support of $\mathcal{G}_i^{\vee\vee}/\mathcal{G}_i$ has codimension at least $2$. Since $\mathcal{G}_i^{\vee\vee}$ is also $v$-stable, 
\begin{equation}
\label{eq-Gi-reflexive-torsionfree}
\left((\ch_{1}^{\T}(\mathcal{G}_{i})^{2} - 2r_i \ch_{2}^{\T}(\mathcal{G}_{i}))\cdot v(\alpha_{\T}) \right)
%\underbrace{}_{=2r_ic^{\T}_2(\mathcal{G}_{i}) - (r_i-1)c^{\T}_1(\mathcal{G}_{i})^2\text{ by Def. \ref{defn-chern-Bogomolov}}}
\underalign{\text{(Prop. \ref{prop-torsion-2ndchern})}}{\ge} 
\left( (\ch_{1}^{\T}(\mathcal{G}^{\vee\vee}_{i})^{2} - 2r_i \ch_{2}^{\T}(\mathcal{G}^{\vee\vee}_{i}))\cdot v(\alpha_{\T}) \right)
\underalign{\text{(Assump. \ref{assump-equiv-v-BG})}}{\ge} 0.
\end{equation}
Therefore we obtain
%Therefore, noting that $\ch_{1}^{\T}(\mathcal{G}_{i})^{2} - 2r_i \ch_{2}^{\T}(\mathcal{G}_{i})=2r_ic^{\T}_2(\mathcal{G}_{i}) - (r_i-1)c^{\T}_1(\mathcal{G}_{i})^2$, we obtain
\begin{align*}
 &((2rc^{\T}_2(\mathcal{E}) - (r-1)c^{\T}_1(\mathcal{E})^2 )\cdot v(\alpha_{\T})) 
 \\
 %&\underalign{\text{(Thm. \ref{thmdefn-chern-equivalence})}}{=}(\widetilde{\Psi} (\Delta^{\G}(E)) \cdot v(\alpha_{\T}) ) 
% \\
% &\underalign{\text{(Lem. \ref{lem-Langer-ineq-chow})}}{=}
%\left( 
% \widetilde{\Psi}\left( 
%r\sum_{i=1}^{l} \frac{\Delta^{\G}(\mathcal{G}_{i})}{r_i} 
 %- \sum_{1 \le i < j \le l} r_i r_j \left( \frac{c^G_1(\mathcal{G}_{i})}{r_i} - \frac{c^G_1(\mathcal{G}_{j})}{r_j} \right)^2 
 %\right) \cdot  v(\alpha_{\T}) )_{\T}  \right)
 %\\
 % &\underalign{\text{(Thm. \ref{thmdefn-chern-equivalence})}}{=}
  &\underalign{\text{as in \eqref{eq-semistable-stable-red}}}{=}
\sum_{i=1}^{s}\frac{r}{r_i}  \underbrace{ \left( (2r_i c^{\T}_2(\mathcal{G}_{i}) - (r_i-1)c^{\T}_1(\mathcal{G}_{i})^2 )
\cdot v(\alpha_{\T}) \right)}_{\ge 0 \text{ by \eqref{eq-Gi-reflexive-torsionfree}}}
   + \sum_{1 \le i < j \le s}
   \underbrace{- r_i r_j \left(\left( \frac{c^{\T}_1(\mathcal{G}_{i})}{r_i} - \frac{c^{\T}_1(\mathcal{G}_{j})}{r_j} \right)^2  \cdot  v(\alpha_{\T}) \right)}_{\ge 0 \text{ by \eqref{eq-JH-Hodge-index-2}}}\\
   & \underalign{}{\ge} 0.
\end{align*}

We now consider the equality case. Suppose that equality holds. Then the following statements hold.
\begin{enumerate}
\item[$(\alpha)$] %Since $\left( (2r_i c^{\T}_2(\mathcal{G}_{i}) - (r_i-1)c^{\T}_1(\mathcal{G}_{i})^2 )\cdot v(\alpha_{\T}) \right)=0$, 
By \eqref{eq-Gi-reflexive-torsionfree}, 
\(
\left( (2r_i c^{\T}_2(\mathcal{G}_{i}^{\vee\vee}) - (r_i-1)c^{\T}_1(\mathcal{G}_{i}^{\vee\vee})^2 )
\cdot v(\alpha_{\T}) \right)=0
\)
holds. Hence Assumption \ref{assump-equiv-v-BG} implies that $\mathcal{G}^{\vee\vee}_{i}$ is locally free and projectively Hermitian flat.
\item[$(\beta)$]  By Proposition \ref{prop-torsion-2ndchern}, equality in \eqref{eq-Gi-reflexive-torsionfree} implies that the support of $\mathcal{G}_{i}^{\vee\vee}/\mathcal{G}_{i}$ has codimension at least $3$. In particular, $\mathcal{G}_i$ is locally free outside a subset of codimension at least $3$. 
%By Proposition \ref{prop-torsion-2ndchern}, if equality holds in \eqref{eq-Gi-reflexive-torsionfree}, then the support of $\mathcal{G}_i^{\vee\vee}/\mathcal{G}_i$ has codimension at least $3$. In particular, $\mathcal{G}_i$ is locally free outside a subset of codimension at least $3$.
\item[$(\gamma)$]  
Since $\mu_v(\mathcal G_i)=\mu_v(\mathcal G_j)$ and equality holds in \eqref{eq-JH-Hodge-index-2}, the line bundle $(\det\mathcal G_i)^{\otimes r_j}\otimes(\det\mathcal G_j)^{\otimes(-r_i)}$ has zero weighted degree and zero weighted square.
Assumption \ref{assump-equiv-v-hodge} therefore gives
\[
\frac{c_1(\mathcal G_i)}{r_i}=\frac{c_1(\mathcal G_j)}{r_j}=\frac{c_1(\mathcal E)}{r}
\quad\text{in }H^2(X,\R).
\]

\end{enumerate}
Since $\mathcal{E}_1$ is reflexive, by $(\alpha)$, $\mathcal{E}_1$ is locally free. Hence, by $(\alpha)$ and $(\beta)$, we can apply \cite[Lemma 9.9]{AD14} to the exact sequence
\(
0 \to \mathcal{E}_1 \to \mathcal{E}_2 \to\mathcal{G}_2 \to 0,
\)
and it follows that both $\mathcal{E}_2$ and $\mathcal{G}_2$ are locally free. Repeating this argument, all $\mathcal{G}_i$ and $\mathcal{E}_i$ are locally free. In particular, $\mathcal{E}=\mathcal{E}_{s}$ is locally free.
Moreover, by $(\gamma)$ and \cite[Theorem 4.2]{LOY24}, it is numerically projectively flat.

(2)
By Remark \ref{rem-soliton-GKP16a} \ref{prop-HN-existence}, take the Harder--Narasimhan filtration
\[
0=\mathcal{E}_0\subsetneq \mathcal{E}_1\subsetneq \cdots \subsetneq \mathcal{E}_s=\mathcal{E}
\]
such that each quotient sheaf
\(
\mathcal{G}_i \coloneqq \mathcal{E}_i/\mathcal{E}_{i-1}
\)
is torsion-free and $v$-semistable of rank $r_i$, and the sequence of slopes satisfies
\begin{equation}
\label{eq-HN-slopeineq}
\mu_{v}^{\max}(\mathcal{E})=\mu_v(\mathcal{G}_1)
>\mu_v(\mathcal{G}_2) > \cdots > 
\mu_v(\mathcal{G}_{s}) =\mu_{v}^{\min}(\mathcal{E}).
\end{equation}
%Here we used Proposition \ref{prop-min-slope} for the last equality.
Thus we obtain
\begin{align*}
  &\left((2rc^{\T}_2(\mathcal{E}) - (r-1)c^{\T}_1(\mathcal{E})^2 )\cdot v(\alpha_{\T})\right) \\
% &\underalign{\text{(Thm. \ref{thmdefn-chern-equivalence})}}{=}(\widetilde{\Psi}(\Delta^{\G}(\mathcal{E})) \cdot v(\alpha_{\T}) ) \\
% &\underalign{\text{(Lem. \ref{lem-Langer-ineq-chow})}}{=}
%\left( \widetilde{\Psi}\left( 
%r\sum_{i=1}^{l} \frac{\Delta^{\G}(\mathcal{G}_{i})}{r_i} 
% - \sum_{1 \le i < j \le l} r_i r_j \left( \frac{c^G_1(\mathcal{G}_{i})}{r_i} - \frac{c^G_1(\mathcal{G}_{j})}{r_j} \right)^2 
% \right) \cdot  v(\alpha_{\T}) \right)\\
%  &\underalign{\text{(Thm. \ref{thmdefn-chern-equivalence})}}{=}
  &\underalign{\text{\eqref{eq-semistable-stable-red}}}{=}
\sum_{i=1}^{s}\frac{r}{r_i}  
\underbrace{ \left( (2r_i c^{\T}_2(\mathcal{G}_{i}) - (r_i-1)c^{\T}_1(\mathcal{G}_{i})^2 )
\cdot v(\alpha_{\T}) \right)}_{\ge 0 \text{ by (1)}}
   + \sum_{1 \le i < j \le s} 
    \underbrace{-r_i r_j 
\left(\left( \frac{c^{\T}_1(\mathcal{G}_{i})}{r_i} - \frac{c^{\T}_1(\mathcal{G}_{j})}{r_j} \right)^2  \cdot  v(\alpha_{\T}) \right)}_{\ge - C_v r_i r_j \frac{(\mu_{v}(\mathcal{G}_i)-\mu_{v}(\mathcal{G}_j))^2}{(v(\alpha_{\T}) )} \text{ by Assump. \ref{assump-equiv-v-hodge}}}\\
  &\underalign{}{\ge}
   -C_v \frac{1}{(v(\alpha_{\T}) )} 
\sum_{1 \le i < j \le s} r_i r_j (\mu_{v}(\mathcal{G}_i)-\mu_{v}(\mathcal{G}_j) )^2  
\\ &
\underalign{\text{\cite[Lem. 1.4]{Lan04}}}{\ge}
   - C_v   \frac{r^2}{(v(\alpha_{\T}) )} 
  \underbrace{ (\mu_v(\mathcal{G}_1)- \mu_v(\mathcal{E}))(\mu_v(\mathcal{E})-\mu_v(\mathcal{G}_{s}))}
  _{=(\mu^{\max}_{v}(\mathcal{E}) - \mu_{v}(\mathcal{E}) )(\mu_{v}(\mathcal{E})  - \mu^{\min}_{v}(\mathcal{E})) \text{ by \eqref{eq-HN-slopeineq}}}.
  %  \\
  %  &\underalign{}{=}  
   % -C_v   \frac{r^2}{(v(\alpha_{\T}) )}  (\mu^{\max}_{\alpha}(\mathcal{E}) - \mu_{v}(\mathcal{E}) )
 %(\mu_{v}(\mathcal{E})  - \mu^{\min}_{\alpha}(\mathcal{E})).
   \end{align*}
We now consider the equality case. Suppose that equality holds. Then the following statements hold.
\begin{enumerate}
\item[$(\alpha)$]  Since $ \left( (2r_i c^{\T}_2(\mathcal{G}_{i}^{\vee\vee}) - (r_i-1)c^{\T}_1(\mathcal{G}_{i}^{\vee\vee})^2 )
\cdot v(\alpha_{\T}) \right)=0$, $\mathcal{G}^{\vee\vee}_{i}$ is locally free and numerically projectively flat by (1).
\item[$(\beta)$]  By equality in \eqref{eq-Gi-reflexive-torsionfree}, together with (1), $\mathcal{G}_i$ is locally free outside a subset of codimension at least $3$.
\item[$(\gamma)$]  By \cite[Lemma 3.2]{IMM24}, which analyzes the equality case of \cite[Lem. 1.4]{Lan04}, we have $s \le 2$. In other words, the $v$-Harder--Narasimhan filtration has length at most $2$.
\end{enumerate}
Hence, by the same argument as in (1), we deduce that $\mathcal{E}$ is locally free.
%In particular, by $(\alpha)$ and $(\beta)$, $\mathcal{E}_i$ is reflexive. Hence, by \cite[Lemma 9.9]{AD14}, the sheaves $\mathcal{G}_i$ and $\mathcal{E}_i$ are locally free. Therefore $\mathcal{E}$ is locally free.
\end{proof}

It is also natural to ask whether the arguments developed so far remain valid in the compact K\"ahler setting or for singular varieties. 
Note that, in the compact K\"ahler case, one cannot necessarily take a finite locally free resolution. In the usual compact K\"ahler case, the existence of Chern classes of coherent sheaves follows from \cite{Gri10}. 
When $X$ is singular, especially when $X$ has klt singularities, it seems that only the first and second Chern classes can be defined and that one needs to consider orbifold Chern classes.

\section{Equivariant Miyaoka--Yau inequality for  K\"ahler--Ricci solitons}
\label{sec-proof-soliton-MY}
\subsection{Slope estimate of the canonical extension sheaf}

We first define the $\T$-equivariant canonical extension sheaf
(cf.~\cite[Section 2]{Tian92}, \cite[Section 7.1]{HL24}).
Let $\T$ be a compact torus, and let $X$ be an $n$-dimensional K\"ahler manifold with a holomorphic $\T$-action.
Fix a moment map $\mu\colon X\to\mathfrak t^\vee$, and set $P\coloneqq\mu(X)$.
Let $v\in C^\infty(P,\R_{>0})$ be any weight.

Equip $T_X$ with its natural $\T$-linearization and $\mathcal O_X$ with the trivial linearization.
For any $\lambda>0$, the  $\T$-invariant class
\[
\lambda c_1(X)\in\operatorname{Ext}^1(T_X,\mathcal O_X)\simeq H^1(X,\Omega_X^1)
\]
defines a $\T$-equivariant holomorphic extension
\begin{equation}
\label{eq-defn-canonical-extension}
0\longrightarrow\mathcal O_X
\longrightarrow\mathscr V_\lambda
\longrightarrow T_X
\longrightarrow0
\end{equation}
by the construction in \cite[Section 7.1]{HL24}.

\begin{defn}
\label{defn-canonical-extension-equiv}
We call the locally free sheaf $\mathscr V_\lambda$ the \emph{$\T$-equivariant canonical extension sheaf with extension class $\lambda c_1(X)$}.
With this linearization, we have
\[
c_1^\T(\mathscr V_\lambda)=c_1^\T(X)
\quad \text{and} \quad
c_2^\T(\mathscr V_\lambda)=c_2^\T(X)
\]
We often simply write $\mathscr V$ for $\mathscr V_\lambda$.
\end{defn}

Next, suppose that $X$ is Fano and $\omega\in c_1(X)$ is a $\T$-invariant K\"ahler form.
Fix a smooth equivariant splitting $\mathscr V\underset{C^\infty}{\simeq}T_X\oplus\mathcal O_X$ with respect to which the holomorphic structure is given by
\[
\bar\partial_{\mathscr V}
=
\begin{pmatrix}
\bar\partial_{T_X}&0\\
\lambda\Psi^\omega&\bar\partial_{\mathcal O_X}
\end{pmatrix}.
\]
Here $\Psi^\omega\in\Omega^{0,1}(\Hom(T_X,\mathcal O_X))$ is defined by
\[
\Psi^\omega_{\overline U}(s)\coloneqq\omega(s,\overline U)
\qquad
(\forall s\in T_X,\ \overline U\in T^{0,1}X).
\]
Let $h_1,h_2$ be $\T$-invariant Hermitian metrics on $T_X,\mathcal O_X$, respectively, and set
\[
h_{\mathscr V}\coloneqq h_1\oplus h_2
\]
with respect to this splitting.
Explicit formulas for the Chern connection $\nabla^{h_{\mathscr V}}$, the Chern curvature $F_{h_{\mathscr V}}$, and the moment map $\Phi_{h_{\mathscr V}}$ are given in \cite[Section 7.1, (40) and (41)]{HL24}.

We now consider a simpler case.
Let $h_\omega$ be the Hermitian metric on $T_X$ induced by $\omega$.
Suppose that $h_1=e^fh_\omega$ and $h_2=e^u$ for $\T$-invariant smooth real-valued functions $f,u$.
Using the explicit formulas in \cite[Section 7.1, (40) and (41)]{HL24}, we obtain
\begin{equation}
\label{eq-canonical-extension-weighted-curvature}
\begin{aligned}
&\frac{\sqrt{-1}}{2\pi}
\Lambda_{\omega,v}(F_{h_{\mathscr V}}+\Phi_{h_{\mathscr V}})\\
&=\frac{1}{2\pi}
\begin{pmatrix}
\sqrt{-1}\Lambda_{\omega,v}(F_{h_\omega}+\Phi_{h_\omega})
+\left(\Delta_v f-\lambda^2v(\mu)e^{u-f}\right)\Id_{T_X}
&
-\sqrt{-1}\lambda v(\mu)
\bigl(\partial(u-f+\log v(\mu))\bigr)^\dagger
\\
\sqrt{-1}\lambda v(\mu)\partial(u-f+\log v(\mu))
&
\Delta_v u+n\lambda^2v(\mu)e^{u-f}
\end{pmatrix}.
\end{aligned}
\end{equation}
Here $\Delta_v g\coloneqq\partial^*(v(\mu)\partial g)$ is the weighted Laplacian used in \cite[Lemma~2.6]{HL24}, and $\partial^*$ is the formal adjoint with respect to $\omega$.
We regard $\partial(u-f+\log v(\mu))$ as a section of $\Hom(T_X,\mathcal O_X)$, and ${}^\dagger$ denotes the adjoint with respect to $h_1,h_2$. 
Indeed, the Chern connection and its curvature are  
\begin{equation*}
\nabla^{h_{\mathscr V}} =
\begin{pmatrix}
\nabla^{h_1} & -(\lambda \Psi^{\omega})^{\dagger} \\
(\lambda \Psi^{\omega}) & \nabla^{h_2} 
\end{pmatrix}, \qquad 
F^{h_{\mathscr V}} =
\begin{pmatrix}
F^{h_1} -(\lambda \Psi^{\omega})^{\dagger} \wedge (\lambda \Psi^{\omega}) & -D(\lambda \Psi^{\omega})^{\dagger} \\
D(\lambda \Psi^{\omega}) & F^{h_2} - (\lambda \Psi^{\omega}) \wedge (\lambda \Psi^{\omega})^{\dagger}
\end{pmatrix}. 
\end{equation*}
Here $D$ denotes the induced connection on the Hom bundle. 
The moment endomorphism defined in 
\eqref{eq-canonically-moment-map} 
therefore takes the form 
\begin{equation}\label{eq-moment-extension}
\langle \Phi_{h_{\mathscr V}}, \xi  \rangle = 
\begin{pmatrix}
\langle \Phi_{h_1}, \xi \rangle  & -(\lambda \Psi^{\omega})^{\dagger}(\xi_X) \\ 
(\lambda \Psi^{\omega})(\xi_X) & \langle \Phi_{h_2}, \xi \rangle
\end{pmatrix}. 
\end{equation}
Since $\Psi^\omega$ has components 
$\sqrt{-1}d\bar{z}^i$ in a unitary coframe, its adjoint has components 
$-\sqrt{-1}e^{u-f}dz^i$. It follows 
\begin{equation*}
\Lambda_{\omega} ((\Psi^\omega)^{\dagger} \wedge \Psi^\omega) = -\sqrt{-1}e^{u-f} \Id_{T_X}, \qquad 
\Lambda_{\omega}( \Psi^\omega  \wedge (\Psi^\omega)^{\dagger} )= n\sqrt{-1}e^{u-f}.  
\end{equation*}
On the other hand, 
conformal change $h \mapsto e^f h$ adds 
\begin{equation*}
\frac{1}{2\pi} e^{-\mu_\xi}\Delta_v f \Id_E 
\end{equation*} 
to $\frac{\sqrt{-1}}{2\pi}(\Lambda_\omega F_h +\langle \Phi_h^\xi, \xi \rangle)$. 
Finally, $D \Psi^{\omega} $ is contracted to 
$\partial (u-f)$ and the off-diagonal 
entry $\Psi^\omega(\xi_X)$ in \eqref{eq-moment-extension} contributes  $\partial \mu_\xi$.

Second, we recall the definition of weighted Ricci curvature and $v$-soliton.
% according to \cite[Section 1]{HL23} and \cite[Section 1]{DJ25}.
\begin{defn}[{\cite[Section 1]{HL23}, \cite[Section 1]{DJ25}}]
\label{defn-vsoliton}
Let $\T=(S^1)^{l}$ be a compact torus with Lie algebra $\mathfrak{t}$. Let $X$ be a compact K\"ahler manifold equipped with a holomorphic $\T$-action, and let $\omega$ be  a $\T$-invariant K\"ahler form. We assume that this $\T$-action is Hamiltonian. Denote by $\mu \colon X \to \mathfrak{t}^{\vee}$ the moment map and by $P \coloneqq \mu(X)$ the moment polytope.

For any positive $C^{\infty}$ function $v \colon P\to \mathbb{R}_{>0}$, we define the \emph{weighted Ricci curvature} by
\[
\Ric_{v}\omega \; \coloneqq \; \Ric\omega-\frac{\sqrt{-1}}{2 \pi}\,\partial\bar\partial \log(v\circ \mu)
%\underalign{\frac{\sqrt{-1}}{2\pi} \partial \bar{\partial}=\frac{\sqrt{-1}}{2 \pi}\,\partial\bar\partial }{=}\Ric\omega- \frac{\sqrt{-1}}{2\pi} \partial \bar{\partial} \log(v\circ \mu).
\]
We say that a K\"ahler metric $\omega$ is a \emph{$v$-soliton} if it satisfies $\Ric_{v}\omega = \omega$.
\end{defn}
If there exists $\xi \in \mathfrak{t}$ such that $v(\mu)=e^{\mu_\xi}$, then $\omega$ is a $v$-soliton if and only if $\omega$ is a K\"ahler--Ricci soliton with holomorphic vector field $\xi_X$ induced by $\xi$.

\begin{lem}\label{lem-kappa} 
  Define a function 
\begin{equation*}
\kappa\coloneqq
\frac{1}{2\pi}e^{-\mu_\xi}\Delta_v\mu_\xi
=
\frac{1}{2\pi}e^{-\mu_\xi}\partial^*(e^{\mu_\xi}\partial\mu_\xi).
\end{equation*} 
It satisfies the moment map identity and the normalization
  \begin{equation}\label{eq-kappa}
d\kappa = -\iota_{\xi_X} \Ric_v{\omega}
\quad \text{and} \quad 
\int_X \kappa e^{\mu_\xi}\omega^{[n]}=0.
  \end{equation}
\end{lem}
\begin{proof}
Consider the vector field $Y:= \grad_\omega^{1, 0}\mu_\xi$. 
The moment map equation gives
$\xi_X^{1,0}=\sqrt{-1}Y$. 
In particular, $Y$ is holomorphic. 
The local formula for $\partial^*$ and the K\"ahler condition give
\begin{align*}
2\pi\kappa
&=e^{-\mu_\xi}\partial^*
       (e^{\mu_\xi}\partial\mu_\xi)
=\partial^*\partial\mu_\xi
  -g^{i\bar j}
     \frac{\partial\mu_\xi}{\partial z^i}
     \frac{\partial\mu_\xi}{\partial\bar z^j}
     =-\left(
     \frac{\partial Y^i}{\partial z^i}
     +Y^i\frac{\partial}{\partial z^i}\log\det(g_{k\bar \ell})
     +Y^i\frac{\partial\mu_\xi}{\partial z^i}
   \right).
\end{align*}
Differentiating and using
$\bar\partial Y=0$, we obtain
\begin{align*}
\bar\partial\kappa
&=-\frac{1}{2\pi}
  Y^i\frac{\partial^2}{\partial z^i\partial\bar z^j}
  \bigl(\log\det(g_{k\bar \ell})+\mu_\xi\bigr)\,d\bar z^j.
\end{align*}
On the other hand,
\begin{equation*}
\Ric_v\omega
=-\frac{\sqrt{-1}}{2\pi}
  \frac{\partial^2}{\partial z^i\partial\bar z^j}
  \bigl(\log\det(g_{k\bar \ell})+\mu_\xi\bigr)
  \,dz^i\wedge d\bar z^j.
\end{equation*}
Using $\xi_X^{1,0}=\sqrt{-1}Y$ again,
we obtain $
\bigl(-\iota_{\xi_X}\Ric_v\omega\bigr)^{0,1}
=\bar\partial\kappa$. 
Since both $d\kappa$ and $-\iota_{\xi_X}\Ric_v\omega$ are real,
their $(1,0)$-parts also agree. 
The normalization can be checked directly: 
\begin{align*}
\int_X\kappa e^{\mu_\xi}\omega^{[n]}
=\frac{1}{2\pi}
  \int_X\partial^*(e^{\mu_\xi}\partial\mu_\xi)\,\omega^{[n]}
=0.
\end{align*}
\end{proof}

In \cite[right after equation (47)]{HL24} they claim  
$\mu_\xi =e^{-\mu_\xi}\Delta_v\mu_\xi $, 
however, it in fact follows from  K\"ahler-Ricci soliton condition. 
We need to distinguish both sides in general.

In the following theorem, we reprove \cite[Theorem~7.2]{HL24} with a correction term.

\begin{thm}[{cf.~\cite[Theorems 7.1 and 7.2]{HL24}}]
\label{thm-slope-equiv-ineq}
As in the setting and notation in Definition \ref{defn-vsoliton}.
Let $X$ be a Fano $\T$-manifold of dimension $n\geq2$, and choose the moment map $\mu$ so that $\{\omega+\mu\}=c_1^{\T}(X)$.
Assume that the weight $v$ is of the form $v(\mu)=e^{\mu_\xi}$ for some $\xi\in\mathfrak t$.
Suppose further that
\begin{equation}
\label{eq-canonical-normalized-mu}
\int_X \mu_\xi e^{\mu_\xi}\omega^{[n]}=0.
\end{equation}
This is equivalent to $v$ being canonically normalized in the sense of \eqref{eq-defn-canonically-normalized}.

Let $0<\gamma\leq1$, and assume that $\Ric_v\omega\geq\gamma\omega$.
Set
\begin{equation}
\label{eq-defn-nu}
\nu\coloneqq
\frac{\gamma}{n+1}
\frac{(e^{c_1^{\T}(X)})_{\T}(\xi)}
     {\int_X\omega^{[n]}},
\end{equation}
and let $\mathscr V$ be the $\T$-equivariant canonical extension sheaf with extension class $\sqrt{2\pi\nu}\,c_1(X)$.

Then, for any nonzero $\T$-equivariant torsion-free quotient $\mathscr V\twoheadrightarrow\mathcal Q$ of rank $q$, we have
\begin{equation}
\label{eq:quotientbound}
(c_1^{\T}(\mathcal Q)\cdot e^{c_1^{\T}(X)})_{\T}(\xi)
\geq
\frac{\gamma q}{n+1}(c_1^{\T}(X)\cdot e^{c_1^{\T}(X)})_{\T}(\xi)
-
(e^{c_1^{\T}(X)})_{\T}(\xi)
\left[\gamma\ell_\xi-(1-\gamma)\frac{ab}{b-a}\right].
\end{equation}
Here we set
\begin{equation}
\label{eq-defn-lxi}
\ell_\xi\coloneqq
\frac{1}{2(e^{c_1^{\T}(X)})_{\T}(\xi)}
\int_X
\left|
\frac{(e^{c_1^{\T}(X)})_{\T}(\xi)}
     {\int_X\omega^{[n]}}
e^{-\mu_\xi}-1-\mu_\xi
\right|
e^{\mu_\xi}\omega^{[n]},
\end{equation}
and $a\coloneqq\min_X\mu_\xi$, $b\coloneqq\max_X\mu_\xi$.
If $\mu_\xi\equiv0$, we interpret $\frac{-ab}{b-a}$ as $0$ since $a=b=0$.
\end{thm}

In the K\"ahler--Ricci soliton case, setting $\gamma=1$ in Theorem \ref{thm-slope-equiv-ineq} gives
\[
(c_1^\T(\mathcal Q)\cdot e^{c_1^\T(X)})_\T(\xi)
\geq\frac{q}{n+1}(c_1^\T(X)\cdot e^{c_1^\T(X)})_\T(\xi)
-(e^{c_1^\T(X)})_\T(\xi)\ell_\xi.
\]
In Example \ref{ex-unstable-canonical-extension} below, we show that the canonical extension sheaf need not be $v$-semistable even when a K\"ahler Ricci soliton exists. Thus the correction term $-(e^{c_1^\T(X)})_\T(\xi)\ell_\xi$ cannot be omitted in general.
In particular, this example gives an explicit counterexample to \cite[Theorem 7.2]{HL24}.

\begin{proof}[Proof of Theorem \ref{thm-slope-equiv-ineq}]
Let $h_\omega$ be the Hermitian metric on $T_X$ induced by $\omega$. 
In the construction of Definition \ref{defn-canonical-extension-equiv}, equip $\mathscr V$ with the metric
\(
h_{\mathscr V}\coloneqq e^{2\mu_\xi}h_\omega\oplus e^{\mu_\xi}.
\)
Substituting $f=2\mu_\xi$, $u=\mu_\xi$, and $\lambda=\sqrt{2\pi\nu}$ into \eqref{eq-canonical-extension-weighted-curvature} and multiplying by $e^{-\mu_\xi}$, we obtain
\begin{equation}
\label{eq:quotient-block-curvature}
\frac{\sqrt{-1}}{2\pi}
\left(\Lambda_\omega F_{h_{\mathscr V}}
+\langle\Phi_{h_{\mathscr V}},\xi\rangle\right)
\underalign{
\text{\eqref{eq-canonical-extension-weighted-curvature}}
}{=}
\begin{pmatrix}
\sharp^\omega(\Ric_v\omega)
+(2\kappa-\nu e^{-\mu_\xi})\Id_{T_X}&0\\
0&\kappa+n\nu e^{-\mu_\xi}
\end{pmatrix}.
\end{equation}
Here, for a real $(1,1)$-form $\eta$, the endomorphism $\sharp^\omega(\eta)\in\End(T_X)$ is defined by
\[
\omega\bigl(\sharp^\omega(\eta)s,\overline t\bigr)
=\eta(s,\overline t)
\qquad( \forall s,t\in T_X).
\]
We also use \cite[Section 3.3]{HL24} to compute the upper-left entry.

Let $X^\circ\subset X$ be an open subset such that $\mathcal Q|_{X^\circ}$ is locally free and $\operatorname{codim}_X(X\setminus X^\circ)\geq2$.
We identify $\mathcal Q$ smoothly with $(\ker(\mathscr V\to\mathcal Q))^\perp$ in $X^\circ$.
Let $P_Q\colon\mathscr V|_{X^\circ}\to\bigl(\ker(\mathscr V\to\mathcal Q)|_{X^\circ}\bigr)^\perp$ and $P_0\colon\mathscr V|_{X^\circ}\to\mathcal O_X|_{X^\circ}$ be the orthogonal projections with respect to $h_{\mathscr V}$.
Denoting the second fundamental form by $A$, we obtain from \cite[Lemma~5.5]{HL24} and $\tr P_Q=q$ that
\begin{align*}
&(c_1^{\T}(\mathcal Q)\cdot e^{c_1^{\T}(X)})_{\T}(\xi)
\underalign{\raisebox{-1.0ex}{\tiny \text{\cite[Lem. 5.5]{HL24}}}}{=}
\frac{\sqrt{-1}}{2\pi}
\int_{X^\circ}\tr\!\left[P_Q\left(\Lambda_\omega F_{h_{\mathscr V}}
+\langle\Phi_{h_{\mathscr V}},\xi\rangle\right)\right]
e^{\mu_\xi}\omega^{[n]}
+\frac{1}{2\pi}\int_{X^\circ}|A|^2e^{\mu_\xi}\omega^{[n]}
\\
&
\underalign{\text{\eqref{eq:quotient-block-curvature}}}{=}
\underbrace{q\int_{X^\circ}(\gamma+2\kappa-\nu e^{-\mu_\xi})e^{\mu_\xi}\omega^{[n]}}
_{=\frac{\gamma q}{n+1}(c_1^{\T}(X)\cdot e^{c_1^{\T}(X)})_{\T}(\xi) \text{ by \eqref{eq-kappa} \& \eqref{eq-defn-nu}}}
+
\int_{X^\circ}\bigl((n+1)\nu e^{-\mu_\xi}-\gamma-\kappa\bigr)
\tr(P_QP_0)e^{\mu_\xi}\omega^{[n]}\\
&\quad+
\underbrace{\int_{X^\circ}\tr\!\left[
P_Q\begin{pmatrix}
\sharp^\omega(\Ric_v\omega-\gamma\omega)&0\\
0&0
\end{pmatrix}
\right]e^{\mu_\xi}\omega^{[n]}}_{\geq0 \text{ by } \Ric_v\omega\geq\gamma\omega}
+\underbrace{\frac{1}{2\pi}\int_{X^\circ}|A|^2e^{\mu_\xi}\omega^{[n]}}_{\geq0}.
\end{align*}

Thus it remains to show that
\begin{equation}
\label{eq-solition-eval-ETS}
\int_{X^\circ}\bigl((n+1)\nu e^{-\mu_\xi}-\gamma-\kappa\bigr)
\tr(P_QP_0)e^{\mu_\xi}\omega^{[n]}
\geq
-
(e^{c_1^{\T}(X)})_{\T}(\xi)
\left[\gamma\ell_\xi-(1-\gamma)\frac{ab}{b-a}\right].
\end{equation}
If $\mu_\xi\equiv0$, then $\kappa=\ell_\xi=0$ and $\nu=\gamma/(n+1)$.
Both sides of \eqref{eq-solition-eval-ETS} are therefore zero.
Henceforth, assume that $\mu_\xi\not\equiv0$, so that $a<0<b$.
Set $p\coloneqq\tr(P_QP_0)$, so that $0\leq p\leq1$.
Also set
\[
f\coloneqq\frac{(e^{c_1^{\T}(X)})_{\T}(\xi)}{\int_X\omega^{[n]}}e^{-\mu_\xi}-1-\mu_\xi.
\]
By \eqref{eq-canonical-normalized-mu}, we have $\int_Xf e^{\mu_\xi}\omega^{[n]}=0$.
Since $fp\geq(f-|f|)/2$, it follows that
\begin{equation}
\label{eq:quotient-f-bound}
\int_{X^\circ}fp\,e^{\mu_\xi}\omega^{[n]}
\geq-\frac12\int_X|f|e^{\mu_\xi}\omega^{[n]}
\underalign{\text{\eqref{eq-defn-lxi}}}{=}
-(e^{c_1^{\T}(X)})_{\T}(\xi)\ell_\xi.
\end{equation}

\begin{claim}
\label{clm:quotient-kappa-H}
\begin{enumerate}
\item If $\gamma=1$, then $\kappa=\mu_\xi$.
\item Suppose that $0<\gamma<1$, and set $H_\gamma\coloneqq(\kappa-\gamma\mu_\xi)/(1-\gamma)$. Then $a\leq H_\gamma\leq b$.
\end{enumerate}
\end{claim}
We postpone the proof of this claim and continue the proof of the theorem.
If $\gamma=1$, then Claim \ref{clm:quotient-kappa-H} (1) and the definition of $\nu$ give $(n+1)\nu e^{-\mu_\xi}-1-\kappa=f$.
Thus \eqref{eq-solition-eval-ETS} follows from \eqref{eq:quotient-f-bound}.

If $0<\gamma<1$, then
\begin{equation}
\label{eq:quotient-H-pointwise}
H_\gamma p
\underalign{\text{(by $0\leq p\leq1$)}}{\leq}
\max\{H_\gamma,0\}
\underalign{\text{(Claim \ref{clm:quotient-kappa-H})}}{\leq}\frac{b(H_\gamma-a)}{b-a}.
\end{equation}
Consequently,
\begin{align*}
&\int_{X^\circ}\bigl((n+1)\nu e^{-\mu_\xi}-\gamma-\kappa\bigr)\tr(P_QP_0)e^{\mu_\xi}\omega^{[n]}
=\gamma\int_{X^\circ}fp\,e^{\mu_\xi}\omega^{[n]}-(1-\gamma)\int_{X^\circ}H_\gamma p\,e^{\mu_\xi}\omega^{[n]}\\
&\underalign{\text{\eqref{eq:quotient-f-bound} \& \eqref{eq:quotient-H-pointwise}}}{\geq}
-\gamma(e^{c_1^{\T}(X)})_{\T}(\xi)\ell_\xi
-
\frac{(1-\gamma)b}{b-a}
\underbrace{\int_XH_\gamma e^{\mu_\xi}\omega^{[n]}}_{=0 \text{ by  \eqref{eq-kappa} \& \eqref{eq-canonical-normalized-mu}}}
+
\frac{(1-\gamma)ab}{b-a}
\underbrace{\int_Xe^{\mu_\xi}\omega^{[n]}}_{=(e^{c_1^{\T}(X)})_{\T}(\xi)}.
\end{align*}
This proves \eqref{eq-solition-eval-ETS}. 
\end{proof}

\begin{proof}[Proof of Claim \ref{clm:quotient-kappa-H}]
Taking the trace of the formula for the moment map of the tangent bundle in \cite[Section~3.3]{HL24}, we obtain
\[
\sqrt{-1}\tr\langle\Phi_{h_\omega},\xi\rangle=\partial^*\partial\mu_\xi.
\]
Hence, using the induced metric on the determinant line bundle and the metric change formula \cite[(5)]{HL24}, we have
\[
\frac{\sqrt{-1}}{2\pi}\left\langle\Phi_{e^{\mu_\xi}\det h_\omega},\xi\right\rangle
=\frac1{2\pi}\left(\partial^*\partial\mu_\xi+\sqrt{-1}\partial\mu_\xi(\xi_X)\right)
\underalign{\raisebox{-1.0ex}{\tiny \text{\cite[Lem. 2.6]{HL24}}}}{=}
\frac1{2\pi}e^{-\mu_\xi}\Delta_v\mu_\xi
=\kappa.
\]
Thus the moment map identity in \cite[Proposition~2.2]{HL24} gives
\[
d\kappa=-\iota_{\xi_X}\Ric_v\omega.
\]
Moreover, the equivariant first Chern form of this metric and $\omega+\mu$ represent the same equivariant class $c_1^\T(X)$, so their difference is equivariantly exact.
It follows that
\(
\kappa-\mu_\xi=\theta(\xi_X)
\)
for some $\T$-invariant $1$-form $\theta$.

(1). If $\gamma=1$, then $\Ric_v\omega=\omega$.
Thus $d\kappa=d\mu_\xi$, and $\kappa-\mu_\xi$ is constant.
By \eqref{eq-kappa} and \eqref{eq-canonical-normalized-mu}, both functions have weighted mean zero.
Hence $\kappa=\mu_\xi$.

(2). First assume that $\Ric_v\omega>\gamma\omega$.
By the definition of $H_\gamma$, we have
\[
dH_\gamma=-\iota_{\xi_X}\left(\frac{\Ric_v\omega-\gamma\omega}{1-\gamma}\right).
\]
Let $x_0$ be a point where $H_\gamma$ attains its maximum.
Since $dH_\gamma(x_0)=0$ and $\Ric_v\omega>\gamma\omega$, we have $\xi_X(x_0)=0$.
Thus $\kappa(x_0)-\mu_\xi(x_0)=\theta_{x_0}(\xi_X(x_0))=0$, and therefore
\[
\max_XH_\gamma=H_\gamma(x_0)=\frac{\kappa(x_0)-\gamma\mu_\xi(x_0)}{1-\gamma}=\mu_\xi(x_0)\leq b.
\]
The same argument at a minimum point gives $a\leq\min_XH_\gamma$.
The case $\Ric_v\omega\geq\gamma\omega$ follows by taking a limit.
\end{proof}

\begin{rem}
In \cite[after equation (54)]{HL24}, 
the rank the orthogonal complement $\mathcal{(\cF')^{\perp}}$ is computed as 
$n-\rk{\cF} +1$. 
It is in fact $n-\rk{\cF}$. 
This is why we cannot omit the term 
involving 
$p=\tr(P_QP_0)$ in the proof of Theorem \ref{thm-slope-equiv-ineq}. 
\end{rem}

\subsection{Equivariant Miyaoka--Yau inequality for K\"ahler--Ricci solitons}

\begin{lem}
\label{lem-numerical-tail-bound}
Let $r_1,\ldots,r_s$ be positive integers and $\mu_1\geq\cdots\geq\mu_s$ be real numbers. Set
\begin{equation}
\label{eq-defn-rmu}
r\coloneqq\sum_{i=1}^s r_i,\qquad
\mu\coloneqq\frac1r\sum_{i=1}^s r_i\mu_i.
\end{equation}
Suppose that real numbers $\beta\leq\mu$ and $B\geq0$ satisfy
\begin{equation}
\label{eq-numerical-tail-assumption}
\sum_{i=k}^s r_i\mu_i\geq\beta\sum_{i=k}^s r_i-B
\end{equation}
for every $k=1,\ldots,s$. Then
\begin{equation}
\label{eq-numerical-tail-conclusion}
\sum_{i<j}r_ir_j(\mu_i-\mu_j)^2
\leq r^2(r-1)(\mu-\beta)^2+2r^2(\mu-\beta)B+2rB^2.
\end{equation}
\end{lem}

\begin{proof}
Since $\mu_1\geq\mu\geq\beta$, we may take $k_0$ to be the largest index such that $\mu_{k_0}\geq\beta$.
Applying \eqref{eq-numerical-tail-assumption} with $k=k_0+1$ when $k_0<s$, we obtain
\begin{equation}
\label{eq-inequality-slope-0}
0\leq\sum_{i=k_0+1}^{s}r_i(\beta-\mu_i)
\underalign{\text{\eqref{eq-numerical-tail-assumption}}}{\leq}B.
\end{equation}
If $k_0=s$, the sum is interpreted as zero, so the same inequality holds.
Moreover,
\begin{equation}
\label{eq-inequality-slope-1}
0\leq\sum_{i=1}^{k_0}r_i(\mu_i-\beta)
\underalign{\text{\eqref{eq-defn-rmu}}}{=}
r(\mu-\beta)+\sum_{i=k_0+1}^{s}r_i(\beta-\mu_i)
\underalign{\text{\eqref{eq-numerical-tail-assumption}}}{\leq}
r(\mu-\beta)+B.
\end{equation}
Thus
\begin{align*}
\sum_{1\leq i<j\leq s}r_ir_j(\mu_i-\mu_j)^2
&\underalign{\eqref{eq-defn-rmu}}{=}
r\sum_{i=1}^{s}r_i(\mu_i-\beta)^2-r^2(\mu-\beta)^2\\
&\underalign{\text{(by $r_i\geq1$)}}{\leq}
r\left(\sum_{i=1}^{k_0}r_i(\mu_i-\beta)\right)^2
+r\left(\sum_{i=k_0+1}^{s}r_i(\beta-\mu_i)\right)^2
-r^2(\mu-\beta)^2\\
&\underalign{\text{\eqref{eq-inequality-slope-0} \& \eqref{eq-inequality-slope-1}}}{\leq}
r\bigl((r(\mu-\beta)+B)^2+B^2\bigr)-r^2(\mu-\beta)^2.
\end{align*}
Expanding the right-hand side gives the desired inequality.
\end{proof}

We next give an equivariant Miyaoka--Yau inequality involving the greatest weighted Ricci lower bound.
In the following, we restrict to exponential weights $v(\mu)=e^{\mu_\xi}$. Let $\delta_v(X)$ denote the $v$-weighted delta invariant, whose definition can be found in \cite{HL23, DJ25}.
We call $\delta'_v(X)\coloneqq\min\{1,\delta_v(X)\}$ the \emph{greatest weighted Ricci lower bound} of $X$.
By \cite[Theorem 1.2]{DJ25}, which is a weighted version of the results in \cite{CRZ19, BBJ21}, we have
\begin{equation}
\label{eq-defn-delta'}
\delta'_v(X)
=
\sup\left\{\beta\in\R \,\middle\vert\, \Ric_v(\omega)>\beta\omega
\text{ for some $\T$-invariant K\"ahler form $\omega\in c_1(X)$}\right\}.
\end{equation}

With this notation, we obtain the following result.

\begin{thm}
\label{thm-equiv-MY-corrected-delta}
Let $\G=(\C^*)^l$ be an algebraic torus, let $\T=(S^1)^l$ be its maximal compact torus, and set $\mathfrak t\coloneqq\Lie(\T)$.
Let $X$ be a smooth Fano $\G$-variety of dimension $n\geq2$, and let $\omega\in c_1(X)$ be a $\T$-invariant K\"ahler form.
Choose the moment map $\mu\colon X\to\mathfrak t^\vee$ so that $\{\omega+\mu\}=c_1^\T(X)$.
For $\xi\in\mathfrak t$, set $v(\mu)\coloneqq e^{\mu_\xi}$, and assume that $v$ is canonically normalized in the sense of \eqref{eq-defn-canonically-normalized}.

Then the following equivariant Miyaoka--Yau inequality with a correction term holds:
\begin{equation}
\label{eq-equiv-MY-corrected-delta}
\begin{aligned}
&\frac{\left(\left(2(n+1)c_2^\T(X)-n c_1^\T(X)^2\right)\cdot e^{c_1^\T(X)}\right)_\T(\xi)}
{(e^{c_1^\T(X)})_\T(\xi)}\\
&\geq -n^3(1-\delta'_v(X))^2-2n(n+1)(1-\delta'_v(X))
\left[\delta'_v(X)\ell_\xi-(1-\delta'_v(X))\frac{ab}{b-a}\right]\\
&\quad-2(n+1)
\left[\delta'_v(X)\ell_\xi-(1-\delta'_v(X))\frac{ab}{b-a}\right]^2.
\end{aligned}
\end{equation}
Here $\delta'_v(X)=\min\{1,\delta_v(X)\}$, and we use the notation $a,b,\ell_\xi$ from Theorem \ref{thm-slope-equiv-ineq}.
\end{thm}

The case $\delta'_v(X)=1$ gives the following corollary.

\begin{cor}
\label{cor-equiv-MY-corrected-delta-one}
In the setting of Theorem \ref{thm-equiv-MY-corrected-delta}, assume that $\delta_v(X)\geq1$.
Then the following equivariant Miyaoka--Yau inequality with a correction term holds:
\begin{equation}
\label{eq-equiv-MY-corrected-delta-one}
\begin{aligned}
&\left(\left(2(n+1)c_2^\T(X)-n c_1^\T(X)^2\right)\cdot e^{c_1^\T(X)}\right)_\T(\xi)\\
&\ge
-\frac{n+1}{2(e^{c_1^\T(X)})_\T(\xi)}
\left(\int_X\left|
\frac{(e^{c_1^\T(X)})_\T(\xi)}{\int_X\omega^{[n]}}e^{-\mu_\xi}-1-\mu_\xi
\right|e^{\mu_\xi}\omega^{[n]}\right)^2.
\end{aligned}
\end{equation}
In particular, this holds if $X$ admits a K\"ahler--Ricci soliton with soliton vector $\xi$.
\end{cor}

In Example \ref{ex-negative-equiv-MY-sixfold} below, we show that even when $X$ admits a K\"ahler--Ricci soliton, its soliton vector $\xi$ may satisfy
\[
\left(\left(2(n+1)c_2^\T(X)-n c_1^\T(X)^2\right)\cdot e^{c_1^\T(X)}\right)_\T(\xi)<0.
\]
Thus the correction term in the above inequality cannot be omitted in general.

\begin{proof}[Proof of Theorem \ref{thm-equiv-MY-corrected-delta}]
We slightly modify the proof of Langer's inequality in Theorem \ref{thm-langer-ineq-equiv}.
Since $X$ satisfies Setup \ref{setup-soliton-situation} and the canonical extension sheaf $\mathscr V$ is $\G$-equivariant and locally free, the argument of Theorem \ref{thm-langer-ineq-equiv} applies to $v(\mu)=e^{\mu_\xi}$.

Fix $0<\gamma<\delta'_v(X)$.
By \eqref{eq-defn-delta'}, there exists a $\T$-invariant K\"ahler form $\omega_\gamma\in c_1(X)$ such that $\Ric_v\omega_\gamma>\gamma\omega_\gamma$.
The quantities $a,b,\ell_\xi$ are independent of the choice of metric.
Let $\mathscr V$ be the canonical extension sheaf in Theorem \ref{thm-slope-equiv-ineq}, and take its $v$-Harder--Narasimhan filtration
\[
0=\mathscr V_0\subsetneq\mathscr V_1\subsetneq\cdots\subsetneq\mathscr V_s=\mathscr V.
\]
Set $r_i\coloneqq\rk(\mathscr V_i/\mathscr V_{i-1})$ and $\mu_i\coloneqq\mu_v(\mathscr V_i/\mathscr V_{i-1})$.
Then
\[
\sum_{i=1}^s r_i=n+1,\qquad
\mu_v(\mathscr V)=\frac{1}{n+1}\sum_{i=1}^s r_i\mu_i
\underalign{\text{\eqref{eq-defn-canonically-normalized}}}{=}
\frac{n}{n+1}(e^{c_1^\T(X)})_\T(\xi).
\]
Set
\[
\beta\coloneqq\frac{n\gamma}{n+1}(e^{c_1^\T(X)})_\T(\xi),
\qquad
B\coloneqq(e^{c_1^\T(X)})_\T(\xi)
\left[\gamma\ell_\xi-(1-\gamma)\frac{ab}{b-a}\right].
\]
Applying Theorem \ref{thm-slope-equiv-ineq} to each torsion-free quotient $\mathscr V/\mathscr V_{k-1}$ gives
\[
\sum_{i=k}^s r_i\mu_i
\underalign{\text{(Thm. \ref{thm-slope-equiv-ineq})}}{\geq}
\beta\sum_{i=k}^s r_i-B
\qquad(k=1,\ldots,s).
\]
Since $\beta\leq\mu_v(\mathscr V)$ and $B\geq0$, we may apply Lemma \ref{lem-numerical-tail-bound}.
Hence
\begin{align*}
&\frac{\left(\left(2(n+1)c_2^\T(X)-n c_1^\T(X)^2\right)\cdot e^{c_1^\T(X)}\right)_\T(\xi)}
{(e^{c_1^\T(X)})_\T(\xi)}\\
&\underalign{\text{as in \eqref{eq-Langer-ineq}}}{\geq}
-\frac{1}{\bigl((e^{c_1^\T(X)})_\T(\xi)\bigr)^2}
\sum_{1\leq i<j\leq s}r_ir_j(\mu_i-\mu_j)^2\\
&\underalign{\text{(Lem. \ref{lem-numerical-tail-bound})}}{\geq}
-n^3(1-\gamma)^2
-2n(n+1)(1-\gamma)\frac{B}{(e^{c_1^\T(X)})_\T(\xi)}
-2(n+1)\left(\frac{B}{(e^{c_1^\T(X)})_\T(\xi)}\right)^2.
\end{align*}
In the last inequality, we substituted $r=n+1$ and $\mu_v(\mathscr V)-\beta=\frac{n(1-\gamma)}{n+1}(e^{c_1^\T(X)})_\T(\xi)$.
Substituting the definition of $B$ and letting $\gamma\uparrow\delta'_v(X)$, we obtain \eqref{eq-equiv-MY-corrected-delta}.
\end{proof}

\begin{rem}
If we apply the above arguments only to the tangent sheaf, 
it yields 
\begin{equation*}
\left((2nc^{\T}_2(T_X) - (n-1)c^{\T}_1(T_X)^2 ) \cdot e^{c_1^\T(X)} \right)_\T(\xi) \geq -n^2(n-1)(1-\delta_v'(X))^2(e^{c_1^\T(X)})_\T(\xi). 
\end{equation*}
This can be rewritten as  
\begin{equation}\label{eq-tangent-MY}
  \begin{aligned}
&\frac{\left(\left(2(n+1)c_2^\T(X)-n c_1^\T(X)^2\right)\cdot e^{c_1^\T(X)}\right)_\T(\xi)}
{(e^{c_1^\T(X)})_\T(\xi)} 
\geq  
%-\frac{\left(\left(c_2^\T(X)\right)\cdot e^{c_1^\T(X)}\right)_\T(\xi)}{n(e^{c_1^\T(X)})_\T(\xi)} 
-\frac{\left(c_1^\T(X)^2\cdot e^{c_1^\T(X)}\right)_\T(\xi)}
{n(e^{c_1^\T(X)})_\T(\xi)}
-n(n^2-1)(1-\delta_v'(X))^2. 
  \end{aligned}
\end{equation}
Unless $v=1$, we cannot say which inequality, \eqref{eq-equiv-MY-corrected-delta} or \eqref{eq-tangent-MY}, is stronger. 
\end{rem}

\subsection{Outlook}

We consider the equality case of the equivariant Miyaoka--Yau inequality for a general weight $v$.
Assuming that the canonical extension sheaf is $v$-semistable, Assumptions \ref{assump-equiv-v-hodge} and \ref{assump-equiv-v-BG} give the following characterization of projective space.

\begin{prop}
\label{prop-equiv-MY-equality}
Assume Setup \ref{setup-soliton-situation}, with $X$ a Fano manifold of dimension $n\geq2$ and $\alpha_{\T}=c_1^{\T}(X)$.
Let $v\colon P\to\R_{>0}$ be a weight for which Assumptions \ref{assump-equiv-v-hodge} and \ref{assump-equiv-v-BG} hold.
If the canonical extension sheaf $\mathscr V$ in Definition \ref{defn-canonical-extension-equiv} is $v$-semistable and
\begin{equation}
\label{eq-equiv-MY-equality}
\left(
\left(2(n+1)c_2^{\T}(X)-n c_1^{\T}(X)^2\right)
\cdot v(c_1^{\T}(X))
\right)=0,
\end{equation}
then $X\cong \C\mathbb P^n$.
\end{prop}

\begin{proof}
Equation \eqref{eq-equiv-MY-equality} means that equality holds in the equivariant Bogomolov--Gieseker inequality for $\mathscr V$.
Since $\mathscr V$ is $v$-semistable, Proposition \ref{prop-num-proj-flat} (1) implies that it is projectively numerically flat.
Thus $\mathscr V$ is projectively flat by \cite[Theorem~1.7]{LOY24}.
Since $X$ is a smooth Fano manifold, it is simply connected.
It follows that $\mathbb P(\mathscr V)$ is trivial, and hence
\(
\mathscr V\cong M^{\oplus(n+1)}
\)
for some line bundle $M$.
Taking determinants, we obtain
\[
\mathcal O_X(-K_X)
\cong\det\mathscr V
\cong M^{\otimes(n+1)}.
\]
Since $-K_X$ is ample, so is $M$.
The Kobayashi--Ochiai theorem \cite{KO73} therefore gives $X\cong\C\mathbb P^n$.
\end{proof}

\section{Counterexamples}
\label{sec-examples}

In this section, we give counterexamples to weighted semistability of the canonical extension and the equivariant Miyaoka-Yau inequality, even for K\"ahler--Ricci solitons (Examples \ref{ex-unstable-canonical-extension} and \ref{ex-negative-equiv-MY-sixfold}).
These counterexamples were suggested by AI and prompted the authors to revisit and revise their earlier paper. Although these computations are difficult to carry out by hand, they are straightforward using AI-based tools.

\subsection{Failure of the semistability of the canonical extension sheaf}
\begin{ex}
\label{ex-unstable-canonical-extension}
Let $N\coloneqq\Z e_1\oplus\Z e_2\oplus\Z e_3$ and $M\coloneqq N^\vee\simeq\Z^3$, and set $\T\coloneqq(S^1)^3$ and $\G\coloneqq(\C^*)^3$.
Consider
\(
X\coloneqq\mathbb P_{\C\mathbb P^2}\bigl(\mathcal O\oplus\mathcal O(2)\bigr).
\)
Since $X$ is a smooth toric Fano threefold, it admits a $\T$-invariant K\"ahler--Ricci soliton by \cite[Theorem~1.1]{WZ04}.
Let $\omega\in c_1(X)$ be the soliton metric, let $\xi_{\mathrm{KRS}}$ be its soliton vector, and set $v(\mu)\coloneqq e^{\mu_{\xi_{\mathrm{KRS}}}}$.
We show that the canonical extension sheaf is not semistable with respect to this exponential weight.

The primitive ray generators of the fan of $X$ can be chosen as
\[
v_0=(-1,-1,2),\quad v_1=(1,0,0),\quad v_2=(0,1,0),\quad v_3=(0,0,1),\quad v_4=(0,0,-1).
\]
The maximal cones are
\[
\sigma_{ij\varepsilon}\coloneqq\Cone(v_i,v_j,v_\varepsilon),
\qquad (i,j)\in\{(0,1),(1,2),(2,0)\},\quad \varepsilon\in\{3,4\}.
\]
Let $D_j$ be the torus-invariant prime divisor corresponding to the ray $\R_{\geq0}v_j$, and equip each $\mathcal O_X(D_j)$ with the linearization for which its canonical section is invariant.

Since $-K_X=\sum_{j=0}^{4}D_j$, the anticanonical polytope in the dual lattice coordinates $(u_1,u_2,y)$ is
\begin{align*}
P\coloneqq P_{-K_X}
&=\{m\in M_\R\mid\langle m,v_j\rangle\geq-1,\ j=0,\ldots,4\}\\
&=\{(u_1,u_2,y)\mid-1\leq y\leq1,\ u_1,u_2\geq-1,\ u_1+u_2\leq1+2y\}.
\end{align*}
Choose the moment map $\mu$ so that $\{\omega+\mu\}=c_1^\T(X)$. Notice that $\mu(X)=P$.
By \cite[Lemma~2.2]{WZ04}, the soliton vector $\xi_{\mathrm{KRS}}=(c_1,c_2,c_3)$ is the unique vector satisfying
\begin{equation}
\label{eq-unstable-weighted-barycenter}
\int_P u_i e^{c_1u_1+c_2u_2+c_3y}\,du_1\,du_2\,dy=0\quad(i=1,2),
\qquad
\int_P y e^{c_1u_1+c_2u_2+c_3y}\,du_1\,du_2\,dy=0.
\end{equation}
A direct computation gives $\xi_{\mathrm{KRS}}=(0,0,\tau)$, where $\tau$ is the unique real number satisfying $A'(\tau)=0$ for
\begin{equation}
\label{eq-defn-At-solition}
A(t)\coloneqq
\int_Pe^{ty}\,du_1\,du_2\,dy
=\frac12\int_{-1}^{1}(3+2y)^2e^{ty}\,dy.
\end{equation}
Using the Taylor expansion of the exponential function up to degree $30$ and estimating the remainder, we obtain
\(
A'(-1.47061)<0<A'(-1.47060).
\)
Since $A''>0$ and $A'(\tau)=0$, it follows that $-1.47061<\tau<-1.47060$.

Next, we explicitly construct the canonical extension sheaf in Definition \ref{defn-canonical-extension-equiv}.
The ordinary Picard group has basis $[D_0],[D_4]$, and
\[
[D_1]=[D_2]=[D_0],\qquad [D_3]=[D_4]-2[D_0],\qquad [-K_X]=[D_0]+2[D_4].
\]
Dualizing the Euler sequence in \cite[Theorem~8.1.6]{CLS10}, we obtain
\begin{equation}
\label{eq-unstable-toric-euler}
0\longrightarrow\mathcal O_X^{\oplus2}
\xrightarrow{\iota}\bigoplus_{j=0}^{4}\mathcal O_X(D_j)
\longrightarrow T_X\longrightarrow0,
\quad
\iota(a,b)\coloneqq(a,a,a,-2a+b,b).
\end{equation}
The extension class of this sequence is
\[
\left(c_1(\mathcal O_X(D_0)),\,c_1(\mathcal O_X(D_4))\right)
\in\operatorname{Ext}^1(T_X,\mathcal O_X^{\oplus2})
\simeq H^1(X,\Omega_X^1)^{\oplus2}.
\]
Taking the pushout along $\ell(a,b)\coloneqq a+2b$, we obtain the following commutative diagram:
\[
\xymatrix@C=20pt@R=20pt{
0 \ar[r] & \mathcal O_X^{\oplus2} \ar[r]^-{\iota} \ar[d]_{\ell}
& \bigoplus_{j=0}^{4}\mathcal O_X(D_j) \ar[r] \ar[d]^{\pi}
& T_X \ar[r] \ar@{=}[d] & 0\\
0 \ar[r] & \mathcal O_X \ar[r]
& \mathscr V \ar[r] & T_X \ar[r] & 0.
}
\]
The extension class of the bottom row is
\[
c_1(\mathcal O_X(D_0))+2c_1(\mathcal O_X(D_4))
=c_1(-K_X)=c_1(X).
\]
Thus $\mathscr V$ is the canonical extension sheaf, and
\[
\mathscr V\cong
\operatorname{coker}\left(
\mathcal O_X\xrightarrow{\rho}
\bigoplus_{j=0}^{4}\mathcal O_X(D_j)
\right),
\quad
\rho(c)\coloneqq(2c,2c,2c,-5c,-c).
\]
Since $D_3\cap D_4=\varnothing$, the cokernel description shows that $\mathcal O_X(D_1)\to\mathscr V$ is a $\T$-equivariant subbundle.
The composite $\mathcal O_X(D_1)\to\mathscr V\to T_X$ identifies $\mathcal O_X(D_1)$ with a saturated subsheaf of $T_X$.

We now compute the slope of $\mathcal O_X(D_1)$.
Since $c_1^\G(\mathcal O_X(D_1))\cap[X]_\G=[D_1]_\G$ and $\rk\mathcal O_X(D_1)=1$, we have
\begin{equation}
\label{eq-slope-D1}
\mu_v(\mathcal O_X(D_1))
=(c_1^\T(\mathcal O_X(D_1))\cdot e^{c_1^\T(X)})_\T(\xi_{\mathrm{KRS}})
\underalign{\text{(Prop. \ref{prop-semipositive-slope})}}{=}
(e^{c_1^\T(X)|_{D_1}})_\T(\xi_{\mathrm{KRS}})
\underalign{\text{(Lem. \ref{lem-intersection-chern})}}{=}
\int_{D_1}e^{\mu_{\xi_{\mathrm{KRS}}}}(\omega|_{D_1})^{[2]}.
\end{equation}
The moment map image of $D_1$ is
\[
\mu(D_1)=P\cap\{u_1=-1\}
=\{(-1,u_2,y)\mid-1\leq y\leq1,\ -1\leq u_2\leq2+2y\}.
\]
The pushforward of $(\omega|_{D_1})^{[2]}$ under the moment map is the lattice area measure $du_2\,dy$ on this facet
(cf.~\cite[Theorems~13.3.7 and~13.5.2]{CLS10}).
Thus, using $\mu_{\xi_{\mathrm{KRS}}}=\tau y$, we obtain
\[
\mu_v(\mathcal O_X(D_1))
\underalign{\eqref{eq-slope-D1}}{=}
\int_{-1}^{1}\int_{-1}^{2+2y}e^{\tau y}\,du_2\,dy
=\int_{-1}^{1}(3+2y)e^{\tau y}\,dy.
\]

On the other hand, the canonical normalization and $\rk\mathscr V=4$ give
\[
\mu_v(\mathscr V)
=\frac14(c_1^\T(X)\cdot e^{c_1^\T(X)})_\T(\xi_{\mathrm{KRS}})
=\frac34\int_Xe^{\mu_{\xi_{\mathrm{KRS}}}}\omega^{[3]}
\underalign{\eqref{eq-defn-At-solition}}{=}\frac34A(\tau)
=\frac38\int_{-1}^{1}(3+2y)^2e^{\tau y}\,dy.
\]
Using $-1.47061<\tau<-1.47060$ in these integral formulas and estimating the errors in the same way, we obtain
\[
\mu_v(\mathcal O_X(D_1))>5.9899>5.7547>\mu_v(\mathscr V).
\]
Therefore, $\mathscr V$ is not $v$-semistable.
The same subsheaf $\mathcal O_X(D_1)\subset T_X$, together with its lift to $\mathscr V$, also gives a counterexample to \cite[Theorem~7.2]{HL24}.
\end{ex}

\subsection{Failure of the Miyaoka--Yau inequality}

\begin{ex}
\label{ex-negative-equiv-MY-sixfold}
Let $X\coloneqq\mathbb P_{\C\mathbb P^5}(\mathcal O\oplus\mathcal O(5))$, and let $\T$ be the maximal compact torus of its acting algebraic torus.
Since $X$ is a smooth toric Fano sixfold, it admits a $\T$-invariant K\"ahler--Ricci soliton by \cite[Theorem~1.1]{WZ04}.
Let $\omega\in c_1(X)$ be the soliton metric, let $\xi_{\mathrm{KRS}}$ be its soliton vector, and set $v(\mu)\coloneqq e^{\mu_{\xi_{\mathrm{KRS}}}}$.

Denote the boundary divisors by $D_0,\ldots,D_7$, and equip each $\mathcal O_X(D_j)$ with the linearization for which its canonical section is invariant.
As in Example \ref{ex-unstable-canonical-extension}, choose the moment map $\mu$ so that $\{\omega+\mu\}=c_1^\T(X)$.
In suitable lattice coordinates $(u_1,\ldots,u_5,y)$, the anticanonical polytope is
\[
P\coloneqq P_{-K_X}
=\left\{(u_1,\ldots,u_5,y)\in\R^6\ \middle|\ -1\leq y\leq1,\ u_i\geq-1,\ \sum_{i=1}^5u_i\leq1+5y\right\}.
\]
Using the weighted barycenter condition as in that example, we obtain $\xi_{\mathrm{KRS}}=(0,0,0,0,0,\tau)$, where $\tau$ is the unique real number satisfying $A'(\tau)=0$ for
\begin{equation}
\label{eq-sixfold-defn-At-solition}
A(t)\coloneqq\int_Pe^{ty}\,du_1\cdots du_5\,dy
=\frac1{120}\int_{-1}^{1}(6+5y)^5e^{ty}\,dy.
\end{equation}
Using the Taylor expansion of the exponential function up to degree $60$ and estimating the remainder, we obtain $A'(-4.71627)<0<A'(-4.71625)$.
Since $A''>0$, it follows that $-4.71627<\tau<-4.71625$.

Next, we compute the equivariant intersection numbers.
The Euler sequence in \cite[Theorem~8.1.6]{CLS10} gives
\begin{equation}
\label{eq-sixfold-c2-integral-p1}
c_2^\T(X)=\sum_{0\leq i<j\leq7}c_1^\T(\mathcal O_X(D_i))c_1^\T(\mathcal O_X(D_j)).
\end{equation}
If $D_i\cap D_j$ is nonempty, it is a smooth submanifold of codimension $2$.
Thus the equivariant projection formula and Lemma \ref{lem-intersection-chern} give
\begin{equation}
\label{eq-sixfold-c2-integral-p2}
\left(c_1^\T(\mathcal O_X(D_i))c_1^\T(\mathcal O_X(D_j))\cdot e^{c_1^\T(X)}\right)_\T(\xi_{\mathrm{KRS}})
=\int_{D_i\cap D_j}e^{\mu_{\xi_{\mathrm{KRS}}}}(\omega|_{D_i\cap D_j})^{[4]}.
\end{equation}
It follows that
\begin{equation}
\label{eq-sixfold-c2-integral}
(c_2^\T(X)\cdot e^{c_1^\T(X)})_\T(\xi_{\mathrm{KRS}})
\underalign{\text{\eqref{eq-sixfold-c2-integral-p1} \& \eqref{eq-sixfold-c2-integral-p2}}}{=}
\frac52\int_{-1}^{1}(6+5y)^3e^{\tau y}\,dy
+\frac{e^{-\tau}+14641e^\tau}{4}.
\end{equation}
On the other hand,
\begin{equation}
\label{eq-sixfold-c1-square-integral}
(c_1^\T(X)^2\cdot e^{c_1^\T(X)})_\T(\xi_{\mathrm{KRS}})
=\frac1{120}\int_{-1}^{1}(30+12\tau y+\tau^2y^2)(6+5y)^5e^{\tau y}\,dy
\underalign{\text{(by $A'(\tau)=0$)}}{=}30A(\tau)+\tau^2A''(\tau).
\end{equation}
Using the above bounds on $\tau$ in \eqref{eq-sixfold-c2-integral} and \eqref{eq-sixfold-c1-square-integral}, and estimating the errors in the same way, we obtain
\[
\left(\left(14c_2^\T(X)-6c_1^\T(X)^2\right)\cdot e^{c_1^\T(X)}\right)_\T(\xi_{\mathrm{KRS}})<-14<0.
\]
Its numerical value is approximately $-14.03728176$.
\end{ex}

\bibliographystyle{alpha}
\bibliography{ref_MY_2_soliton.bib}
\end{document}